\documentclass[11pt,letterpaper]{amsart}
\usepackage{amssymb,mathrsfs,graphicx,enumerate,color}
\usepackage{amsthm,amsfonts,amssymb,epsfig,graphics,amsmath,amsbsy,enumerate}
\usepackage{colortbl}
\definecolor{black}{rgb}{0.0, 0.0, 0.0}
\definecolor{red}{rgb}{1.0, 0.5, 0.5}

\title[Stability of generic composite waves for Navier-Stokes-Fourier equations]
{Time-asymptotic stability of generic Riemann solutions for compressible Navier-Stokes-Fourier equations}

\author[Kang]{Moon-Jin Kang}
\address[Moon-Jin Kang]{\newline Department of Mathematical Sciences,
\newline
Korea Advanced Institute of Science and Technology, Daejeon 34141, Korea}
\email{moonjinkang@kaist.ac.kr}

\author[Vasseur]{Alexis F. Vasseur}
\address[Alexis F. Vasseur]{\newline Department of Mathematics, \newline The University of Texas at Austin, Austin, TX 78712, USA}
\email{vasseur@math.utexas.edu}

\author[Wang]{Yi Wang}
\address[Yi Wang]{\newline Institute of Applied Mathematics, AMSS, CAS, Beijing 100190, P. R. China
\newline
and School of Mathematical Sciences, University of Chinese Academy of Sciences,
\newline Beijing 100049, P. R. China}
\email{wangyi@amss.ac.cn}

\newtheorem{theorem}{Theorem}[section]
\newtheorem{lemma}{Lemma}[section]

\newtheorem{proposition}{Proposition}[section]
\newtheorem{remark}{Remark}[section]

\newcommand{\bbr}{\mathbb {R}}
\newcommand{\R}{\mathbb {R}}

\newcommand{\mb}{\mathbf}

\newcommand{\deltas}{\delta_S}
\newcommand{\deltar}{\delta_R}
\newcommand{\dc}{\delta_C}

\numberwithin{figure}{section}

\newcommand{\beq}{\begin{equation}}
\newcommand{\eeq}{\end{equation}}
\newcommand{\bsp}{\begin{split}}
\newcommand{\esp}{\end{split}}

\newcommand{\di}{\displaystyle}

\newcommand{\s}{\sigma}
\newcommand{\x}{\xi}

\newcommand{\ds}{\delta_S}

\def\eps{\varepsilon }

\newcommand\adots{\mathinner{\mkern2mu\raise1pt\hbox{.}
\mkern3mu\raise4pt\hbox{.}\mkern1mu\raise7pt\hbox{.}}}

\def\charf {\mbox{{\text 1}\kern-.30em {\text l}}}

\def\lam{\lambda}  

\newcommand \vc{v^{C}}
\DeclareMathSizes{11}{11}{3}{2}

 \newcommand{\bmat}{\begin{pmatrix}} 
  \newcommand{\emat}{\end{pmatrix}}
  
\def\l{\lambda}

\begin{document}
\bibliographystyle{plain}

\date{\today}

\subjclass[2010]{}

\thanks{\textbf{Acknowledgment.}   M.-J. Kang was partially supported by the National Research Foundation of Korea  (NRF-2019R1C1C1009355  and NRF-2019R1A5A1028324),  and the POSCO Science Fellowship of POSCO TJ Park Foundation.
A. Vasseur was partially supported by the NSF grants DMS 2219434 and 2306852. Y. Wang is supported by NSFC grants No. 12090014 and 11688101.
}

\begin{abstract}
 We establish the time-asymptotic stability of solutions to the one-dimensional compressible Navier-Stokes-Fourier equations, with initial data perturbed from Riemann data that forms a generic Riemann solution. The Riemann solution under consideration is composed of a viscous shock, a viscous contact wave, and a rarefaction wave. We prove that the perturbed solution of Navier-Stokes-Fourier converges, uniformly in space as time goes to infinity, to a viscous ansatz composed of viscous shock with time-dependent shift, a viscous contact wave and an inviscid rarefaction wave.  
 This is a first resolution of the challenging open problem associated with the generic Riemann solution.   
Our approach relies on the method of a-contraction with shifts, specifically applied to both the shock wave and the contact discontinuity wave. It enables the application of a global energy method for the generic combination of three waves.
\end{abstract}

\maketitle \centerline{\date}

\tableofcontents

\section{Introduction}
\setcounter{equation}{0}

Consider the  one-dimensional full compressible Navier-Stokes-Fourier equations. 
In the Lagrangian  mass coordinates, the system is described as 
\begin{align}\label{NS}
\begin{aligned}
\left\{ \begin{array}{ll}
        v_t- u_x =0,\qquad \quad\quad x\in\bbr,\ t\geq 0,\\[1mm]
       u_t  +p(v,\theta)_x=  (\mu\frac{u_x}{v})_x, \\[1mm]
       E_t+(p(v,\theta) u)_x=(\kappa\frac{\theta_x}{v})_x+(\mu\frac{uu_x}{v})_x
        \end{array} \right.
\end{aligned}
\end{align}
where the unknown functions $v=v(t,x)>0$, $u(t,x)$, and $\theta(t,x)$  represent respectively the specific volume, the velocity and the temperature of the gas, $E=e+\frac{u^2}2$ is the total energy function. For the ideal polytropic gas, the pressure function $p$ and the internal energy function $e$ is  given by 
\beq\label{state}
p(v,\theta)=\frac{R\theta}v,\qquad e(v, \theta)=\frac{R}{\gamma-1}\theta+const.
\eeq
where $R>0, \gamma>1$ are both constants depending on the fluid, and  the constants  $\mu>0$ and $\kappa>0$ correspond to the viscosity coefficient and the heat-conductivity coefficient. \\
Consider initial data $(v_0(x), u_0(x),\theta_0(x))$ connecting the two different end states $(v_\pm, u_\pm, \theta_\pm)\in \R_+\times \R\times \R_+$ (with $\bbr_+:=(0,+\infty)$) at far fields:
\begin{equation}\label{in}
(v_0(x), u_0(x),\theta_0(x)) \rightarrow(v_\pm, u_\pm,\theta_\pm), \quad {\rm as}\quad x\rightarrow\pm\infty.
\end{equation}

\vskip0.3cm
We aims to prove that 
the large-time behavior of  solutions to \eqref{NS}, with initial values verifying \eqref{in}, is determined by the Riemann problem of the associated full Euler equations: 
\begin{align}
\begin{aligned}\label{E}
\left\{ \begin{array}{ll}
        v_t- u_x =0,\\[1mm]
       u_t  +p_x= 0, \\[1mm]
       (e+\frac{u^2}2)_t+(pu)_x=0,
        \end{array} \right.
\end{aligned}
\end{align}
with the Riemann initial data
\begin{align}
\begin{aligned}\label{Ei}
(v,u,\theta)(t=0,x)=\left\{ \begin{array}{ll}
        (v_-,u_-,\theta_-),\quad x<0,\\[1mm]
       (v_+,u_+,\theta_+),\quad x>0,
        \end{array} \right.
\end{aligned}
\end{align}
corresponding to the end states \eqref{in}.
To simplify the exposition, we focus on the most challenging situation, called generic, where the related Riemann problem involves a shock, a contact discontinuity, and a rarefaction. Notably, the treatment of combinations of different waves, including both a rarefaction and a shock, was, before this paper, an outstanding open problem (see, for instance, the review paper \cite{MatsumuraBook} and the very recent book \cite{LiuBook}). 

\vskip0.3cm
In 2005  \cite{BB}, Bianchini-Bressan showed, for small BV initial values, the convergence at the inviscid limit of solutions to  a parabolic system with ``artificial'' viscosity to the unique solution of the associated hyperbolic system.
However, to this day, the result in full generality is still unknown for the physical Navier-Stokes system \eqref{NS}, even in the barotropic case. The study of large-time behavior of solutions to compressible Navier-Stokes equations \eqref{NS} towards the Riemann solutions may shed some insights about the physical inviscid limit to the Euler equations.

\vskip0.3cm
In the case of the scalar equation (replacing the system \eqref{NS} with a single viscous equation), the time-asymptotic stability of viscous waves and their connection to the inviscid problem were initially proven by Ilin-Oleinik in 1960  \cite{O}, with further contributions from Sattinger \cite{S}.

However, when dealing with systems like \eqref{NS}, the analysis becomes significantly more complex. Matsumura-Nishihara  demonstrated in \cite{MN-S} the stability of viscous shock waves for the barotropic Navier-Stokes equations, and a similar result was independently obtained by Goodman for a general system with artificial diffusion in \cite{G}, where diffusion is added to all equations in the system. Both papers relied on the crucial assumption of zero mass, in order to use the anti-derivative method. 
Subsequently, Liu \cite{Liu} , Szepessy-Xin \cite{SX}, and Liu-Zeng \cite{LZ} eliminated the need for the zero mass condition by introducing a constant shift for the viscous shock and diffusion waves, as well as for the coupled diffusion waves in the transverse characteristic fields. Masica-Zumbrun \cite{MZ} established the spectral stability of viscous shocks for the 1D compressible Navier-Stokes system, relaxing the zero mass condition to a slightly weaker spectral condition. Huang-Matsumura addressed in \cite{HM} the case of superimposed shocks for the Navier-Stokes-Fourier system. 
Han-Kang-Kim \cite{HKK} improved the result \cite{HM} by handling two shocks with small amplitudes ``independently'' chosen, for the barotropic case.   
Matsumura-Wang investigated the asymptotic stability of viscous shocks for Navier-Stokes systems with degenerate viscosities in \cite{matwan}, which was further generalized in  \cite{VY} to a wider range of viscosities using the BD entropy introduced by Bresch-Desjardins in \cite{BD-06}.
\vskip0.3cm
Regarding the stability of rarefaction waves, a different set of techniques based on direct energy methods was employed. Matsumura-Nishihara \cite{MN86, MN} initially established the time-asymptotic stability of rarefaction waves for the compressible and isentropic Navier-Stokes equations. It was later  generalized to the Navier-Stokes-Fourier system \eqref{NS} by Liu-Xin \cite{LX} and Nishihara-Yang-Zhao \cite{NYZ}. The time-asymptotic stability and vanishing viscosity limit of isentropic Navier-Stokes equatios/Navier-Stokes-Fourier equations to the planar rarefaction wave of 2D/3D compressible Euler equations could be found in \cite{LW, LWW-1, LWW1, LWW2, BWX}.  For the linear degenerate wave in the second characteristic field, it is known that the inviscid contact discontinuity is unstable for the Euler equations time-asymptotically. While for the Euler equations with viscosities, a viscous contact wave, which is a viscous version of inviscid contact discontinuity, can be constructed and proved to be time-asymptotically stable to both 1D ``artificial'' viscosity system \cite{Xin-c, Liu-Xin97} and the 1D physical Navier-Stokes system \cite{Huang-Xin-Yang, HLM}. This viscous contact wave can converge to the inviscid contact discontinuity in $L^p-$norm $(\forall p\geq1)$ at any finite time interval, however, they could be far away at large time, which is called the meta-stability of the inviscid contact discontinuity to the viscous system. For the composite wave, Huang-Li-Matsumura \cite{HLM} proved the time-asymptotic stability of the combination of two rarefaction waves and a viscous contact wave to the full compressible Navier-Stokes-Fourier system \eqref{NS} by introducing a new heat-kernel type estimate.

\vskip0.3cm
\noindent
However, the time-asymptotic stability of the superposition of a viscous shock wave, a viscous contact wave, and a rarefaction wave has remained an open problem until now, \cite{MatsumuraBook, LiuBook}.
The main challenge in addressing the general case stems from the fact that the classical anti-derivative method, commonly employed for the stability analysis of shocks, does not align well with the energy method traditionally used for the stability analysis of rarefaction and viscous contact waves. Resolving this issue has been the primary focus of our research, and our main theorem successfully tackles this problem. The main ingredient is the theory of a-contraction with shifts first introduced in the viscous setting for the scalar case in \cite{Kang-V-1}, and for the barotropic Navier-Stokes equation in \cite{Kang-V-NS17}. The method was employed for the first time in \cite{KVW-2022} to tackle the time asymptotic stability of composite waves with a shock and a rarefaction in the context of the barotropic Navier-Stokes equations by using the Bresch-Desjardins entropy to transfer the viscosity from the momentum equation to the continuity equation. However, compared with the barotropic Navier-Stokes equations, there is no Bresch-Desjardins entropy for the full compressible Navier-Stokes-Fourier system \eqref{NS} due to the energy dissipation equation. Preferably, we carry out the stability analysis in the original Navier-Stokes-Fourier equations \eqref{NS} by fully using the $a-$weighted functions $a(\xi)$ in \eqref{weight}. On the other hand, there are two dissipation terms in the momentum and energy equations in  \eqref{NS} and the sharp weighted Poincare inequality will be used respectively in the momentum and energy equations and so we need more accurate information about the velocity and the temperature in the viscous ansatz, in particular, for the viscous shock profile.

\vskip0.3cm
\noindent
 {\bf Riemann problem for the inviscid model:} Let us first describe the well-known solution  of the Riemann problem for the inviscid model \eqref{E}-\eqref{Ei}, first proposed and solved by Riemann \cite{Riemann} in the 1860s, then generalized by Lax \cite{Lax}. 
This system of conservation laws is   strictly hyperbolic. This means that the derivative of the flux function $(-u,p, p u)$ with respect to the conserved variables, about a fixed  state $(v,u,E)\in \R_+\times \R\times \R_+$:
$$
A(U)=\left(\begin{array}{l}0\ \ \ \ \ \ -1\ \ \ \ \ \ \ \ \  \  0\\[1mm]
-\frac{p}{v}\ \ -\frac{(\gamma-1)u}{v}\ \ \ \
 \frac{\gamma-1}{v}\\[1mm]
-\frac{pu}{v}\ \ p-\frac{(\gamma-1)u}{v}\ \  \frac{(\gamma-1)u}{v}\end{array}\right),
$$
is diagonalizable with real distinct eigenvalues. This matrix defines the waves generated by the linearization of the system \eqref{E} about this fixed state $(v,u, \theta)\in \R_+\times \R\times\bbr_+$.
Its  eigenvalues $\lambda_1=\lambda_1(v,\theta)=-\sqrt{\frac{\gamma p}{v}}<0$, $\lambda_2=0$ and  $\lambda_3=\lambda_3(v,\theta)=\sqrt{\frac{\gamma p}{v}}>0$ generate two genuinely nonlinear characteristic fields, and one linearly degenerate  characteristic field. Therefore, the self-similar solution, called {\it Riemann solution} of the Riemann problem, is determined by a combination of at most three  elementary solutions from the following five families: 1-rarefaction; 3-rarefaction; 1-shock; 3-shock and 2-contact discontinuity (see for instance \cite{Dafermos1}). These families are completely defined through their associated curves in the states plane $\R_+\times \R \times \R_+$.
For any  $(v_R,u_R, \theta_R)\in \R_+\times \R\times \R_+$, the 1-rarefaction curve $R_1(v_R,u_R, \theta_R)$ corresponds to the integral curve of the  first eigenvalue $\lambda_1$, and  is defined by  
\begin{equation}\label{(1.9)}
 R_1 (v_R, u_R, \theta_R):=\Bigg{ \{} (v, u, \theta)\Bigg{ |}v<v_R ,~s(v,\theta)=s(v_R, \theta_R):=s_R,~u=u_R-\int^v_{v_R}
 \lambda_1(s_R,v^\prime) dv^\prime\Bigg{ \}},
\end{equation}
where $s(v,\theta)=\frac{R}{\gamma-1}\ln(\frac{R}{A}\theta v^{\gamma-1})$ and $\lambda_1(s,v)=-\sqrt{\frac{\gamma p(v,s)}{v}}$.
The 3-rarefaction curve $R_3$ can be defined in the same way from the second eigenvalue  $\lambda_3$.  
For any initial values of the Riemann problem \eqref{Ei} with $(v_-,u_-, \theta_-)=(v_L,u_L, \theta_L)$, $(v_+,u_+,\theta_+)=(v_R,u_R,\theta_R)$, such that $(v_L,u_L,\theta_L)\in R_1(v_R,u_R,\theta_R)$, the solution $(v^r,u^r,\theta^r)$ of \eqref{E} is the 1-rarefaction wave defined as 
\begin{equation} \label{BESintro}
\lambda_1(v^r(t, x),\theta^r(t,x))  = \begin{cases}
\lambda_1(v_L,\theta_L) , \qquad x < \lambda_1(v_L,\theta_L)t, \\[2mm]
\di \frac{x}{t}, \qquad  \lambda_1(v_L,\theta_L)t \leq x \leq  \lambda_1(v_R,\theta_R)t, \\[2mm]
\lambda_1(v_R,\theta_R), \qquad x > \lambda_1(v_R,\theta_R)t,
\end{cases}
\end{equation}
together with 
\begin{equation}\label{rareintro}
\begin{array}{ll}
  &z_1(v^r(t,x), u^r(t,x))=  z_1 (v_L, u_L)=z_1 (v_R,u_R),\\[3mm]
  &s(v^r(t,x), \theta^r(t,x))=s(v_L, \theta_L)=s(v_R,\theta_R),
\end{array}
\end{equation}
where $\di z_1(v,u)=u+\int^v \lambda_1(s,v^\prime)dv^\prime$ and $s(v,\theta)$  are the 1-Riemann invariants to the Euler equation \eqref{E}.  The case of 3-rarefaction wave is treated similarly from the third eigenvalue $\lambda_3$.
 \vskip0.3cm
 The second characteristic field corresponding to the eigenvalue $\l_2\equiv 0$ is linearly degenerate and corresponds to the 2-contact discontinuity. For any  $(v_R,u_R, \theta_R)\in \R_+\times \R \times \R_+$, the 2-contact discontinuity curve $CD_2(v_R,u_R, \theta_R)$ can be defined by  
\begin{equation}\label{cd-c}
 CD_2 (v_R, u_R, \theta_R):=\Big{ \{} (v, u, \theta)\Big{ |}u=u_R ,~p(v,\theta)=p(v_R, \theta_R)\Big{ \}}.
\end{equation}
Whenever $(v_L,u_L,\theta_L)\in CD_2(v_R,u_R,\theta_R)$, the contact discontinuity $(v^c,u^c,\theta^c)$ connecting $(v_-,u_-,\theta_-)=(v_L,u_L,\theta_L)$ with $(v_+,u_+,\theta_+)=(v_R,u_R,\theta_R)$, is uniquely determined by 
\begin{equation} \label{cd}
 (v^c,u^c,\theta^c)(t,x)= \begin{cases}
(v_L,u_L,\theta_L) , \qquad x < u_Rt=u_L t, \\
(v_R,u_R,\theta_R), \qquad x > u_Rt=u_L t,
\end{cases}
\end{equation}
as a solution of \eqref{E}-\eqref{Ei}, that is,
\begin{eqnarray}
\begin{cases}
  v_{t}-u _{x} = 0,     \cr
    u_{t} +  p_{x}
    =0,\qquad \qquad\qquad       t>0, x\in\mathbb{R}, \cr
       \left(\frac{R}{\gamma-1}  \theta+\frac{1}{2}u^2\right) _t+    (p u)
       _{x}
    =0,\cr (v,u, \theta)( 0, x)  =
 \begin{cases}
(v_L,u_L,\theta_L), \quad x<0,\cr(v_R,u_R,\theta_R),\quad x>0.
\end{cases}
\end{cases}
\end{eqnarray}

We can now define the shock  curves using the Rankine-Hugoniot condition, as the one-parameter family of all the $(v,u,\theta)$  such that there exists $\sigma$ with:
\begin{equation}\label{RH}
\begin{array}{l} -\sigma(v_R-v)-(u_R-u)=0,\\[2mm]
\di  -\sigma(u_R-u)+(p(v_R,\theta_R)-p(v,\theta))=0,\\[2mm]
\di -\sigma(E_R-E)+(p(v_R,\theta_R)u_R-p(v,\theta)u)=0.
\end{array}
\end{equation}

The general theory shows that this condition defines actually 2 curves that meet at the point $(v_R,u_R,\theta_R)$, one for the value $\sigma=-\sqrt{-\frac{p(v_R,\theta_R)-p(v,\theta)}{v_R -v}}$ (the 1-shock curve $S_1(v_R,u_R)$ which corresponds to admissible 1-shock for $v>v_R$), and one for the value $\sigma=\sqrt{-\frac{p(v_R,\theta_R)-p(v,\theta)}{v_R -v}}$ (the 3-shock curve $S_3(v_R,u_R,\theta_R)$ with admissible 3-shock for $v<v_R$).
\vskip0.1cm
Whenever $(v_L,u_L,\theta_L)\in S_1(v_R,u_R,\theta_R)\cup S_3(v_R,u_R,\theta_R)$, the shock solution $(v^s,u^s,\theta^s)$ to \eqref{E}-\eqref{Ei} with $(v_-,u_-,\theta_-)=(v_L,u_L,\theta_L)$, $(v_+,u_+,\theta_+)=(v_R,u_R,\theta_R)$, is given by the discontinuous traveling wave defined as
\begin{equation} \label{BERintro}
 (v^s,u^s,\theta^s)(t,x)= \begin{cases}
(v_L,u_L,\theta_L) , \qquad x < \sigma t, \\
(v_R,u_R,\theta_R), \qquad x > \sigma t.
\end{cases}
\end{equation}
\vskip0.3cm
For the general case of any  states $(v_-, u_-, \theta_-), (v_+, u_+, \theta_+)\in \R_+\times \R\times \R_+$, it can be shown that there exists two (unique) intermediate states $(v_*, u_*, \theta_*), (v^*,u^*,\theta^*)\in \R^+\times \R\times \R^+$ such that $(v^*,u^*,\theta^*)$ is on a curve of the third families from $(v_+, u_+, \theta_+)$ (either $R_3(v_+, u_+, \theta_+)$ or $S_3(v_+,u_+,\theta_+)$), $(v_*, u_*, \theta_*)\in CD_2 (v^*,u^*,\theta^*)$, that is, $(v_*, u_*, \theta_*)$ is on the second contact continuity curve from $(v^*,u^*,\theta^*)$, and $(v_-, u_-, \theta_-)$ is on a curve of the first families from $(v_*,u_*,\theta_*)$ (either $R_1(v_*,u_*,\theta_*)$ or $S_1(v_*,u_*,\theta_*)$). The solution $(v,u,\theta)$ of \eqref{E}-\eqref{Ei} is then obtained by the juxtaposition of the three associated waves
$$
(v,u, \theta)(t,x)=(v_1,u_1, \theta_1)(t,x)+(v_2,u_2, \theta_2)(t,x)+(v_3,u_3, \theta_3)(t,x)-(v_*, u_*, \theta_*)-(v^*,u^*,\theta^*).
$$
The wave  $(v_1,u_1,\theta_1)$ is 1-rarefaction fan solution to \eqref{BESintro}-\eqref{rareintro} if $(v_-,u_-,\theta_-)\in R_1(v_*, u_*, \theta_*)$, or 1-shock solution to \eqref{BERintro} if $(v_-,u_-,\theta_-)\in S_1(v_*, u_*, \theta_*)$, with $(v_L,u_L,\theta_L)=(v_-, u_-, \theta_-)$, $(v_R,u_R,\theta_R)=(v_*, u_*, \theta_*)$. The wave  $(v_2,u_2,\theta_2)$ is  2-contact discontinuity solution to \eqref{cd}  if 
$(v_*, u_*, \theta_*)\in CD_2(v^*,u^*,\theta^*)$ with $(v_L,u_L,\theta_L)=(v_*, u_*, \theta_*)$, $(v_R,u_R,\theta_R)=(v^*,u^*,\theta^*)$. And the shock solution to 
\eqref{BERintro} if $(v^*,u^*,\theta^*)\in S_3(v_+,u_+,\theta_+)$, or, 3-rarefaction fan solution if $(v^*,u^*,\theta^*)\in R_3(v_+,u_+,\theta_+)$, both with the end states   $(v_L,u_L,\theta_L)=(v^*,u^*,\theta^*)$, $(v_R,u_R,\theta_R)=(v_+,u_+,\theta_+)$. Note that the cases of three single waves, i.e., shock, rarefaction, and contact discontinuity, are included as degenerate cases when $(v_-,u_-, \theta_-)$ and $(v_+,u_+,\theta_+)$ are exactly on a single wave curve without intermediate states. 

 \vskip0.3cm
 \noindent
 {\bf Viscous ansatz of the Riemann solution:}
The time-asymptotic behavior of the viscous solution to \eqref{NS} depends on whether the associated Riemann solution to the associated inviscid model \eqref{E}-\eqref{Ei} involves shock waves, a contact discontinuity, or rarefaction waves. In the case where \eqref{Ei} is a shock, the viscous counterpart   for \eqref{NS}, called viscous shock, is the traveling wave $(v^S(x-\sigma t),u^S(x-\sigma t),\theta^S(x-\sigma t))$ to \eqref{NS}:
\begin{equation}\label{VS}
\left\{
\begin{array}{ll}
\di -\sigma (v^S)^\prime-( u^S)^\prime=0,\\[3mm]
\di -\sigma( u^S)^\prime+(p^S)^\prime=\Big(\mu\frac{( u^S)^\prime}{ v^S}\Big)^\prime,\\[3mm]
\di -\sigma( E^S)^\prime+(p^S u^S)^\prime=\Big(\kappa\frac{( \theta^S)^\prime}{ v^S}\Big)^\prime+\Big(\mu\frac{u^S( u^S)^\prime}{ v^S}\Big)^\prime ,\\[3mm]
\di ( v^S, u^S, \theta^S)(-\infty)=(v^*, u^*, \theta^*), \qquad ( v^S, u^S, \theta^S)(+\infty)=(v_+, u_+, \theta_+).
\end{array}
\right.
\end{equation}
where $E^S:=\frac{R}{\gamma-1}\theta^S+\frac{(u^S)^2}{2}$ and $p^S:=p(v^S, \theta^S).$

\vskip0.3cm
The viscous version of the inviscid contact
discontinuity  connecting $(v_*,u_*,\theta_*)$ with $(v^*,u^*,\theta^*)$, called viscous contact wave $\big(v^{C}, u^{C},\theta^{C}\big)(t,x)$ can be defined by  \cite{Huang-Xin-Yang}:
\begin{equation}\label{vc}
\begin{array}{ll}
  \di  v^{C}(\frac{x}{\sqrt{1+t}})
    =  \frac{R\Theta^{\rm sim}}{p_*},\\[3mm]
   \di u^{C}(t,\frac{x}{\sqrt{1+t}})
    = u_* +\frac{(\gamma-1)\kappa\Theta^{\rm sim}_x}{R\gamma\Theta^{\rm sim}},\\[4mm]
   \di \theta^{C}(\frac{x}{\sqrt{1+t}})
    = \Theta^{\rm sim},
\end{array}
\end{equation}
where $\Theta^{\rm sim}=\Theta^{\rm sim}\left(\frac{x}{\sqrt{1+t}}\right)$ is the unique
self-similar solution to the following nonlinear diffusion equation
\begin{eqnarray}
\begin{cases}
 \Theta_t= \frac{(\gamma-1)\kappa p_*}{R^2 \gamma}\left(\frac{\Theta_x}{\Theta}\right)_x, \cr
    \Theta(t,-\infty)=\theta_{*},\ \  \Theta(t,+\infty)=\theta^{*}.
\end{cases}
\end{eqnarray}

\vskip0.3cm
The definition of the aproximate ansatz 
associated with the self-similar rarefaction fan $(v^r,  u^r, \theta^r)(\frac xt)$ will be described in  the next section ( see Section 2.1).

\vskip0.3cm

Given the end states $(v_\pm, u_\pm, \theta_\pm)\in\bbr_+\times\bbr\times\bbr_+$ in \eqref{in}, we consider the general case that there exist two unique intermediate states $(v_*, u_*, \theta_*)$ and $(v^*, u^*, \theta^*)$ such that 
\begin{equation}
(v_-,u_-,\theta_-)\in R_1(v_*, u_*, \theta_*), \ \ (v_*, u_*, \theta_*)\in CD_2 (v^*, u^*, \theta^*),\ \  (v^*, u^*, \theta^*) \in S_3 (v_+, u_+, \theta_+).
\end{equation}
We consider, as viscous ansatz, the superposition wave:
\beq\label{SW}
\begin{array}{ll}
\di \Big( v^r(\frac xt)+ v^C(\frac{x}{\sqrt{1+t}})+ v^S(x-\s t)-v_*-v^*,
\\
\di \quad u^r(\frac xt)+ u^C(t,\frac{x}{\sqrt{1+t}})+ u^S(x-\s t)-u_*-u^*, \\
\di \quad  \theta^r(\frac xt)+ \theta^C(\frac{x}{\sqrt{1+t}})+ \theta^S(x-\s t)-\theta_*-\theta^*\Big)
\end{array}
\eeq
where $(v^r,  u^r, \theta^r)(\frac xt)$ is the 1-rarefaction wave defined in \eqref{rare} , $(v^C, \theta^C)(\frac x{\sqrt{1+t}}), u^C(t, \frac x{\sqrt{1+t}})$ is the 2-viscous contact wave defined in \eqref{vc} and $(v^S, u^S, \theta^S)(\xi)~(\xi=x-\s t)$ is the 3-viscous shock wave defined in Lemma \ref{lemma1.3}.

For simplicity, we consider the following system in non-divergence form which is equivalent to the original system \eqref{NS}:
 \begin{align}
\begin{aligned}\label{NS-0}
\left\{ \begin{array}{ll}
        v_t- u_x =0,\\
       u_t +p(v,\theta)_x=  (\mu\frac{u_x}{v})_x,\\
       \frac{R}{\gamma-1}\theta_t +p(v,\theta) u_x=(\kappa\frac{\theta_x}{v})_x+\mu\frac{u_x^2}{v}. \\
        \end{array} \right.
\end{aligned}
\end{align}

\begin{theorem}\label{thm:main}
For a given constant state $(v_+,u_+,\theta_+)\in\bbr_+\times\bbr\times\bbr_+$, there exist constants $\delta_0, \eps_0>0$ such that the following holds true.\\
For any $(v^*,u^*,\theta^*)\in S_3(v_+,u_+,\theta_+),$ $(v_*,u_*,\theta_*)\in CD_2(v^*,u^*,\theta^*)$ and $(v_-,u_-,\theta_-)\in R_1(v_*,u_*,\theta_*)$ such that 
$$|v_+-v^*|+|v^*-v_*|+|v_*-v_-|\leq \delta_0,$$ 
denote $(v^r,u^r,\theta^r)(\frac xt)$ the 1-rarefaction solution to \eqref{E} with end states $(v_-,u_-,\theta_-)$ and $(v_*,u_*,\theta_*)$, $(v^C,u^C,\theta^C)(\frac{x}{\sqrt{1+t}})$ the 2-viscous contact wave defined in \eqref{vc} and \eqref{vc-equ} with end states $(v_*,u_*,\theta_*)$ and $(v^*,u^*,\theta^*)$, and $(v^S,u^S,\theta^S)(x-\sigma t)$ the 3-viscous shock solution of  \eqref{NS-0} with end states $(v^*,u^*,\theta^*)$ and $(v_+,u_+,\theta_+)$.
Let $(v_0, u_0,\theta_0)$ be  any initial data such that 
\begin{equation}\label{i-p}
\sum_{\pm}\Big( \|(v_0-v_\pm, u_0-u_\pm,\theta_0-\theta_\pm)\|_{L^2(\bbr_\pm)} \Big) + \|(v_{0x},u_{0x},\theta_{0x})\|_{L^2(\bbr)} < \eps_0,
\end{equation}
where $\bbr_- :=-\bbr_+ = (-\infty,0)$.\\
 Then,  the compressible Navier-Stokes-Fourier system \eqref{NS-0} (and \eqref{NS})  admits a unique global-in-time solution $(v,u,\theta)(t,x)$ for all $t\in \bbr_+$. Moreover, there exists an absolutely continuous shift $\mb{X}(t)$ such that
 \begin{align}
\begin{aligned}\label{ext-main}
&v(t,x)- \Big( v^r(\frac xt)+v^C(\frac{x}{\sqrt{1+t}})+ v^S(x-\sigma t-\mb X(t))-v_*-v^* \Big) \in C(0,+\infty;H^1(\bbr)),\\
& u(t,x)- \Big( u^r(\frac xt)+u^C(t,\frac{x}{\sqrt{1+t}})+u^S(x-\sigma t-\mb X(t))-u_*-u^* \Big) \in C(0,+\infty;H^1(\bbr)),\\
& \theta(t,x)- \Big( \theta^r(\frac xt)+\theta^C(\frac{x}{\sqrt{1+t}})+\theta^S(x-\sigma t-\mb X(t))-\theta_*-\theta^* \Big) \in C(0,+\infty;H^1(\bbr)),\\
& \theta_{xx}(t,x)-\theta^S_{xx}(x-\sigma t-\mb X(t))\in L^2(0,+\infty; L^2(\bbr)),\\
& u_{xx}(t,x)-u^S_{xx}(x-\sigma t-\mb X(t))\in L^2(0,+\infty; L^2(\bbr)).
\end{aligned}
\end{align}
In addition, as $t\to+\infty$,
\begin{equation}\label{con}
\begin{array}{l}
\di \sup_{x\in\bbr}\Big|(v,u,\theta)(t,x)-  \Big(v^r(\frac xt)+v^C(\frac{x}{\sqrt{1+t}}) +v^S(x-\sigma t-\mb X(t))-v_*-v^*,\\[3mm] \di \qquad\qquad \qquad\qquad\quad u^r(\frac xt)+u^C(t,\frac{x}{\sqrt{1+t}})+u^S(x-\sigma t-\mb X(t))-u_*-u^*,\\[3mm] \di \qquad\qquad \qquad\qquad \quad\theta^r(\frac xt)+\theta^C(\frac{x}{\sqrt{1+t}})+\theta^S(x-\sigma t-\mb X(t))-\theta_*-\theta^* \Big) \Big| \to 0,
\end{array}
\end{equation}
and
\begin{equation}\label{as}
\lim_{t\rightarrow+\infty} |\dot {\mb X}(t) |=0.
\end{equation}
\end{theorem}

\begin{remark}
Theorem \ref{thm:main} states that if the two far-field states $(v_\pm, u_\pm,\theta_\pm)$ in \eqref{in} are connected by the superposition of a shock, a contact discontinuity,  and a rarefaction wave, then the solution to the full compressible Navier-Stokes-Fourier equations \eqref{NS} or \eqref{NS-0} converges in long time to the superposition wave made from the inviscid self-similar rarefaction wave, the viscous contact wave, and viscous shock wave with the shift $\mb{X}(t)$.
\end{remark}

\begin{remark}
The shift function $\mb{X}(t)$ (defined in \eqref{X(t)}) is proved to satisfy the time-asymptotic behavior \eqref{as}, which implies that 
$$
\lim_{t\rightarrow+\infty}\frac{{\mb X}(t)}{t}=0,
$$
that is, the shift function ${\mb X}(t)$ grows at most sub-linearly w.r.t. the time $t$ and the shifted viscous shock wave still keeps the original traveling wave profile time-asymptotically.

\end{remark}



The main new ingredient of our proof is the use of the method of $a$-contraction with shifts  \cite{Kang-V-NS17} to track the stability of the viscous shock. The method is based on the relative entropy introduced by Dafermos \cite{Dafermos4} and DiPerna \cite{DiPerna}. It is  energy based, and so meshes seamlessly with the treatments of the rarefaction and the viscous contact wave as in \cite{Huang-Xin-Yang, HLM}.  
\vskip0.3cm
\noindent
{\bf The method of $a$-contraction with shifts:}
The method of $a$-contraction with shifts was developed in \cite{KVARMA} (see also \cite{LV}) to study the stability of extremal shocks for inviscid system of conservation laws, as for example, the Euler system \eqref{E}.  
Consider the entropy of the system (which is actually the physical energy) defined for any state $U=(v,u,E)$ as $s(U).$
We then consider the relative entropy defined in \cite{Dafermos4} for any two states $U=(v,u, E)$, $\overline{U}=(\bar v,\bar u, \bar E)$:
$$
s(U|\bar U):=s(U)-s(\bar U)-ds(\bar U)\cdot (U-\bar U).
$$
Note that the physical entropy $$\eta(U|\overline U):=-\bar\theta \ s(U|\bar U)$$ is nonnegative and equal to zero if and only if $U=\overline U$. Therefore $\eta(U|\overline U)$ can be used as a pseudo-distance between $U$ and $\overline U$. 
It can be shown that rarefactions $\overline U$ (that is solutions to  \eqref{BESintro}-\eqref{rareintro}) have a contraction property for this pseudo-metric (see for instance \cite{Vasseur_Book}). Indeed, for any weak entropic solution $U$ to \eqref{E},  it can be shown that
$$
\frac{d}{dt} \int_\R \eta(U|\bar U)\,dx\leq 0.
$$
 The contraction property is not true  if $\overline U$ is a shock (that is traveling waves  \eqref{BERintro} verifying the Rankine-Hugoniot conditions \eqref{RH}). However, the contraction property can be recovered  up to a shift, after weighting the relative entropy (see  \cite{KVARMA}). Indeed, there exists weights $a_-,a_+>0$ (depending only on the shock $\overline U$) such that for any weak entropic solution $U$ of \eqref{E} (verifying a mild condition called strong trace property)  there exists a Lipschitz  shift function $t\to X(t)$ such that 
 $$
\frac{d}{dt}\left\{ a_-\int_{-\infty}^{X(t)}\eta(U|\bar{U})\,dx+a_+\int_{X(t)}^\infty\eta(U|\bar U)\,dx\right\}\leq 0.
 $$
This was first proved in the scalar case by Leger \cite{Leger} for $a_-=a_+$. It has been shown in \cite{Serre-Vasseur} that the contraction with $a_-=a_+$ is usually false for most systems. Therefore the weighting via the coefficients $a_-,a_+$ is essential. Note that in the case of the full Euler system, the a-contraction property up to shifts is  true for all the single wave patterns, including the 1-shocks, 3-shocks (see \cite{Vasseur-2013}), and  the 2-contact discontinuities (see \cite{SV-16}). Although the $a$-contraction property with shifts holds for general extremal shocks, it is not always true for intermediate shocks  (see \cite{Kang} for instance). 
\vskip0.3cm
The first extension of the method  to viscous models was done in the 1D scalar case \cite{Kang-V-1}  (see also \cite{Kang19}) and then in the multi-D case \cite{KVW}. The case of the barotropic Navier-Stokes equation \eqref{NS} was treated in \cite{Kang-V-NS17} (see also \cite{Kang-V-2shocks} for the case of 2 shocks). The $a$-contraction property takes place in variables associated with the BD entropy (see \cite{BD-06}): $U=(v,h)$, where $h$ is the effective velocity defined as $h=u-(\mathrm{ln} \ v)_x$. In these variables, system \eqref{NS} with $ \mu=1$ is transformed as 
\begin{align}
\begin{aligned}\label{hNS-1intro}
\left\{ \begin{array}{ll}
 v_t -h_x= (\ln v)_{xx}, \\ 
        h_t+p(v)_x=0. 
        \end{array} \right.
\end{aligned}
\end{align}
 The only nonlinear term of the hyperbolic system \eqref{E} is the pressure which is a function of $v$. The system \eqref{hNS-1intro} is then better than \eqref{NS} since the diffusion is in the variable $v$ corresponding to the nonlinear term $p(v)$. It was shown in \cite{Kang-V-NS17} that  there exists a monotonic function $x\to a(x)$ (with limits $a_\pm$ at $\pm\infty$), depending only on the viscous shock $\bar{U}=(\bar v,\bar h)$ solution to 
 (in the $(v,u)$ variables), such that for any solution $U$ to \eqref{hNS-1intro}, there exists a shift function $t\to \mb{X}(t)$ with
 $$
 \frac{d}{dt}\int_\R a(x-\mb{X}(t))\eta(U(t,x)|\bar U(x-\mb{X}(t)))\,dx \leq 0.
$$
Note that the $a$-contraction result of   \cite{Kang-V-NS17} provides uniform stability for viscous shocks with respect to the strength of the viscosity. This is used in  \cite{KV-Inven} to obtain the stability of inviscid shocks of the Isentropic Euler equation among any inviscid limits of the associated Navier-Stokes equation (see also \cite{VW, KVW2} for an example of similar stability for a single 1D contact discontinuity and 3D planar contact discontinuity).

\vskip0.3cm
In the simplified situation of the barotropic Navier-Sotkes equation, the $a$-contraction method has been extended to composite waves made of a shock and a rarefaction in \cite{KVW-2022} (see also \cite{WangTW} for a 3D extension). The aim of this paper is to extend the $a$-contraction theory to the more complex Navier-Stokes-Fourier system \eqref{NS} and to apply it to the study of the long-time stability  to the composite wave made of a viscous shock wave, a viscous contact wave,  and an inviscid rarefaction wave.

\vskip0.3cm
The rest of the paper is organized as follows. We begin with preliminaries in Section \ref{sec-preli}. It includes known properties on the rarefaction and on the viscous contact and shock waves. 
The general set up is laid out in Section \ref{sec-thm}. We introduce the local existence of solution, the construction of the shift function and an a priori estimates result  in Proposition \ref{prop2}, which together with a continuing argument implies Theorem \ref{thm:main}. The last two sections are dedicated to the proof of Proposition \ref{prop2} for the main a priori estimates . The $a$-contraction argument is set up in Section \ref{sec-acontraction} where  global a priori estimates are proved to conclude the proof of Proposition \ref{prop2}. Finally, three Appendixs are devoted to the derivation of relative entropy, the sharp estimate for the diffusion and the proof of Lemma \ref{cw-lemma}, respectively.
\section{Preliminaries} \label{sec-preli}
\setcounter{equation}{0}
We gather in this section some well-known results which will be useful in the rest of the paper.


\subsection{Approximate rarefaction wave}

We now recall important properties of the 1-rarefaction waves. Consider a $(v_*, u_*,\theta_*)$ in \eqref{in}, and $(v_-,u_-,\theta_-)\in R_1(v_*, u_*, \theta_*)$.
Set $w_- = \lambda_1(v_-,\theta_-), w_* = \lambda_1(v_*,\theta_*)$, and
consider the Riemann problem for the inviscid Burgers equation:
\begin{equation} \label{BE}
\begin{cases}
\displaystyle w_t + ww_x = 0, \\
\displaystyle w(0, x) = w_0^r(x) = \begin{cases}
w_-,  \quad x < 0,\\
w_*,  \quad x > 0.
\end{cases}
\end{cases}
\end{equation}
If $w_- < w_*$, then \eqref{BE} has the rarefaction wave fan $w^r(t, x) = w^r(x/t)$ given by
\begin{equation} \label{BES}
w^r(t, x) = w^r(\frac{x}{t}) = \begin{cases}
w_- , \qquad x < w_-t, \\
\frac{x}{t}, \qquad w_-t \leq x \leq w_*t, \\
w_* , \qquad x > w_*t.
\end{cases}
\end{equation}
 It is easy to check that the 1-rarefaction wave $(v^r, u^r,\theta)(t, x) = (v^r, u^r,\theta^r)(x/t)$ to the Riemann problem \eqref{E}-\eqref{Ei}, defined in \eqref{BESintro}-\eqref{rareintro},  is given explicitly by
\begin{equation}\label{rare}
\begin{array}{ll}
  &\lambda_1(v^r(\frac xt),\theta^r(\frac xt)) = w^r(\frac xt), \\[1mm]
  &z_1(v^r(\frac xt), u^r(\frac xt))=  z_1 (v_-, u_-)=z_1 (v_*,u_*),\\[1mm]
  &s(v^r(\frac xt),\theta^r(\frac xt))=s(v_-,\theta_-)=s(v_*,\theta_*).
\end{array}
\end{equation}
 The self-similar 1-rarefaction wave $(v^r, u^r, \theta^r)(x/t)$ is Lipschitz continuous and satisfies the Euler system a.e. for $t>0$,
\begin{align}
\begin{aligned}\label{E-ra}
\left\{ \begin{array}{ll}
        v^r_t- u^r_x =0,\\[1mm]
       u^r_t  +p(v^r,\theta^r)_x= 0,\\[1mm]
       \frac{R}{\gamma-1}\theta^r_t+p(v^r,\theta^r) u^r_x=0. 
        \end{array} \right.
\end{aligned}
\end{align}
Let $\delta_R:=|v_*-v_-|$ denote the strength of the rarefaction wave. Notice that $\delta_R\sim |u_*-u_-|\sim |\theta_*-\theta_-|$ by $\eqref{rare}_2$.

As in \cite{MN86}, by using the smooth solution to the Burgers equation:
\begin{equation} \label{ABE}
\begin{cases}
\displaystyle w_t + ww_x = 0, \\
\displaystyle w(0, x) = w_0(x) = \frac{w_* + w_-}{2} + \frac{w_* - w_-}{2} \tanh x.
\end{cases}
\end{equation}
we will consider the smooth approximate 1-rarefaction wave $(v^R, u^R, \theta^R)(t, x)$ of the 1-rarefaction wave fan $(v^r, u^r, \theta^r)(\frac xt)$ by
\begin{align}
  \begin{aligned} \label{AR}
   &\lambda_{1-}:=\lambda_1(v_-,\theta_-) = w_-,\ \lambda_{1*}:=\lambda_1(v_*, \theta_*) = w_*,\\
    &\lambda_1(v^R,\theta^R)(t, x) = w(1+t, x), \\
    &z_1(  v^R,  u^R)(t, x) = z_1 (v_-, u_-)=z_1 (v_*,u_*),\\
    &s(v^R, \theta^R)(t, x)=s(v_-, \theta_-)=s(v_*,\theta_*),
  \end{aligned}
\end{align}
where $w(t,x)$ is the smooth solution to the Burgers equation in \eqref{ABE}.
One can easily check that the above approximate rarefaction wave $(v^R, u^R, \theta^R)$ satisfies the Euler system:
\begin{equation}  \label{ARW}
  \begin{cases}
    \displaystyle v^R_t -u^R_x = 0, \\
    \displaystyle u^R_t + p(v^R, \theta^R)_x = 0,\\
    \di \frac{R}{\gamma-1}\theta^R_t+p(v^R, \theta^R) u^R_x=0. \\
  \end{cases}
\end{equation}
The following lemma comes from \cite{MN86}.
\begin{lemma} \label{lemma1.2}
	The smooth approximate 1-rarefaction wave $(v^R, u^R, \theta^R)(t, x)$ defined in \eqref{AR} satisfies the following properties. Let $\delta_R$ denote the rarefaction wave strength as $\delta_R := |v_* - v_-|\sim |u_*-u_-|\sim |\theta_*-\theta_-|$.
	
	(1)~$ u^R_x = \frac{2 v^R}{(\gamma + 1) }w_x > 0,$ $ v^R_x = \frac{v^R}{\sqrt{R\gamma\theta^R}} u^R_x>0$ and $\theta^R_x=-\frac{(\gamma-1)\theta^R}{v^R}v^R_x<0,$  $\forall x \in \bbr,\ t \geq 0$.
	
	(2)~The following estimates hold for all $t \geq 0$ and $p \in [1, + \infty]$:
	\begin{align*}
	&\|(v^R_x, u^R_x, \theta^R_x)\|_{L^p} \leq C \min\{\delta_R, \delta_R^{1/p}(1+t)^{-1+1/p}\}, \\
	&\|(v^R_{xx}, u^R_{xx}, \theta^R_{xx})\|_{L^p} \leq C \min\{\delta_R, (1+t)^{-1}\}, \\
	& |u^R_{xx}| \leq C |u^R_{x}|,\quad |\theta^R_{xx}| \leq C |\theta^R_{x}| ,\quad \forall x\in\bbr.
	\end{align*}
	
	(3)~ For $x\geq \lambda_{1*}(1+t), t\geq 0,$ it holds that
	\begin{align*}
    &|(v^R, u^R, \theta^R)(t, x)-(v_*, u_*, \theta_*)| \leq C\delta_R \ e^{-2|x-\lambda_{1*}(1+t)|}, \\
    &|(v^R_x, u^R_x, \theta^R_x)(t, x)|\leq C\delta_R\ e^{-2|x-\lambda_{1*}(1+t)|}.
  \end{align*}
  
  (4)~ For $x\le\lambda_{1-}(1+t), t\geq 0,$ it holds that
	\begin{align*}
    &|(v^R, u^R, \theta^R)(t, x)-(v_-,u_-, \theta_-)| \leq C\delta_R \ e^{-2|x-\lambda_{1-}(1+t)|}, \\
    &|(v^R_x, u^R_x, \theta^R_x)(t, x)|\leq C\delta_R\ e^{-2|x-\lambda_{1-}(1+t)|}.
  \end{align*}
  
(5)~$\di \lim_{t \to +\infty} \sup_{x \in \bbr} |(v^R, u^R, \theta^R)(t, x) - (v^r, u^r, \theta^r)( \frac xt)| = 0$.
\end{lemma}

\subsection{Viscous contact wave}
It is known that the inviscid contact discontinuity is  time-asymptotically unstable for the compressible Euler equations \eqref{E}. However, a viscous contact wave, which is a viscous version of inviscid contact discontinuity, can be constructed and proved to be time-asymptotically stable to both "artificial viscosity" system \cite{Xin-c, Liu-Xin97} and the physical Navier-Stokes system \cite{Huang-Xin-Yang, HLM}. 
\begin{lemma}\label{Lemma 2.2.}
(\cite{Huang-Xin-Yang}) The
viscous contact wave  $\big(\vc, u^{C},
\theta^{C}\big)(t,x)$ defined in \eqref{vc} satisfies
\beq\label{vc-pe}
\begin{array}{ll}
\di\big(v^C-v_*, u^C-u_*,\theta^{C}-\theta_* \big)= O(1)\delta_{_C}e^{ -\frac{C_1 x^2}{1+t}}, \quad \forall x<0;\\
\di\big(v^C-v^*, u^C-u^*,\theta^{C}-\theta^* \big)= O(1)\delta_{_C}e^{ -\frac{C_1 x^2}{1+t}}, \quad \forall x>0;\\
\di (\partial_x^nv^{C},\partial_x^n\theta^{C})(t,x)
   =O(1)\delta_{_C}(1+t)^{-\frac{n}{2}}e^{ -\frac{C_1 x^2}{1+t}}, \qquad\forall x\in \mathbb{R}, \quad n=1,2,\cdots;\\
\di \partial_x^n u^{C}(t, x)
  =O(1)\delta_{_C}(1+t)^{-\frac{1+n}{2}}e^{ -\frac{C_1 x^2}{1+t}},\qquad\qquad\ \   \forall x\in \mathbb{R},\quad n=1,2,\cdots;\\
\end{array}
\eeq
where $\delta_{_C}=|v^*-v_*|\sim|\theta^{*}-\theta_{*}|$ is the amplitude of the viscous contact
wave, and $C_1>0 $ is generic constant. 
\end{lemma}

Then
 the viscous contact wave  $\big(\vc, u^{C},
\theta^{C}\big)(t,x)$ defined in \eqref{vc}
 satisfies the system
\begin{eqnarray}\label{vc-equ}
\begin{cases}
  v^{C}_t-u^{C}_{x} = 0, \cr
    u^{C}_t+ p^{C}_x
    =\mu\big( \frac{u^{C}_{x}}{ v^{C}} \big)_x+Q^C_1, \cr
        \frac{R}{\gamma-1}\theta^{C}_t+  p^{C} u^{C}_x
    =\kappa\big(\frac{\theta^{C}_{x}}{v^{C}} \big)_x+\mu\frac{(u^{C}_{x})^2}{v^{C}}+Q^C_2
\end{cases}
\end{eqnarray}
where $ p^{C}=p\big( v^{C}, \theta^{C}\big)$ and
\beq\label{QC12}
\begin{aligned}
&Q^C_1=u^C_t-\mu\big( \frac{u^{C}_{x}}{ v^{C}} \big)_x=O(1)\delta_{_C}(1+t)^{-\frac32}e^{-\frac{2C_1 x^2}{1+t}},\\
& Q^C_2=-\mu\frac{(u^{C}_{x})^2}{v^{C}}=O(1)\delta_{_C}(1+t)^{-2}e^{-\frac{2C_1 x^2}{1+t}},
\end{aligned}
\eeq
as $x\rightarrow\pm \infty$ due to Lemma \ref{Lemma 2.2.}. Moreover, from \eqref{vc} and Lemma \ref{Lemma 2.2.}, it holds that for $\forall p\geq1,$
\beq
\|\big(\vc, u^{C},
\theta^{C}\big)(t,\cdot)-\big(v^c, u^{c},
\theta^{c}\big)(t,\cdot)\|_{L^p(\bbr)}=O(1)\kappa^{\frac{1}{2p}}(1+t)^{\frac{1}{2p}},
\eeq
which implies that viscous contact wave $(\vc, u^{C},
\theta^{C})(t,x)$ can converge to the inviscid contact discontinuity $(v^c, u^{c},
\theta^{c})(t,x)$ in $L^p-$norm $(\forall p\geq1)$ at any finite time interval as the heat conductivity coefficient $\kappa\rightarrow0+$, however, they could be far away at large time.

\subsection{Viscous shock wave}
We turn to the 3-viscous shock wave connecting $(v^*, u^*, \theta^*)$ and $(v_+,u_+,\theta_+)$ such that $(v^*, u^*, \theta^*)\in S_3(v_+,u_+,\theta_+)$. Recall the Rankine-Hugoniot condition \eqref{RH}
and the Lax entropy condition
\begin{equation}\label{ec}
\lambda_3(v_+,\theta_+)<\sigma<\lambda_3(v^*,\theta^*).
\end{equation}
The Riemann problem \eqref{E}-\eqref{Ei} admits a unique 3-shock solution
\begin{align}
\begin{aligned}\label{Shock}
(v^s,u^s,\theta^s)(t,x)=\left\{ \begin{array}{ll}
        (v^*,u^*,\theta^*),\quad x<\sigma t,\\[2mm]
       (v_+,u_+,\theta_+),\quad x>\sigma t, \\
        \end{array} \right.
\end{aligned}
\end{align}
where it follows from \eqref{RH} that
\begin{equation}
\sigma=\sqrt{-\frac{p(v_+,\theta_+)-p(v^*,\theta^*)}{v_+-v^*}}.
\end{equation}
As a traveling wave solution of \eqref{NS-0} (equivalently, of \eqref{NS}), the 3-viscous shock wave $(v^S, u^S, \theta^S)(\xi)$ satisfies the system of ODEs \eqref{VS}.

Integrating the system \eqref{VS} over $(\pm \infty, \xi]$ gives that
\begin{equation}\label{VS1}
\left\{
\begin{array}{ll}
\di -\sigma (v^S-v_+)-( u^S-u_+)=-\sigma (v^S-v^*)-( u^S-u^*)=0,\\[3mm]
\di \mu\frac{( u^S)^\prime}{v^S}=-\sigma( u^S-u_+)+(p^S-p_+)=-\sigma( u^S-u^*)+(p^S-p^*),\\[3mm]
\di \kappa\frac{( \theta^S)^\prime}{ v^S}+\mu\frac{u^S( u^S)^\prime}{ v^S}=-\sigma( E^S-E_+)+(p^S u^S-p_+u_+)\\[3mm]
\di \qquad\qquad\qquad\qquad\ =-\sigma( E^S-E^*)+(p^S u^S-p^*u^*).\\[3mm]
\end{array}
\right.
\end{equation}
Then the existence of shock profiles  $(v^S, u^S, \theta^S)(\xi)$ is equivalent to the existence of solution to the following autonomous system of ODEs:
\begin{equation}\label{VS2}
\left\{
\begin{array}{ll}
\di \mu\sigma\frac{( v^S)^\prime}{v^S}=(p^S-p_+)+\sigma^2(v^S-v_+)=(p^S-p^*)+\sigma^2(v^S-v^*),\\[3mm]
\di \kappa\frac{( \theta^S)^\prime}{\sigma v^S}=\frac{R}{\gamma-1}(\theta^S-\theta_+)+p_+(v^S -v_+)-\frac12\sigma^2(v^s-v_+)^2\\[3mm]
\di \qquad\quad =\frac{R}{\gamma-1}(\theta^S-\theta^*)+p^*(v^S -v^*)-\frac12\sigma^2(v^s-v^*)^2.
\end{array}
\right.
\end{equation}

The properties of the 3-viscous shock wave $(v^S, u^S, \theta^S)(\xi)$ can be listed as follows. 

\begin{lemma} \label{lemma1.3}
For any state  $(v_+,u_+,\theta_+)$, there exists a constant $C>0$ such that the following is true. For any end state such that $(v^*, u^*,\theta^*)\in S_3(v_+,u_+,\theta_+)$, there exists a unique solution  $(v^S, u^S, \theta^S)(\xi)$ to \eqref{VS}. Let $\delta_S$ denote the strength of the shock as $\delta_S:=|v_+-v^*|\sim|u_+-u^*|\sim |\theta_+-\theta^*|$. Then, it holds that
 \begin{align}
\begin{aligned}\label{shock-base}
& u^S_\xi<0, \qquad v^S_\xi >0, \qquad \theta^S_\xi<0,\quad  \forall\xi\in\bbr,\\
& | (v^S(\xi) -v^*, u^S(\xi) -u^*, \theta^S(\xi) -\theta^*)|\leq C\deltas\ e^{-C\delta_S |\xi|}, \quad \xi<0,\\[1mm]
& | (v^S(\xi) -v_+, u^S(\xi) -u_+, \theta^S(\xi) -\theta_+)|\leq C\deltas\ e^{-C\delta_S |\xi|}, \quad \xi>0,\\[1mm]
&|( v^S_\xi, u^S_\xi,\theta^S_\x)|\leq C\deltas^2\ e^{-C\delta_S |\xi|}, \quad\forall\xi\in\bbr,\\[1mm]
&|( v^S_{\xi\xi}, u^S_{\xi\xi},\theta^S_{\x\x})|\leq C\delta_S |( v^S_\xi, u^S_\xi, \theta_\x^S)|, \quad \forall\xi\in\bbr.
\end{aligned}
\end{align}		
In particular,  $|v^S_\xi|\sim  |u^S_\xi | \sim |\theta^S_\xi |$ for all $\xi\in\bbr$, more explicitly,
\beq\label{shock-vu}
\Big| (u^S)_\xi + \sigma^* (v^S)_\xi \Big| \le C \deltas|(v^S)_\xi|, \quad\forall \xi\in\bbr,
\eeq
and
\beq\label{theta-s}
\Big| (\theta^S)_\xi + \frac{(\gamma-1)p^*}{R}(v^S)_\xi \Big| \le C \deltas|(v^S)_\xi|,\quad\forall \xi\in\bbr,
\eeq
where 
\beq\label{pstar}
p^*:=p(v^*, \theta^*)=\frac{R\theta^*}{v^*}\quad \mbox{and}\quad \sigma^*:=\sqrt\frac{\gamma p^*}{v^*} = \frac{\sqrt{\gamma R\theta^*}}{v^*},
\eeq
which satisfies
\beq\label{sm1}
|\sigma-\sigma^*|\le C\delta_S.
\eeq
\end{lemma}

\begin{proof}
We here prove the last estimate of \eqref{shock-base}, and the explicit estimates \eqref{shock-vu}-\eqref{theta-s}. The remaining estimates can be found in \cite{HM} and \cite{MN-S} (see also \cite{Kang-V-NS17}).

\noindent$\bullet$ {\bf Proof of \eqref{shock-vu}-\eqref{theta-s}:}

The estimate \eqref{shock-vu} is easily shown by \eqref{sm1} and the equation 
$( u^S)_\xi=-\sigma (v^S)_\xi$.\\
To prove \eqref{theta-s}, we use the following fraction from  \eqref{VS2} that
 \begin{align}
\begin{aligned}\label{first-der0}
\frac{\theta^S_\x}{v^S_\x}=\frac{\mu \sigma^2}{\kappa}\frac{\frac{R}{\gamma-1}(\theta^S-\theta^*)+p^*(v^S -v^*)-\frac12\sigma^2(v^S-v^*)^2}{(p^S-p^*)+\sigma^2(v^S-v^*)}.
\end{aligned}
\end{align}
Since $\theta^S_\xi<0$ and $v^S_\xi >0$ for all $\x$, we regard $\theta^S$ as a smooth function of $v^S$ by considering $\theta^S\circ (v^S)^{-1} : [v^*,v_+]\to [\theta_+,\theta^*]$ with $\theta^*=\theta^S(v^*)$ and  $\theta_+=\theta^S(v_+)$.
Then, we define a smooth function $F(v^S)$ by the right-hand side of \eqref{first-der0}, that is,
\[
F(v^S):= \frac{\mu \sigma^2}{\kappa}\frac{\frac{R}{\gamma-1}(\theta^S-\theta^*)+p^*(v^S -v^*)-\frac12\sigma^2(v^S-v^*)^2}{(p^S-p^*)+\sigma^2(v^S-v^*)}.
\]
We will estimate $F(v_-)$ as below.
Since
\[
F(v_-)= \lim_{\xi\to-\infty} \frac{\theta^S_\x}{v^S_\x} = \lim_{v^S\to v^*} \frac{d\theta^S}{dv^S} =  \lim_{v^S\to v^*} \frac{\theta^S-\theta^*}{v^S-v^*} =: l^*<0,
\]
and
\[
\lim_{v^S\rightarrow v^*}\frac{p^S-p^*}{v^S-v^*}=\lim_{v^S\rightarrow v^*}\Big(\frac{R}{v^S}\frac{\theta^S-\theta^*}{v^S-v^*}-\frac{p^*}{v^S}\Big)
=\frac{Rl^*-p^*}{v^*},
\]
we take $v^S\to v^*$ on \eqref{first-der0} to have
$$
l^* = \lim_{\xi\to-\infty} \frac{\theta^S_\x}{v^S_\x}=\frac{\mu \sigma^2}{\kappa}\frac{\frac{R}{\gamma-1}  \lim_{v^S\to v^*} \frac{\theta^S-\theta^*}{v^S-v^*}+p^*}{\di \lim_{v^S\rightarrow v^*}\frac{p^S-p^*}{v^S-v^*}+\sigma^2} = \frac{\mu \sigma^2}{\kappa}\frac{\frac{R}{\gamma-1}  l^* +p^*}{\frac{Rl^*-p^*}{v^*} +\sigma^2} ,
$$
and so,
\[
(l^*)^2-\left(\frac{p^*-\s^2v^*}{R}+\frac{\mu\s^2 v^*}{\kappa(\gamma-1)}\right)l^*-\frac{\mu\s^2\theta^*}{\kappa}=0.
\]
But, since $l_*:=-\frac{(\gamma-1)p^*}{R}$ satisfies
\[
(l_*)^2-\left(\frac{p^*-(\s^*)^2v^*}{R}+\frac{\mu(\s^*)^2 v^*}{\kappa(\gamma-1)}\right)l_*-\frac{\mu(\s^*)^2\theta^*}{\kappa}=0,
\]
we use \eqref{sm1} to have
\[
|l^*-l_*| \le C\deltas,
\]
and so,
\[
\left| F(v_-) + \frac{(\gamma-1)p^*}{R}\right| \le C\deltas.
\]
This together with $|F(v^S)-F(v_-)|\le C|v^S-v_-|\le C\deltas$ for all $\xi$ implies
\[
\left| \frac{\theta^S_\x}{v^S_\x} + \frac{(\gamma-1)p^*}{R}\right| \le C\deltas,\quad \forall \x,
\]
which proves \eqref{theta-s}.

\noindent$\bullet$ {\bf Proof of $|( v^S_{\xi\xi}, u^S_{\xi\xi},\theta^S_{\x\x})|\leq C\delta_S |( v^S_\xi, u^S_\xi, \theta_\x^S)|, \quad \forall\xi\in\bbr$:}

First, it holds from \eqref{VS} that 
\[
\mu \left(\frac{u^S_\xi}{v^S}\right)_\x = \s^2 v^S_\xi + p^S_\x ,
\]
and so,
\beq\label{unisec}
\mu  \frac{u^S_{\xi\xi}}{v^S}  = \mu \frac{u^S_{\xi} v^S_\xi}{(v^S)^2} + \s^2 v^S_\xi + \frac{R\theta^S_\x}{v^S} - \frac{p^S v^S_\x}{v^S}.
\eeq
Then, using \eqref{theta-s} and \eqref{sm1}, we have
 \begin{align*}
\begin{aligned}
\mu \left| \frac{u^S_{\xi\xi}}{v^S} \right| &\le \mu \frac{|u^S_{\xi}| |v^S_\xi| }{(v^S)^2}  +\underbrace{ \left|  (\s^*)^2 -\frac{(\gamma-1)p^*}{v^*} - \frac{p^*}{v^*}  \right|}_{=0} |v^S_\xi | + |\s^2-(\s^*)^2| |v^S_\xi | \\
&\quad +    \frac{|R\theta^S_\x - (\gamma-1)p^* v^S_\x|}{v^S} + (\gamma-1)p^* \left| \frac{1}{v^S}  -\frac{1}{v^*}  \right| |v^S_\xi | + \left| \frac{p^*}{v^*}  - \frac{p^S}{v^S} \right| |v^S_\xi | \\
&\le C\deltas  |v^S_\xi |,
\end{aligned}
\end{align*}
which yields
\beq\label{uvsec}
|u^S_{\xi\xi}| \le C\deltas  |v^S_\xi |.
\eeq
To estimate $|\theta^S_{\xi\xi}|$, we first find from $\eqref{VS}$ that $(v^S, u^S, \theta^S)$ satisfies (as a solution of \eqref{NS-0})
\beq\label{stheta-eq}
-\frac{R\sigma}{\gamma-1}(\theta^S)_\xi + p^S(u^S)_\xi=\kappa\Big(\frac{(\theta^S)_\xi}{v^S}\Big)_\xi+\mu \frac{((u^S)_\xi)^2}{v^S}.
\eeq
This together with \eqref{unisec} gives
\[
\kappa  \frac{\theta^S_{\xi\xi}}{v^S}  = \kappa \frac{\theta^S_{\xi} v^S_\xi}{(v^S)^2} -\mu \frac{(u^S_\xi)^2}{v^S} -\frac{\sigma}{\gamma-1} \left( \mu u^S_{\xi\xi} -\mu  \frac{u^S_{\xi} v^S_\xi}{v^S} 
-  \s^2 v^S v^S_\xi + p^S v^S_\x \right)+  p^S u^S_\x.
\]
Then, using \eqref{uvsec}, we have
\[
\left|\kappa  \frac{\theta^S_{\xi\xi}}{v^S}  \right| \le C\deltas  |v^S_\xi | + \frac{\sigma}{\gamma-1} \left|- \s^2 v^S v^S_\xi + p^S v^S_\x  - \frac{\gamma-1}{\s} p^S u^S_\x \right|.
\]
In addition, since $u^S_\xi=-\sigma v^S_\xi$ and \eqref{sm1} imply
\[
\left|- \s^2 v^S v^S_\xi + p^S v^S_\x  - \frac{\gamma-1}{\s} p^S u^S_\x \right| \le \underbrace{|-\s^* v^* + p^* + (\gamma-1)p^*|}_{=0} | v^S_\xi | + C\deltas  |v^S_\xi |,
\]
we have
\[
|\theta^S_{\xi\xi}| \le C\deltas  |v^S_\xi |.
\]
Finally, using $u^S_\xi=-\sigma v^S_\xi$ and \eqref{shock-vu}-\eqref{theta-s}, we complete the desired estimates.
\end{proof}

\subsection{Weighted Poincare inequality}

The method of $a$-contraction with shift in the viscous cases relies on the following Poincar\'e type inequality (see \cite[Lemma 2.9]{Kang-V-NS17}).  
\begin{lemma}\label{lem-poin}
For any $f:[0,1]\to\bbr$ satisfying $\int_0^1 y(1-y)|f'|^2 dy<\infty$, 
\beq\label{poincare}
\int_0^1\Big|f-\int_0^1 f dy \Big|^2 dy\le \frac{1}{2}\int_0^1 y(1-y)|f'|^2 dy.
\eeq
\end{lemma}


\section{Proof of Theorem \ref{thm:main}}\label{sec-thm}
\setcounter{equation}{0}

\subsection{Local in time estimates on the solution} 
For simplicity, we rewrite the system \eqref{NS-0} into the following system, based on the change of variable associated to the speed of propagation of the shock $(t,x)\mapsto (t, \xi=x-\s t)$: 
 \begin{align}
\begin{aligned}\label{NS-1}
\left\{ \begin{array}{ll}
        v_t-\sigma v_\xi- u_\xi =0,\\
       u_t -\sigma u_\xi +p(v,\theta)_\xi=  (\mu\frac{u_\xi}{v})_\xi,\\
       \frac{R}{\gamma-1}\theta_t-\frac{R\sigma}{\gamma-1}\theta_\x+p(v,\theta) u_\xi=(\kappa\frac{\theta_\x}{v})_\x+\mu\frac{u_\x^2}{v}. \\
        \end{array} \right.
\end{aligned}
\end{align}
It follows from \eqref{ARW} that the approximate rarefaction wave $(v, u, \theta)(t,\xi)=(v^R, u^R, \theta^R)(t,\xi+\sigma t) $ verifies 
\begin{equation}  \label{rarexi}
  \begin{cases}
    \displaystyle v_t -\s  v_\xi-u_\x = 0, \\
    \displaystyle u_t -\s  u_\xi+ p(v, \theta)_\x = 0,\\
    \di \frac{R}{\gamma-1}\theta_t-\frac{R\s}{\gamma-1}\theta_\x+p(v, \theta) u_\x=0. \\
  \end{cases}
\end{equation}
and from \eqref{vc-equ} that
 the viscous contact wave  $(v, u, \theta)(t,\xi)=\big(\vc, u^{C},
\theta^{C}\big)(t,\xi+\sigma t)$ satisfies the system
\begin{eqnarray}\label{vcex}
\begin{cases}
  v_t-\s v_\x-u_{\x} = 0, \cr
    u_t-\s u_\x+ p_\x
    =\mu\big( \frac{u_{\x}}{ v} \big)_\x+Q^C_1, \cr
        \frac{R}{\gamma-1}\theta_t-\frac{R\s}{\gamma-1}\theta_\x+  p u_\x
    =\kappa\big(\frac{\theta_{\x}}{v} \big)_\x+\mu\frac{(u_{\x})^2}{v}+Q^C_2,
\end{cases}
\end{eqnarray}
with the error terms $Q^C_i~(i=1,2)$ defined in \eqref{QC12}.
We will consider stability of the solution to \eqref{NS-1} around the superposition wave of the approximate rarefaction wave, the viscous contact wave and  the viscous shock wave shifted by $\mb X(t)$ (to be defined in \eqref{X(t)}) : 
\beq\label{shwave}
\begin{array}{l}
\di (\bar v, \bar u,\bar \theta) (t,\xi):= \Big(v^R(t,\xi+\sigma t)+v^C(\xi+\sigma t)+ v^S(\xi -\mb X(t))-v_*-v^*,\\[4mm]
\di \qquad\qquad\qquad\qquad u^R(t,\xi+\sigma t)+u^C(t,\xi+\sigma t)+u^S(\xi-\mb X(t))-u_*-u^*,\\[4mm]
\di \qquad\qquad\qquad\qquad  \theta^R(t,\x+\s t)+ \theta^C(\xi+\sigma t)+ \theta^S(\xi-\mb X(t))-\theta_*-\theta^* \Big).
\end{array}
\eeq
Then the superposition wave $(\bar v, \bar u,\bar \theta) (t,\xi)$ satisfies the system
\beq\label{bar-system}
\left\{
\begin{array}{ll}
\di \bar v_t-\s \bar v_\xi+\dot{\mb X}(t)(v^S)^{-\mb X}_\xi-\bar u_\x=0,\\[3mm]
\di \bar u_t-\s \bar u_\xi+\dot{\mb X}(t)(u^S)^{-\mb X}_\xi+\bar p_\x=\mu\left(\frac{\bar u_\x}{\bar v}\right)_\xi+Q_1,\\[3mm]
\di \frac{R}{\gamma-1}\bar \theta_t-\frac{R\s}{\gamma-1}\bar \theta_\xi+\frac{R}{\gamma-1}\dot{\mb X}(t)(\theta^S)^{-\mb X}_\xi+\bar p\bar u_\x=\kappa\left(\frac{\bar \theta_\x}{\bar v}\right)_\xi+\mu\frac{\bar u_\xi^2}{\bar v}+Q_2,
\end{array}
\right.
\eeq
where $\bar p=\frac{R\bar\theta}{\bar v}$ and the error terms
\beq\label{Q12}
Q_i:=Q^I_i+Q^R_i+Q^C_i, \quad i=1,2,
\eeq
with the wave interactions terms
\beq\label{QI1}
Q_1^I:= \left(\bar p-p^R-p^C-(p^S)^{-\mb X}\right)_\xi-\mu\left(\frac{\bar u_\x}{\bar v}-\frac{u^R_\x}{v^R}-\frac{u^C_\x}{v^C}-\frac{(u^S)^{-\mb X}_\x}{(v^S)^{-\mb X}}\right)_\x,
\eeq
\beq\label{QI2}
\begin{array}{ll}
\di 
Q_2^I:= \left(\bar p\bar u_\x-p^Ru^R_\xi-p^Cu^C_\xi-(p^S)^{-\mb X}(u^S)^{-\mb X}_\x\right)-\kappa\left(\frac{\bar \theta_\x}{\bar v}-\frac{\theta^R_\x}{v^R}-\frac{\theta^C_\x}{v^C}-\frac{(\theta^S)^{-\mb X}_\x}{(v^S)^{-\mb X}}\right)_\x\\
\di \qquad\quad -\mu\left(\frac{\bar u^2_\x}{\bar v}-\frac{(u^R_\x)^2}{v^R}-\frac{(u^C_\x)^2}{v^C}-\frac{((u^S)^{-\mb X}_\x)^2}{(v^S)^{-\mb X}}\right),
\end{array}
\eeq
and the error terms due to the inviscid rarefaction wave
\beq\label{QR}
Q^R_1:=-\mu\left(\frac{u^R_\x}{v^R}\right)_\x,\qquad
Q^R_2:=-\kappa\left(\frac{\theta^R_\x}{v^R}\right)_\x -\mu \frac{(u^R_\x)^2}{v^R},
\eeq
and the error terms $Q_1^C, Q_2^C$ due to the viscous contact wave are given in \eqref{QC12}. 

For any initial $H^1$ perturbation of the superposition of waves \eqref{shwave}, there exists a global strong solution to \eqref{NS-1} (see for instance \cite{Kanel-79}). We will use a standard argument of continuation process to show the global in time control of this perturbation. 
For that, we first recall local in time estimates for strong solutions to \eqref{NS} (and so also for \eqref{NS-1}). They can be found in  \cite{Nash} in the general setting.
\begin{proposition} \label{prop:soln}
Let $\underline v$, $\underline u$ and $\underline{\theta}$ be smooth monotone functions such that
\beq\label{sm}
(\underline v(x) , \underline u(x), \underline \theta(x)) =(v_\pm, u_\pm, \theta_\pm), \quad\mbox{for } \pm x \ge 1.
\eeq
For any constants $M_0, M_1,  \underline \kappa_0,  \overline \kappa_0, \underline\kappa_1, \overline\kappa_1$ with $M_1>M_0>0$ and $ \overline \kappa_1>\overline \kappa_0>  \underline \kappa_0>\underline\kappa_1>0$, there exists a constant $T_0>0$ such that if 
\begin{align*}
\begin{aligned}
&\|(v_0-\underline v, u_0 -\underline u, \theta_0-\underline \theta)\|_{H^1(\bbr)}  \le M_0,\\
&0< \underline \kappa_0 \leq v_0(x), \theta_0(x)\leq  \overline \kappa_0, \qquad  \forall x\in \bbr, \\
\end{aligned}
\end{align*}
then \eqref{NS-1} has a unique solution $(v,u)$ on $[0,T_0]$ such that 
\begin{align*}
\begin{aligned}
&v -\underline v \in C([0,T_0];H^1(\bbr)), \\
&u -\underline u, \theta-\underline{\theta}\in C([0,T_0];H^1(\bbr)) \cap  L^2(0,T_0;H^2(\bbr)).
\end{aligned}
\end{align*}
and
\[
\|(v-\underline v, u-\underline u, \theta-\underline\theta)\|_{L^\infty(0,T_0;H^1(\bbr))} \le M_1.
\]
Moreover:
\beq\label{bddab}
\underline  \kappa_1 \le v(t,x), \theta(t,x) \le \overline  \kappa_1,\qquad \forall (t,x)\in [0,T_0]\times \bbr.
\eeq
\end{proposition}

\subsection{Construction of shift} For the continuation argument, the main tool is the a priori estimates of  Proposition \ref{prop2}. These estimates depend on the shift function, and for this reason, we are giving its definition right now. The definition depends on the weight function  $a:\bbr\to\bbr$ defined in \eqref{weight}. For now, we will only use the fact that $\|a\|_{C^1(\bbr)}\leq 2$. 
We then define the shift  $\mb X$ as a solution to the ODE:
\begin{equation}\label{X(t)}
\left\{
\begin{array}{ll}
\di \dot{\mb{X}}(t)=-\frac{M}{\delta_S}\Big[\int_{\mathbb{R}}a(\xi-\mb{X})\Big[u^S_\xi(\xi-\mb{X})(u-\bar u)+\frac{R\theta^S_\xi(\xi-\mb{X})}{(\gamma-1)\bar\theta}(\theta-\bar\theta)\\[4mm]
\di \qquad\qquad\qquad \qquad\qquad \qquad +\frac{\bar p v^S_\x(\xi-\mb{X})}{\bar v}(v-\bar v)\Big]d\xi,\\[4mm]
\di \mb X(0)=0,
\end{array}
\right.
\end{equation}
where $M$ is the specific constant chosen as $
M:=\frac{\gamma(\gamma+1)p^*}{2(v^*)^2(\sigma^*)^3} \left(1+\frac{2\kappa (\gamma-1)^2}{\mu R\gamma}\right)$, which will be used in the proof of Lemma \ref{lem-sharp}.

The following lemma ensures that \eqref{X(t)} has a unique absolutely continuous solution defined on any interval in time $[0,T]$ for which   \eqref{bddab} is verified.
\begin{lemma}\label{lem:xex}
For any $c_1,c_2>0$, there exists a constant $C>0$ such that the following is true.  For any $T>0$, and any   function $v\in L^\infty ((0,T)\times \R)$ 
verifying
\beq\label{odes}
c_1 \le v(t,x), \theta(t,x)\le c_2,\qquad \forall (t,x)\in [0,T]\times \bbr,
\eeq
 the ODE \eqref{X(t)} has a unique absolutely continuous solution $\mb X$ on $[0,T]$. Moreover, 
\beq\label{roughx}
|{\mb X}(t)| \le Ct,\quad \forall t\le T.
\eeq
\end{lemma}
\begin{proof}
We will use the following lemma as a simple adaptation of the well-known Cauchy-Lipschitz theorem.  

 \begin{lemma} \cite[Lemma A.1]{CKKV} \label{lem_ckkv}
 Let $p>1$ and $T>0$. Suppose that a function 
 $F:[0,T]\times\bbr\rightarrow\bbr$  satisfies 
 $$\sup_{x\in\bbr }|F(t,x)|\leq f(t)\ \ \mbox{and}\ \ 
\sup_{x,y\in\bbr,x\neq y }\Big|\frac{F(t,x)-F(t,y)}{x-y}\Big|\leq g(t)
\quad \mbox{for } t\in[0,T] $$ for some functions $f \in L^1(0,T)$ and $\, g\in L^p(0,T)$. Then for any $x_0\in\bbr$, there exists a unique absolutely continuous function $\mb{X}:[0,T]\rightarrow \bbr$ satisfying
\beq\label{ode_eq}\left\{ \begin{array}{ll}
        \dot{\mb{X}}(t) = F(t,\mb{X}(t))\quad\mbox{for \textit{a.e.} }t\in[0,T],\\
       \mb{X}(0)=x_0.\end{array} \right.\eeq
 \end{lemma}

To apply the above lemma, let $F(t,\mb{X})$ denote the right-hand side of the ODE \eqref{X(t)}. \\
Then the sufficient conditions of the above lemma are verified thanks to the facts that $\|a\|_{C^1(\bbr)}\leq 2$, $\|(v^S,u^S,\theta^S)\|_{C^2(\bbr)}\leq C$, and $\|( v^S_\xi, u^S_\x,\theta^S_\x)\|_{L^1} \le C\delta_S$.
Indeed, using \eqref{odes}, we find that for some constant $C>0$,
\beq\label{f1t}
\begin{array}{l}
\di \sup_{\mb{X}\in\bbr}|F(t,\mb{X})| \le \frac{C}{\deltas}  \|\big(v-\bar v, u-\bar u, \theta-\bar\theta) \|_{L^\infty(\bbr)} \int_\bbr |( v^S_\xi, u^S_\x,\theta^S_\x)| d\xi\le C,
\end{array}
\eeq
and
\begin{align*}
\begin{aligned}
\sup_{\mb{X}\in \bbr} |\partial_{\mb{X}}F(t,\mb{X})| &\le  \frac{C}{\deltas} \|a\|_{C^1}\|\big(v-\bar v, u-\bar u, \theta-\bar\theta) \|_{L^\infty(\bbr)} \int_\bbr |( v^S_\xi, u^S_\x,\theta^S_\x)|d\xi\\ & \le C.
\end{aligned}
\end{align*}
Especially, since $|\dot{\mb X}(t)| \le C$ by \eqref{f1t}, we have \eqref{roughx}.
\end{proof}

\subsection{A priori estimates}


We now state the key step for the proof of Theorem \ref{thm:main}. 

\begin{proposition} \label{prop2}
For a given point $(v_+,u_+,\theta_+)\in\bbr_+\times\bbr\times\bbr_+$, there exist positive constants $C_0, \delta_0,\eps_1$ such that the following holds.\\
Suppose that $(v,u,\theta)$ is the solution to \eqref{NS-1} on $[0,T]$ for some $T>0$,  and $(\bar v,\bar u,\bar\theta)$ is the superposition wave defined in \eqref{shwave} with the shift $\mb X$ only performed in the viscous shock and being the absolutely continuous solution to \eqref{X(t)} with weight function $a$ defined in \eqref{weight}. Assume that both the rarefaction and shock waves strength satisfy $\deltar, \delta_C, \deltas<\delta_0$ and that 
\begin{align*}
\begin{aligned}
&v -\bar v \in C([0,T];H^1(\bbr)), \\
&(u -\bar u,\theta-\bar\theta)\in C([0,T];H^1(\bbr)) \cap  L^2(0,T;H^2(\bbr)),
\end{aligned}
\end{align*}
and 
\beq\label{apri-ass}
\|(v-\bar v, u-\bar u, \theta-\bar\theta)\|_{L^\infty(0,T;H^1(\bbr))} \le \eps_1.
\eeq
Then,  for all $t\le T$,
\begin{align}
\begin{aligned}\label{finest}
&\sup_{t\in[0,T]}\Big[\|(v-\bar v, u-\bar u, \theta-\bar \theta)\|_{H^1(\bbr)}\Big] +\sqrt{\deltas\int_0^t|\dot{\mb{X}}(\tau)|^2 d\tau} \\
&\quad+\sqrt{\int_0^t ( \mathcal{G}^R(U) + \mathcal{G}^S(U))d\tau}\\
&\quad+\sqrt{\int_0^t\|(v-\bar v)_x\|^2_{L^2(\bbr)} d\tau+\int_0^t\|(u-\bar u, \theta-\bar \theta)_x\|_{H^1(\bbr)}^2 d\tau}\\
&\le C_0 \|\big(v_0-\bar v(0,\cdot),u_0 -\bar u(0,\cdot),\theta-\bar\theta(0,\cdot)\big)\|_{H^1(\bbr)}  + C_0\delta_0^{1/4} ,
\end{aligned}
\end{align}
where 
\begin{align}
\begin{aligned}\label{maingood}
& \mathcal{G}^R(U):= \int_\bbr|v^R_\x| |(v-\bar v, \theta-\bar \theta)|^2 d\x ,\\
&\mathcal{G}^S(U):=\int_\bbr|(v^S)^{-\mb X}_\x| |(v-\bar v, u-\bar u, \theta-\bar \theta)|^2 d\x.
\end{aligned}
\end{align}
In addition, by \eqref{X(t)},
\beq\label{xprop}
|\dot{\mb{X}}(t)|\leq C_0\|(v-\bar v,u-\bar u, \theta-\bar \theta)(t,\cdot)\|_{L^\infty(\bbr)},\qquad t\le T.
\eeq
\end{proposition}

We postpone the proof of this key proposition to Sections \ref{sec-acontraction} and \ref{sec-vu}. We are proving in the rest of this section  how Proposition \ref{prop2} implies Theorem \ref{thm:main}.

\subsection{Conclusion}
Based on Propositions \ref{prop:soln} and \ref{prop2}, we use the continuation argument to prove \eqref{ext-main} for the global-in-time existence of perturbations. We can also use Proposition \ref{prop2} to prove  \eqref{con} for the long-time behavior. Those proofs are typical and use the same arguments as in the previous paper \cite{KVW-2022}. 
Therefore, we omit those details, and complete the proof of Theorem \ref{thm:main}.

Hence, the remaining part of this paper is dedicated to the proof of Proposition \ref{prop2}.
 
\noindent$\bullet$ {\bf Notations:} In what follows, we use the following notations  for notational simplicity. \\
1. $C$ denotes a positive $O(1)$-constant which may change from line to line, but which is independent of the small constants $\delta_0, \eps_1, \deltas,\deltar$, $\lambda$ (to appear in \eqref{weight}) and the time $T$.\\
2. For any function $f : \bbr_+\times \bbr\to \bbr$ and any time-dependent shift $\mb X(t)$, 
\[
f^{\pm \mb X}(t, \xi):=f(t,\xi\pm \mb X(t)).
\]
For simplicity, we also omit the arguments of the waves without confusion: for example, 
\begin{align*}
\begin{aligned}
&v^R:=v^R(t,\xi+\s t),\quad  (v^R)^{\mb X}:=v^R(t,\xi+\s t + \mb X(t)),\\
&v^C:=v^C(t,\xi+\s t),\quad  (v^C)^{\mb X}:=v^C(t,\xi+\s t + \mb X(t)),\\
&\bar v^{\mb X}:=v^R(t,\xi+\s t + \mb X(t))+v^C(\frac{\xi+\s t+ \mb X(t)}{\sqrt{1+t}})+v^S(\xi)-v_*-v^*.
\end{aligned}
\end{align*}

\section{Relative entropy estimates}\label{sec-acontraction}
\setcounter{equation}{0}
This section is dedicated to the proof of the following lemma.
\begin{lemma}\label{lem-zvh}
Under the hypotheses of Proposition \ref{prop2}, there exists $C>0$ (independent of $\delta_0, \eps_1, T$) such that for all $t\in [0,T]$,
\begin{align}
\begin{aligned}\label{esthv}
&\sup_{t\in[0,T]}\Big[\|(v-\bar v, u-\bar u, \theta-\bar \theta)(t,\cdot)\|^2_{L^2(\bbr)}\Big] +\deltas\int_0^t|\dot{\mb{X}}(\tau)|^2 d\tau \\
&\quad+\int_0^t ( \mathcal{G}^R(U) + \mathcal{G}^S(U)) d\tau +\int_0^t\|(u-\bar u, \theta-\bar \theta)_x\|^2_{L^2(\bbr)} d\tau\\
&\le C \|\big(v_0-\bar v(0,\cdot),u_0 -\bar u(0,\cdot),\theta_0-\bar\theta(0,\cdot)\big)\|^2_{L^2(\bbr)} +C\delta_0\int_0^t\|(v-\bar v)_x\|^2_{L^2(\bbr)} d\tau+ C\delta_0^{1/2},
\end{aligned}
\end{align}
where the good terms $ \mathcal{G}^R(U),  \mathcal{G}^S(U)$ are as in \eqref{maingood}.
\end{lemma}

\subsection{Wave interaction estimates}
We first present useful estimates for the wave interaction terms such as $Q_i^I$ in \eqref{QI1}. 
Notice that the a priori assumption \eqref{apri-ass} with the Sobolev embedding implies
\beq\label{smp1}
\|(v-\bar v, u-\bar u, \theta-\bar \theta)\|_{L^\infty((0,T)\times\bbr)}\le C\eps_1, \quad\mbox{and so}\quad v, u, \theta \in L^\infty((0,T)\times\bbr).
\eeq
Then, the ODE \eqref{X(t)} together with Lemma \ref{lemma1.3} yields that
\beq\label{dxbound}
\begin{aligned}
|\dot{\mb X}(t)| &\le \frac{C}{\deltas}  \|(v-\bar v, u-\bar u, \theta-\bar \theta)(t,\cdot) \|_{L^\infty(\bbr)} \int_\bbr ( v^S)^{-\mb X}_\xi d\xi \\
&\le C \|(v-\bar v, u-\bar u, \theta-\bar \theta)(t,\cdot) \|_{L^\infty(\bbr)}.
\end{aligned}
\eeq
This especially proves \eqref{xprop}, and 
will be used to get the wave interaction estimates in Lemma \ref{lemma2.2}.

\begin{lemma}\label{lemma2.2}
Let $\mb X$ be the shift defined by \eqref{X(t)} and $Q_i^I(~i=1,2)$ is defined in \eqref{QI1} and \eqref{QI2}. Under the same hypotheses as in Proposition \ref{prop2}, the following holds: for $i=1,2,$ and $\forall t\le T,$ $\exists$ a  constant $C>0$ independent of $T, \delta_R, \deltas$ and $\delta_C$ such that
\beq\label{wave-interactions}
\begin{aligned}
&\|Q_i^I\|_{L^2(\bbr)} \le C\deltas (\deltar+\delta_C) e^{-C \deltas t}+C\deltar\delta_Ce^{-Ct},\\[3mm]
& \| |(v^S)^{-\mb X}_\xi| |\big(v^R -v_*,\theta^R -\theta_*\big) |\|_{L^2(\bbr)}+\| |(v^S)^{-\mb X}_\xi| |\big(v^C -v^*,\theta^C -\theta^*\big) |\|_{L^2(\bbr)}\\
&\leq C\deltas^{3/2} (\deltar+\delta_C) e^{-C \deltas t} .
\end{aligned}
\eeq
\end{lemma}
\begin{proof} For brevity, we only estimate $\|Q_1^I\|_{L^2(\bbr)}$, since the proof for $\|Q_2^I\|_{L^2(\bbr)}$ is almost the same. 
Recall
\[
Q_1^I:= \left(\bar p-p^R-p^C-(p^S)^{-\mb X}\right)_\xi-\mu\left(\frac{\bar u_\x}{\bar v}-\frac{u^R_\x}{v^R}-\frac{u^C_\x}{v^C}-\frac{(u^S)^{-\mb X}_\x}{(v^S)^{-\mb X}}\right)_\x.
\]
Since $\bar p = \frac{R\bar\theta}{\bar v}$, $\bar{\theta}_\xi = \theta^R_\xi+\theta^C_\xi+\theta^S_\xi$ and $\bar{v}_\xi = v^R_\xi+v^C_\xi+v^S_\xi$, the first term of $Q_1^I$ can be written as
$$\begin{array}{ll}
\di  \big(\bar p-p^R-p^C-(p^S)^{-\mb X}\big)_\xi \\[3mm]
\di =R\theta^R_\xi(\frac {1}{\bar v}-\frac {1}{v^R})+R\theta^C_\xi(\frac {1}{\bar v}-\frac {1}{v^C})+R(\theta^S)^{-\mb X}_\xi(\frac {1}{\bar v}-\frac {1}{(v^S)^{-\mb X}})\\[3mm]
\di\quad - v^R_\xi(\frac {\bar p}{\bar v}-\frac {p^R}{v^R}) -v^C_\xi(\frac {\bar p}{\bar v}-\frac {p^C}{v^C}) - (v^S)^{-\mb X}_\xi(\frac {\bar p}{\bar v}-\frac {(p^S)^{-\mb X}}{(v^S)^{-\mb X}}).
\end{array}
$$
Thus, we have
\beq\label{pcomp}
 \big|\big(\bar p-p^R-p^C-(p^S)^{-\mb X}\big)_\xi \big| \le C(R_1+R_2+R_3),
\eeq
where
\[
\begin{array}{ll}
\di  R_1:= |(v^R_\xi,\theta^R_\xi)| |(v^C-v_*, \theta^C-\theta_*, (v^S)^{-\mb X}-v^*, (\theta^S)^{-\mb X}-\theta^*)|, \\[3mm]
\di R_2:= |(v^C_\xi,\theta^C_\xi)| |(v^R-v_*, \theta^R-\theta_*, (v^S)^{-\mb X}-v^*, (\theta^S)^{-\mb X}-\theta^*) | ,\\[3mm]
\di R_3:= |((v^S)^{-\mb X}_\xi,(\theta^S)^{-\mb X}_\xi)| |(v^R-v_*, \theta^R-\theta_*,v^C-v^*, \theta^C-\theta^*)| .
\end{array}
\]
First, for fixed $t\in[0,T]$ and $\l_{1*}:=\l_1(v_*, \theta_*)<0$, we have
\beq\label{Q1I-1}
\begin{aligned}
 \|R_1 \|_{L^2(\bbr)} &\le  \bigg[\bigg(\int_{\big\{\xi+\s t \leq \frac{\l_{1*}}{2}(1+t)\big\}}+\int_{\big\{\xi+\s t > \frac{\l_{1*}}{2}(1+t)\big\}}\bigg) \big|(v^R_\xi,\theta^R_\xi)\big|^2 \\[4mm]
& \qquad \cdot\big|(v^C-v_*, \theta^C-\theta_*, (v^S)^{-\mb X}-v^*, (\theta^S)^{-\mb X}-\theta^*) \big|^2 d\xi\bigg]^{\frac12}\\
& =: C\sqrt{R_{11}+R_{12}}
\end{aligned}
\eeq
By \eqref{smp1}
and
\eqref{dxbound}, 
it holds that
$$
|\dot{\mb{X}}(t)|\leq C\eps_1,\qquad 0\le t\le T,
$$
which together with $\mb X(0)=0$ yields
$$
|\mb{X}(t)|\leq C\eps_1 t,\qquad 0\le t\le T.
$$
Let us take $\eps_1$ so small such that the above bound is less than $ \frac{\s t}{4}$, that is,
\beq\label{X(t)-bound}
|\mb{X}(t)|\leq C\eps_1t < \frac{\s t}{4}.
\eeq
Since 
\[
\xi+\s t \leq  \frac{\l_{1*}}{2}(1+t)\quad\Rightarrow\quad\xi- \mb{X}(t)\leq \frac{\l_{1*}}{2}(1+t)-\s t-\mb{X}(t)\leq  \frac{\l_{1*}}{2}(1+t)-\frac 34\s t < -\frac 34\s t <0,
\] 
and so $|\xi- \mb{X}(t)|\ge \frac 34\s t $, it follows from Lemma \ref{lemma1.3} and Lemma \ref{Lemma 2.2.} that for all $\xi$ with $\xi+\s t \leq  \frac{\l_{1*}}{2}(1+t)<0$,
\beq\label{shock-pe}
\begin{aligned}
\big|\big((v^S)^{-\mb X}-v^*, (\theta^S)^{-\mb X}-\theta^*\big) \big| &\le C\deltas e^{-C\deltas |\xi -\mb X(t)|} \\
&\le C\deltas e^{-\frac{C\deltas |\xi -\mb X(t)|}{2}} e^{-C\deltas t},
\end{aligned}
\eeq
and
\beq\label{vc-pe1}
\big|(v^C-v_*, \theta^C-\theta_*)|+\big|(v^C_\x, \theta^C_\x)\big| \le C\delta_C e^{-\frac{C|\xi+\s t|^2}{1+t}}\le C\delta_C e^{-\frac{C|\xi+\s t|^2}{2(1+t)}}e^{-Ct}.
\eeq
Thus,
\begin{align*}
R_{11} &\le C\Big[(\deltas)^2 e^{-2C\deltas t}+(\delta_C)^2 e^{-2Ct} \Big]  \int_\bbr \big|(v^R_\xi,\theta^R_\xi)\big|^2 d\xi,
\end{align*}
which together with Lemma \ref{lemma1.2} (1) yields
\[
R_{11} \le C\delta_R^2 \Big[(\deltas)^2 e^{-2C\deltas t}+(\delta_C)^2 e^{-2Ct} \Big].
\]
On the other hand, since 
$$
\xi+\s t > \frac{\l_{1*}}{2}(1+t)  \quad\Rightarrow\quad \xi+\s t-\l_{1*}(1+t)> -\frac{\l_{1*}}{2}(1+t)=\frac{|\l_{1*}|}{2}(1+t)>0,
$$
it follows from Lemma \ref{lemma1.2} (3) that  for all $\xi$ with $\xi+\s t > \frac{\l_{1*}}{2}(1+t) $,
\beq\label{rarefaction-pe}
|(v^R -v_*, \theta^R -\theta_*)| + |(v_\xi^R,\theta_\x^R)| \le C\deltar e^{-2 |\xi +\s t-\lambda_{1*}(1+t) |}.
\eeq
In addition, for $\xi+\s t > \frac{\l_{1*}}{2}(1+t) $, 
\beq\label{rarefaction-pe2}
e^{-2 |\xi +\s t-\lambda_{1*}(1+t) |} = e^{-2 (\xi +\s t-\frac{\lambda_{1*}}{2}(1+t) )} e^{\lambda_{1*}(1+t)} = e^{-2 |\xi +\s t-\frac{\lambda_{1*}}{2}(1+t) |} e^{-|\lambda_{1*}|(1+t)}.
\eeq
This together with  Lemma \ref{lemma1.3} and Lemma \ref{Lemma 2.2.} implies
\begin{align*}
R_{12} &\le C\deltar^2 \big(\deltas^2+\delta_C^2 \big) e^{-|\lambda_{1*}|(1+t)} \int_{\big\{\xi+\s t > \frac{\l_{1*}}{2}(1+t)\big\}}   e^{-2 |\xi +\s t-\frac{\lambda_{1*}}{2}(1+t) |} d\xi \\
&\le  C\deltar^2 \big(\deltas^2+\delta_C^2 \big) e^{-C t} 
\end{align*}
Thus,
\[
 \|R_1 \|_{L^2(\bbr)} \le C \deltar \deltas e^{-C \deltas t}+C\deltar\big(\deltas+\delta_C \big)e^{-Ct}.
\]
To estimate $ \|R_2\|_{L^2(\bbr)}$, consider
\beq\label{Q1I-2}
\begin{aligned}
\|R_2 \|_{L^2(\bbr)}
 &\le C \bigg[\bigg(\int_{\big\{\xi+\s t \leq \frac{\l_{1*}}{2}(1+t)\big\}}+\int_{\big\{\xi+\s t > \frac{\l_{1*}}{2}(1+t)\big\}}\bigg)   \big|(v^C_\xi,\theta^C_\xi)\big|^2 \big|(v^R-v_*, \theta^R-\theta_*) \big|^2 d\x\\[4mm]
&\quad \quad\quad + \bigg(\int_{\big\{\xi\leq -\frac{\s}{2} t\big\}} + \int_{\big\{\xi> -\frac{\s}{2} t\big\}}\bigg) \big|(v^C_\xi,\theta^C_\xi)\big|^2 \big|( (v^S)^{-\mb X}-v^*, (\theta^S)^{-\mb X}-\theta^*) \big|^2 d\xi\bigg]^{\frac12}\\
& =: C\sqrt{R_{21}+R_{22}+R_{23}+R_{24}}.
\end{aligned}
\eeq
First, using \eqref{vc-pe1}, we have
\[
R_{21} \le C \delta_R^2 \delta_C^2 e^{-Ct} \int_{\big\{\xi+\s t \leq \frac{\l_{1*}}{2}(1+t)\big\}}  e^{-\frac{C|\xi+\s t|^2}{2(1+t)}} d\xi \le C \delta_R^2 \delta_C^2 e^{-Ct} \int_{\bbr}  e^{-C|\xi+\s t|} d\xi \le C \delta_R^2 \delta_C^2 e^{-Ct}.
\]
Using the same estimates on $R_{12}$ with \eqref{rarefaction-pe}, 
\[
R_{22} \le C\deltar^2 \delta_C^2 e^{-Ct}.
\]
For $R_{23},$ since
\begin{align*}
\forall \xi\leq-\frac{\s t}{2}<0,\quad & \xi-\mb X(t) \leq -\frac{\s t}{2} +C\eps_1 t <-\frac{\s t}{4}<0 \quad\mbox{and}\  \mbox{then}  \\
&| \xi-\mb X(t)| \ge  \frac{\s t}{4},
\end{align*}
it holds from Lemma \ref{lemma1.3} that $\forall \xi\leq-\frac{\s t}{2},$
\beq\label{shock-pe-1}
\begin{aligned}
\big|\big((v^S)^{-\mb X}-v^*, (\theta^S)^{-\mb X}-\theta^*\big) \big| &\le C\deltas e^{-C\deltas |\xi -\mb X(t)|} \\
&\le C\deltas e^{-\frac{C\deltas |\xi -\mb X(t)|}{2}} e^{-C\deltas t},\\[3mm]
|((v^S)^{-\mb X}_\x, (\theta^S)^{-\mb X}_\x)| &\le C\deltas^2 e^{-C\deltas |\xi -\mb X(t)|} \\
&\le C\deltas^2 e^{-\frac{C\deltas |\xi -\mb X(t)|}{2}} e^{-C\deltas t}.
\end{aligned}
\eeq
This and Lemma \ref{Lemma 2.2.} imply
\[
R_{23} \le C \deltas^2 e^{-C\deltas t} \big\|(v^C_\x, \theta^C_\x)\big\|_{L^2(\bbr)}^2 \le C \deltas^2\delta_C^2 e^{-C\deltas t}.
\]
On the other hand, since
\[
\forall \xi> -\frac{\s t}{2},\quad \xi +\s t> \frac{\s t}{2} > 0, 
\]
it holds from Lemma \ref{Lemma 2.2.} that $\forall \xi>-\frac{\s t}{2},$
\beq\label{vc-pe-2}
\big|(v^C-v_*, \theta^C-\theta_*)|+ |(v^C_\x, \theta^C_\x)\big|=O(1)\delta_C e^{-\frac{C|\xi+\s t|^2}{1+t}}=O(1)\delta_C e^{-\frac{C|\xi+\s t|^2}{2(1+t)}}e^{-Ct}.
\eeq
Then,
\[
R_{24} \le C \delta_S^2 \delta_C^2 e^{-Ct} \int_{\big\{\xi +\s t> \frac{\s t}{2} \big\}}  e^{-\frac{C|\xi+\s t|^2}{2(1+t)}} d\xi \le C \delta_S^2 \delta_C^2 e^{-Ct}.
\]
Thus, 
\[
R_2 \le C \delta_C \deltas e^{-C\deltas t} + C \delta_C \deltar e^{-Ct}.
\]
As in $R_{23}, R_{24}$, decompose $ R_3$ as
\beq\label{Q1I-3}
\begin{aligned}
\|R_3 \|_{L^2(\bbr)}
 &=\bigg[\bigg(\int_{\big\{\xi\leq -\frac{\s}{2} t\big\}} + \int_{\big\{\xi> -\frac{\s}{2} t\big\}}\bigg) \big|((v^S)^{-\mb X}_\xi,(\theta^S)^{-\mb X}_\xi)\big|^2 \\[4mm]
& \qquad\qquad\qquad\qquad\qquad\qquad \cdot \big|(v^R-v_*, \theta^R-\theta_*,v^C-v^*, \theta^C-\theta^*) \big|^2 d\x\bigg]^{\frac12}\\
& =: C\sqrt{R_{31}+R_{32}}.
\end{aligned}
\eeq
Using \eqref{shock-pe-1},
\[
R_{31} \le C\big(\deltar^2+\delta_C^2\big) \int_\bbr (\deltas)^4 e^{-C\deltas |\xi -\mb X(t)|} e^{-C\deltas t} d\xi \le C\deltas^3 \big(\deltar^2+\delta_C^2\big) e^{-C\deltas t} .
\]
Since 
\[
\forall \xi> -\frac{\s t}{2},\quad \xi +\s t> \frac{\s t}{2} > 0 > \frac{\lambda_{1*}}{2}(1+t), 
\]
using \eqref{rarefaction-pe} with  \eqref{rarefaction-pe2}, and \eqref{vc-pe-2}, we have
\[
R_{32} \le C\deltas^2 \big(\deltar^2+\delta_C^2\big)e^{-Ct} \int_\bbr \big|((v^S)^{-\mb X}_\xi,(\theta^S)^{-\mb X}_\xi)\big| d\x \le C\deltas^3 \big(\deltar^2+\delta_C^2\big)e^{-Ct} .
\]
Thus, 
\[
R_3\le C\deltas^{3/2} (\deltar+\delta_C) e^{-C \deltas t} ,
\]
which gives $\eqref{wave-interactions}_2$.  \\
Combining the above estimates, we have
\beq\label{pcomp-1}
\big\|\big(\bar p-p^R-p^C-(p^S)^{-\mb X}\big)_\xi\big\|_{L^2(\bbr)}\leq C\deltas (\deltar+\delta_C) e^{-C \deltas t}+C\deltar\delta_Ce^{-Ct}.
\eeq
Moreover, we have
\begin{align*}
\begin{aligned}
&\left|\mu\left(\frac{\bar u_\x}{\bar v}-\frac{u^R_\x}{v^R}-\frac{u^C_\x}{v^C}-\frac{(u^S)^{-\mb X}_\x}{(v^S)^{-\mb X}}\right)_\x\right|\\[3mm]
&\le C\Big[|(u^R_{\xi\xi}, u^R_\x v^R_\xi)||(v^C-v_*, (v^S)^{-\mb X}-v^*)|+|(u^C_{\xi\xi}, u^C_\x v^C_\xi)|\\
& \qquad\quad\ \cdot|(v^R-v_*, (v^S)^{-\mb X}-v^*)|+|(u^S)^{-\mb X}_{\xi\xi}, (u^S)^{-\mb X}_{\xi} (v^S)^{-\mb X}_{\xi})||(v^R-v_*, v^C-v^*)| \\
&\qquad\quad\  +|u^R_\xi||( v^C_\xi, (v^S)^{-\mb X}_\xi)|+|u^C_\xi||(v^R_\xi, (v^S)^{-\mb X}_\xi) |+|(u^S)^{-\mb X}_\xi||( v^R_\xi, v^C_\xi) | \Big],
\end{aligned}
\end{align*}
which can be bounded by the same bound as in \eqref{pcomp-1}, by using the same techniques as above. Therefore, the wave interactions estimate $\eqref{wave-interactions}_1$ holds true.
\end{proof}

\subsection{Construction of weight function}
We define the weight function $a$ by
\begin{equation}\label{weight}
a(\xi):=1+\frac{\lam}{\deltas}(v^S(\xi)-v^*),
\end{equation}
where the constant $\lam$ is chosen to be so small but far bigger than $\deltas:=v_+-v^*$ such that 
\beq\label{lamsmall}
\deltas\ll \lam \le C\sqrt{\deltas}.
\eeq
Notice that 
\begin{equation}\label{a-bound}
1<a(\xi)<1+\lam,
\end{equation}
and 
\begin{equation}\label{a-prime}
a^\prime(\xi)=\frac{\lam}{\deltas}v^S_\xi>0,
\end{equation}
and so,
\beq\label{d-weight}
|a'|\sim\frac{\lam}{\deltas} |u^S_\xi|\sim \frac{\lam}{\deltas} |\theta^S_\xi|.
\eeq

\subsection{Relative entropy method} \label{ssec:ent}
For simplicity of computations on the evolution of the relative entropy, we will use the non-conserved quantities $U=(v,u,\theta)^t$ and $\bar U=(\bar v,\bar u,\bar \theta)^t$, where $U=(v,u,\theta)^t$ is the solution to the system \eqref{NS-1}, and $\bar U=(\bar v,\bar u,\bar \theta)^t$ the superposition of 1-rarefaction wave, 2-viscous contact wave and 3-viscous shock wave shifted by $\mb{X}$. Then, by \eqref{shwave},
\begin{equation}\label{tilvh}
\bar U(t,\xi) =\left(
\begin{array}{c}
\bar v(t,\xi)\\
\bar u(t,\xi)\\
\bar \theta(t,\xi)
\end{array}
\right) 
=\left(
\begin{array}{c}
\di v^R(t,\xi+\sigma t)+v^C(\xi+\sigma t)+ v^S(\xi -\mb X(t))-v_*-v^*\\
\di u^R(t,\xi+\sigma t)+u^C(t,\xi+\sigma t)+u^S(\xi-\mb X(t))-u_*-u^* \\
\di  \theta^R(t,\x+\s t)+ \theta^C(\xi+\sigma t)+ \theta^S(\xi-\mb X(t))-\theta_*-\theta^* 
\end{array}
\right) .
\end{equation}
First of all, it follows from Appendix \ref{app-ent} that for the (mathematical) entropy $\eta:=-s$, the relative entropy of $U$ and $\bar U$ is given by
\[
\eta\big(U|\bar U \big) = R  \left(\frac{v}{\bar v} -1- \log\frac{v}{\bar v}\right) +\frac{R}{\gamma-1} \left(\frac{\theta}{\bar\theta} -1- \log\frac{\theta}{\bar\theta}\right)  + \frac{(u-\bar u)^2}{2\bar\theta}.
\] 
Then, using the convex function $\Phi(z):=z-1-\ln z$, we have
\[
\eta\big(U|\bar U \big) =R\Phi\left(\frac{v}{\bar v}\right)+\frac{R}{\gamma-1}\Phi\left(\frac{\theta}{\bar \theta}\right) + \frac{(u-\bar u)^2}{2\bar\theta}.
\]
So, the relative entropy weighted by $\bar\theta$ is given by
\beq\label{def-rent}
\bar\theta\eta\big(U|\bar U \big) =R\bar\theta\Phi\left(\frac{v}{\bar v}\right)+\frac{R\bar\theta}{\gamma-1}\Phi\left(\frac{\theta}{\bar \theta}\right) + \frac{(u-\bar u)^2}{2}.
\eeq
Below, we will compute the evolution of the relative entropy of $U$ and $\bar U$ weighted by $a(\xi-\mb{X})\bar\theta(t,\xi)$ : 
\[
\int_\bbr a^{-\mb X}(\xi)\bar \theta(t,\xi) \eta\big(U(t,\xi)|\bar U(t,\xi) \big) d\xi.
\]

\begin{lemma}\label{lem-relative}
Let $a$ be the weight function defined by \eqref{weight}. Let $U$ be a solution to \eqref{NS-1}, and $\bar U$ the shifted wave given by \eqref{tilvh}.
Then,
\begin{align}
\begin{aligned}\label{ineq-0}
\frac{d}{dt}\int_{\bbr} a^{-\mb X}(\xi)\bar \theta(t,\xi) \eta\big(U(t,\xi)|\bar U(t,\xi) \big) d\xi =\dot{\mb X} (t) \mb{Y}(U) +\mathcal{J}^{bad}(U) - \mathcal{J}^{good}(U),
\end{aligned}
\end{align}
where
\begin{align}
\begin{aligned}\label{ybg-first}
&\mb Y(U):=-\int_{\bbr} \! a_\xi^{-\mb X} \bar \theta\eta(U|\bar U ) d\xi +\int_\bbr a^{-\mb X}\Big[-R (\theta^S)^{-\mb{X}}_\xi\Phi\left(\frac{v}{\bar v}\right) - \frac{R }{\gamma-1}(\theta^S)^{-\mb{X}}_\xi\Phi\left(\frac{\theta}{\bar \theta}\right) \Big]d\xi\\
&\quad\qquad\quad+\int_\bbr a^{-\mb X} \Big[(u^S)^{-\mb{X}}_\xi(u-\bar u)+\frac{(v^S)^{-\mb{X}}_\xi \bar p}{\bar v}(v-\bar v)+\frac{R}{\gamma-1}\frac{(\theta^S)^{-\mb{X}}_\xi}{\bar\theta}(\theta-\bar\theta)\Big]d\xi,
\end{aligned}
\end{align}
\begin{align}
\begin{aligned}\label{ybg-bad}
&\mathcal{J}^{bad}(U):= \int_\bbr a_\xi^{-\mb X} (u-\bar u)(p-\bar p) d\xi \\
&\quad+ \int_\bbr \!a^{-\mb X} \Big[R\Big((\theta_t^R-\sigma \theta_\x^R)+(\theta_t^C-\sigma \theta_\x^C)-\sigma (\theta^S)^{-\mb X}_\xi \Big)\Phi(\frac{v}{\bar v})-\frac{\bar p\bar u_\xi}{v\bar v}(v-\bar v)^2 \Big]d\x\\
& \quad + \int_\bbr \!a^{-\mb X} \Big[\frac{R}{\gamma-1}\Big((\theta_t^R-\sigma \theta_\x^R)+(\theta_t^C-\sigma \theta_\x^C)-\sigma (\theta^S)^{-\mb X}_\xi \Big)\Phi\left(\frac{\theta}{\bar \theta}\right) - \frac{\bar u_\xi}{\theta}(\theta-\bar \theta)(p-\bar p) +\bar p \bar u_\xi \frac{(\theta-\bar\theta)^2}{\theta\bar\theta} \Big]d\x\\ 
&\quad + \int_\bbr \!a^{-\mb X} \Big[ \mu (u-\bar u) \Big(\frac{u_\xi}{v}- \frac{\bar u_\xi}{\bar v} \Big) + \kappa \frac{\theta-\bar\theta}{\theta} \Big(\frac{\theta_\xi}{v}- \frac{\bar \theta_\xi}{\bar v}  \Big) \Big]_\xi \\
&\quad +\int_\bbr \!a^{-\mb X}  \Big[-\mu \bar u_\xi (\frac{1}{v}-\frac{1}{\bar v})(u-\bar u)_\x+\kappa \frac{\theta-\bar\theta}{\theta^2}\theta_\x(\frac{\theta_\x}{v}-\frac{\bar\theta_\x}{\bar v})-\kappa\frac{(\theta-\bar\theta)_\xi}{\theta}\bar\theta_\x(\frac1v-\frac 1{\bar v})\\
&\qquad\qquad\quad +\mu\frac{\theta-\bar\theta}{\theta}(\frac{u_\x^2}{v}-\frac{\bar u_\x^2}{\bar v})-Q_1(u-\bar u)-Q_2(1-\frac{\bar\theta}{\theta})\Big] d\xi,
\end{aligned}
\end{align}
and
\begin{align}
\begin{aligned}\label{ybg-good}
&\mathcal{J}^{good}(U):= \s\int_{\bbr} \! a_\xi^{-\mb X} \bar \theta\eta(U|\bar U ) d\xi  +\int_\bbr \!a^{-\mb X}\Big[ \frac{\mu}{v}|(u-\bar u)_\xi|^2+\frac{\kappa}{v\theta}|(\theta-\bar \theta)_\xi|^2\Big]d\x.
\end{aligned}
\end{align}
\end{lemma}
\begin{remark}
Since $\s a'(\xi) >0$, $\mathcal{J}^{good}$ consists of nonnegative terms. 
\end{remark}
\begin{proof}
First, we use \eqref{def-rent} to have
\begin{align}
\begin{aligned}\label{fieq}
&\frac{d}{dt}\int_{\bbr} a^{-\mb X}(\xi)\bar \theta(t,\xi) \eta\big(U(t,\xi)|\bar U(t,\xi) \big) d\xi \\
&= -\dot{\mb X} (t)  \int_{\bbr} \! a_\xi^{-\mb X} \bar \theta\eta(U|\bar U ) d\xi + \int_{\bbr} \! a^{-\mb X} \partial_t\bigg[R\bar\theta\Phi\left(\frac{v}{\bar v}\right)+\frac{R\bar\theta}{\gamma-1}\Phi\left(\frac{\theta}{\bar \theta}\right) + \frac{(u-\bar u)^2}{2}\bigg] d\xi 
\end{aligned} 
\end{align}
To compute the second term above, we first use the two systems \eqref{NS-1} and \eqref{bar-system} (satisfied by $\bar U$) to find
\beq\label{pereq1}
\left\{
\begin{array}{ll}
\di (v-\bar v)_t-\s (v-\bar v)_\xi-\dot{\mb X}(t)(v^S)^{-\mb X}_\xi-(u-\bar u)_\x=0,\\[3mm]
\di (u-\bar u)_t-\s (u-\bar u)_\xi-\dot{\mb X}(t)(u^S)^{-\mb X}_\xi+(p-\bar p)_\x=\mu\left(\frac{u_\xi}{v}-\frac{\bar u_\x}{\bar v}\right)_\xi-Q_1,\\[3mm]
\di \frac{R}{\gamma-1}(\theta-\bar \theta)_t-\frac{R\s}{\gamma-1}(\theta-\bar \theta)_\xi -\frac{R}{\gamma-1}\dot{\mb X}(t)(\theta^S)^{-\mb X}_\xi+(pu_\x -\bar p\bar u_\x)\\
\di\qquad\qquad\qquad \qquad\qquad\qquad=\kappa\left(\frac{\theta_\x}{v}-\frac{\bar \theta_\x}{\bar v}\right)_\xi+\mu\left(\frac{u_\x^2}{v}-\frac{\bar u_\xi^2}{\bar v}\right)-Q_2.
\end{array}
\right.
\eeq
Using $\Phi'(z)=1-1/z$ and $\eqref{pereq1}_1$ and $\eqref{bar-system}_1$, we have
\begin{align*}
\begin{aligned} 
\partial_t \bigg[\bar\theta\Phi\left(\frac{v}{\bar v}\right)\bigg] &= \bar\theta_t \Phi\left(\frac{v}{\bar v}\right) + \bar\theta \Phi' \left(\frac{v}{\bar v}\right) \left(\frac{v}{\bar v}\right)_t \\
&=  \bar\theta_t \Phi\left(\frac{v}{\bar v}\right) + \bar\theta\left(\frac{1}{\bar v}-\frac{1}{v} \right) \bigg[ (v-\bar v)_t +\bar v_t \left(1- \frac{v}{\bar v}\right)  \bigg]  \\
&=  \bar\theta_t \Phi\left(\frac{v}{\bar v}\right) + \bar\theta\left(\frac{1}{\bar v}-\frac{1}{v} \right) \bigg[ \s (v-\bar v)_\xi +\dot{\mb X}(t)(v^S)^{-\mb X}_\xi +(u-\bar u)_\x  \bigg] \\
&\quad +  \bar\theta\left(\frac{1}{\bar v}-\frac{1}{v} \right)  \left(1- \frac{v}{\bar v}\right) \bigg[ \s \bar v_\xi -\dot{\mb X}(t)(v^S)^{-\mb X}_\xi +\bar u_\x  \bigg]. 
\end{aligned} 
\end{align*}
In addition, since
\[
\partial_\x \bigg[\s\bar\theta\Phi\left(\frac{v}{\bar v}\right)\bigg] = \s\bar\theta_\x \Phi\left(\frac{v}{\bar v}\right)  +\s \bar\theta\left(\frac{1}{\bar v}-\frac{1}{v} \right) \bigg[  (v-\bar v)_\xi  +\bar v_\x \left(1- \frac{v}{\bar v}\right)  \bigg] ,
\]
we have
\begin{align}
\begin{aligned} \label{tvest}
&\partial_t \bigg[R\bar\theta\Phi\left(\frac{v}{\bar v}\right)\bigg] -\partial_\x \bigg[R\s\bar\theta\Phi\left(\frac{v}{\bar v}\right)\bigg]  \\
&= R( \bar\theta_t -\s\bar\theta_\x) \Phi\left(\frac{v}{\bar v}\right) + \dot{\mb X}(t)(v^S)^{-\mb X}_\xi \frac{\bar p}{\bar v} (v-\bar v) \\
&\quad -\frac{\bar p\bar u_\xi}{v\bar v}(v-\bar v)^2 + R\bar\theta \left(\frac{1}{\bar v} -\frac{1}{v} \right) (u-\bar u)_\x .
\end{aligned} 
\end{align}
Likewise, using $\eqref{pereq1}_3$ and $\eqref{bar-system}_3$, we have
\begin{align}
\begin{aligned} \label{tuest} 
&\partial_t \bigg[\frac{R\bar\theta}{\gamma-1}\Phi\left(\frac{\theta}{\bar \theta}\right)  \bigg] -\partial_\x \bigg[\frac{R\s\bar\theta}{\gamma-1}\Phi\left(\frac{\theta}{\bar \theta}\right)  \bigg]  \\
&= \frac{R}{\gamma-1}( \bar\theta_t -\s\bar\theta_\x)\Phi\left(\frac{\theta}{\bar \theta}\right)  + \dot{\mb X}(t)\frac{R}{\gamma-1}\frac{(\theta^S)^{-\mb{X}}_\xi}{\bar\theta}(\theta-\bar\theta) \\
&\quad + \frac{R}{v} (\bar\theta - \theta) (u-\bar u)_\x -\frac{\bar u_\x}{\theta}(\theta-\bar\theta) (p-\bar p) + \bar p \bar u_\xi \frac{(\theta-\bar\theta)^2}{\theta\bar\theta}   \\
&\quad +\Big[\kappa \frac{\theta-\bar\theta}{\theta} \Big(\frac{\theta_\xi}{v}- \frac{\bar \theta_\xi}{\bar v}  \Big) \Big]_\xi -\frac{\kappa}{v\theta}|(\theta-\bar \theta)_\xi|^2 +\kappa \frac{\theta-\bar\theta}{\theta^2}\theta_\x(\frac{\theta_\x}{v}-\frac{\bar\theta_\x}{\bar v}) \\
&\quad -\kappa\frac{(\theta-\bar\theta)_\xi}{\theta}\bar\theta_\x(\frac1v-\frac 1{\bar v}) +\mu\frac{\theta-\bar\theta}{\theta}(\frac{u_\x^2}{v}-\frac{\bar u_\x^2}{\bar v})-Q_2(1-\frac{\bar\theta}{\theta}) .
\end{aligned} 
\end{align}
Using  $\eqref{pereq1}_2$, we find
\begin{align}
\begin{aligned}  \label{ttest}
&\partial_t \bigg[\frac{(u-\bar u)^2}{2}\bigg] - \partial_\x \bigg[ \s \frac{(u-\bar u)^2}{2}\bigg] \\
&= \dot{\mb X}(t)(u^S)^{-\mb X}_\xi (u-\bar u) -\Big[ (p-\bar p) (u-\bar u)\Big]_\x + (p-\bar p) (u-\bar u)_\x \\
&\quad + \Big[ \mu\left(\frac{u_\xi}{v}-\frac{\bar u_\x}{\bar v}\right) (u-\bar u)\Big]_\xi -\frac{\mu}{v}|(u-\bar u)_\xi|^2 -\mu(u-\bar u)_\xi \bar u_\x \left(\frac{1}{v}-\frac{1}{\bar v}\right)  -Q_1 (u-\bar u)
\end{aligned} 
\end{align}
Therefore, combining \eqref{tvest}, \eqref{tuest} and \eqref{ttest}, we have
\begin{align*}
\begin{aligned} 
&\partial_t\bigg[R\bar\theta\Phi\left(\frac{v}{\bar v}\right)+\frac{R\bar\theta}{\gamma-1}\Phi\left(\frac{\theta}{\bar \theta}\right) + \frac{(u-\bar u)^2}{2}\bigg]  \\
& = \s  \partial_\x \bigg[ \bar\theta \eta (U |\bar U)\bigg] + ( \bar\theta_t -\s\bar\theta_\x) \bigg[ R\Phi\left(\frac{v}{\bar v}\right) +\frac{R}{\gamma-1}\Phi\left(\frac{\theta}{\bar \theta}\right) \bigg] \\
&\quad + \dot{\mb X}(t) \Big[\frac{(v^S)^{-\mb{X}}_\xi \bar p}{\bar v}(v-\bar v)+\frac{R}{\gamma-1}\frac{(\theta^S)^{-\mb{X}}_\xi}{\bar\theta}(\theta-\bar\theta)+(u^S)^{-\mb{X}}_\xi(u-\bar u)\Big] \\
&\quad -\frac{\bar p\bar u_\xi}{v\bar v}(v-\bar v)^2 -\frac{\bar u_\x}{\theta}(\theta-\bar\theta) (p-\bar p) +\partial_\x \Big[ \mu (u-\bar u) \Big(\frac{u_\xi}{v}- \frac{\bar u_\xi}{\bar v} \Big) + \kappa \frac{\theta-\bar\theta}{\theta} \Big(\frac{\theta_\xi}{v}- \frac{\bar \theta_\xi}{\bar v}  \Big) \Big] \\
&\quad -\mu(u-\bar u)_\xi \bar u_\x \left(\frac{1}{v}-\frac{1}{\bar v}\right)  +\kappa \frac{\theta-\bar\theta}{\theta^2}\theta_\x(\frac{\theta_\x}{v}-\frac{\bar\theta_\x}{\bar v})-\kappa\frac{(\theta-\bar\theta)_\xi}{\theta}\bar\theta_\x(\frac1v-\frac 1{\bar v})\\
&\quad +\mu\frac{\theta-\bar\theta}{\theta}(\frac{u_\x^2}{v}-\frac{\bar u_\x^2}{\bar v})-Q_1(u-\bar u)-Q_2(1-\frac{\bar\theta}{\theta})  -\frac{\mu}{v}|(u-\bar u)_\xi|^2-\frac{\kappa}{v\theta}|(\theta-\bar \theta)_\xi|^2
\end{aligned} 
\end{align*}
In addition, since
\[
\bar\theta_t -\s\bar\theta_\x = (\theta_t^R-\sigma \theta_\x^R)+(\theta_t^C-\sigma \theta_\x^C)-\sigma (\theta^S)^{-\mb X}_\xi - \dot{\mb X}(t) (\theta^S)^{-\mb X}_\xi,
\]
we substitute the above relations into \eqref{fieq} to get the desired representation.
 
\end{proof}

\subsection{Decompositions}
First of all, we will decompose the second and third terms of $\mathcal{J}^{bad}$ into the main term $\mb B_1$ of leading order, and some lower order terms, together with an additional good term $ \mathcal{G}^R$ as below:
\begin{align}
\begin{aligned}\label{b1gr}
&\mb B_1(U):=  \int_\bbr a^{-\mb X} (v^S)_\xi^{-\mb X} \Big[ \frac{(\gamma+1)p^*\sigma^*}{2(v^*)^2} (v-\bar v)^2+\frac{R\sigma^*}{2v^*\theta^*}(\theta-\bar\theta)^2-\frac{p^*\sigma^*}{v^*\theta^*}(v-\bar v)(\theta-\bar\theta)\Big]d\x,\\
& \mathcal{G}^R(U):= \int_\bbr|v^R_\x| |(v-\bar v, \theta-\bar \theta)|^2 d\x,
\end{aligned}
\end{align}
where the constant $p^*$ and $\s^*$ are as in \eqref{pstar}, that is,
\[
p^*:=p(v^*, \theta^*)=\frac{R\theta^*}{v^*}\quad \mbox{and}\quad \sigma^*:=\sqrt\frac{\gamma p^*}{v^*} = \frac{\sqrt{\gamma R\theta^*}}{v^*},
\]
For that, we first handle the second term of $\mathcal{J}^{bad}$:
\[
\int_\bbr \!a^{-\mb X} \underbrace{ \Big[R\Big((\theta_t^R-\sigma \theta_\x^R)+(\theta_t^C-\sigma \theta_\x^C)-\sigma (\theta^S)^{-\mb X}_\xi \Big)\Phi(\frac{v}{\bar v})-\frac{\bar p\bar u_\xi}{v\bar v}(v-\bar v)^2 \Big] }_{=: J_2}d\x .
\]
Since $\bar u_\xi = u_\xi^R +u_\xi^C +u_\xi^S$ and it holds from $\eqref{rarexi}_3$ that
\[
\theta_t^R-\sigma \theta_\x^R= -\frac{\gamma-1}{R} p^R u_\xi^R,
\]
we have
\begin{align*}
\begin{aligned}
J_2 & =\underbrace{-u^R_\xi\Big[(\gamma-1)p^R\Phi(\frac{v}{\bar v})+\frac{\bar p(v-\bar v)^2}{v\bar v}\Big] }_{=:J_{21}}+\underbrace{R(\theta_t^C-\sigma \theta_\x^C)\Phi(\frac{v}{\bar v})-\frac{\bar p u^C_\xi}{v\bar v}(v-\bar v)^2}_{=:J_{22}} \\
&\quad \ \underbrace{-R\sigma (\theta^S)^{-\mb X}_\xi \Phi(\frac{v}{\bar v})-\frac{(u^S)^{-\mb X}_\xi \bar p }{v\bar v}(v-\bar v)^2}_{=:J_{23}}.
\end{aligned}
\end{align*}
Using the fact $|v-\bar v|\le C\eps_1$, and by $\Phi(1)=\Phi'(1)=0, \Phi''(1)=1$,
\beq\label{app-phi}
\Phi(\frac{v}{\bar v}) =  \frac{(v-\bar v)^2}{2\bar v^2} + O(|v-\bar v|^3),
\eeq
and 
\begin{align*}
\begin{aligned}
|\bar p - p^R| & \le C( |\bar v - v^R| +|\bar \theta - \theta^R| ) \\
&\le C(|v^S - v^*| +|v^C - v_*| +|\theta^S - \theta^*| +|\theta^C - \theta_*| )  \le C(\deltas +\delta_C), 
\end{aligned}
\end{align*}
we have
\begin{align*}
J_{21} & \le -\frac{(\gamma+1)\bar p}{2\bar v^2}u^R_\xi (v-\bar v)^2 + C(\delta_0 +\eps_1 ) |u^R_\xi| |v-\bar v|^2.
\end{align*}
Using $\eqref{vcex}_3$, we have
\[
J_{22} \le C( |u^C_\xi|+ |\theta^C_{\xi\xi}| + |\theta^C_\xi||v^C_\xi| + |Q^C_2| ) (v-\bar v)^2,
\]
which together with \eqref{vc-pe} and \eqref{QC12} yields
\[
J_{22} \le C \dc(1+t)^{-1}e^{-\frac{C_1 |\xi+\sigma t|^2}{1+t}}(v-\bar v)^2.
\]
Using \eqref{shock-vu}, \eqref{theta-s}, \eqref{sm1} and 
\begin{align} 
\begin{aligned}\label{perror}
|\bar p - p^*| & \le C( |\bar v - v^*| +|\bar \theta - \theta^*| ) \\
&\le C(|v^R - v_*| + |v^S - v^*| +|v^C - v^*| +|\theta^R - \theta_*|+|\theta^S - \theta^*| +|\theta^C - \theta^*| ) \\
& \le C(\deltar+\deltas +\delta_C) \le C\delta_0, 
\end{aligned}
\end{align}
we have
\[
J_{23} \le \frac{(\gamma+1)p^*\sigma^*}{2(v^*)^2}(v^S)^{-\mb X}_\xi (v-\bar v)^2+ C(\delta_0 +\eps_1 ) |(v^S)^{-\mb X}_\xi | |v-\bar v|^2.
\]
Similarly, we handle the third term of $\mathcal{J}^{bad}$:
\[
\int_\bbr \!a^{-\mb X} \underbrace{ \Big[ \frac{R}{\gamma-1}\Big((\theta_t^R-\sigma \theta_\x^R)+(\theta_t^C-\sigma \theta_\x^C)-\sigma (\theta^S)^{-\mb X}_\xi \Big) \Phi\left(\frac{\theta}{\bar \theta}\right) - \frac{\bar u_\xi}{\theta}(\theta-\bar \theta)(p-\bar p) +\bar p \bar u_\xi \frac{(\theta-\bar\theta)^2}{\theta\bar\theta}  \Big] }_{=: J_3}d\x ,
\]
as follows: Using $p-\bar p =\frac{R}{v} (\theta-\bar\theta) -\frac{\bar p}{v} (v-\bar v) $ and $(u^S)^{-\mb X}_\xi <0$, we have
\begin{align*}
J_3 & = u^R_\xi\Big[ - p^R \Phi\left(\frac{\theta}{\bar \theta}\right) - \frac{ \theta-\bar \theta}{\theta} \Big(\frac{R}{v} (\theta-\bar\theta) -\frac{\bar p}{v} (v-\bar v) \Big) +\bar p \frac{(\theta-\bar\theta)^2}{\theta\bar\theta}  \Big] \\
&\quad -\Big[ \frac{R}{\gamma-1}(\theta_t^C-\sigma \theta_\x^C) \Phi(\frac{\bar\theta}{\theta}) +  u_\xi^C \frac{ \theta-\bar \theta}{\theta} \Big(\frac{R}{v} (\theta-\bar\theta) -\frac{\bar p}{v} (v-\bar v) \Big) -\bar p u^C_\xi \frac{(\theta-\bar\theta)^2}{\theta\bar\theta} \Big] \\
&\quad +\Big[\frac{R\sigma}{\gamma-1}(\theta^S)^{-\mb X}_\xi \Phi(\frac{\bar\theta}{\theta}) -(u^S)^{-\mb X}_\xi \frac{ \theta-\bar \theta}{\theta} \Big(\frac{R}{v} (\theta-\bar\theta) -\frac{\bar p}{v} (v-\bar v) \Big)  \Big] +\bar p (u^S)^{-\mb X}_\xi \frac{(\theta-\bar\theta)^2}{\theta\bar\theta}  \\
&\le u^R_\x \Big[ -\frac{R}{2\bar v\bar\theta}(\theta-\bar \theta) + \frac{\bar p}{\bar v \bar \theta} (v-\bar v)\Big] (\theta-\bar \theta) +C \dc(1+t)^{-1}e^{-\frac{C_1 |\xi+\sigma t|^2}{1+t}}|(v-\bar v,\theta-\bar\theta)|^2\\
&\quad + \frac{\sigma^* (v^S)^{-\mb X}_\xi}{2v^*\theta^*}\Big( R (\theta-\bar\theta)- 2p^* (v-\bar v)\Big)(\theta-\bar\theta) \\
&\quad +C(\delta_0 +\eps_1 ) (|u^R_\xi |+ |(v^S)^{-\mb X}_\xi | )  |(v-\bar v,\theta-\bar\theta)|^2.
\end{align*}
Therefore, by combining the above estimates together with $\bar p = \frac{R\bar\theta}{\bar v}$,
\begin{align*}
&\int_\bbr \!a^{-\mb X} (J_2 +J_3) d\xi \\
&\quad\le \mb B_1(U)  -  \underbrace{ \int_\bbr \!a^{-\mb X}   \frac{R u^R_\xi }{2\bar v^2}\Big[\frac{\gamma \bar\theta}{\bar v}(v-\bar v)^2+\Big(\sqrt{\frac{\bar\theta}{\bar v}}(v-\bar v) - \sqrt{\frac{\bar v}{\bar\theta}}(\theta-\bar\theta)\Big)^2\Big] d\xi }_{=: K} \\
&\qquad + C \dc(1+t)^{-1} \int_\bbr \!a^{-\mb X}  e^{-\frac{C_1 |\xi+\sigma t|^2}{1+t}} |(v-\bar v,\theta-\bar\theta)|^2 d\xi \\
&\qquad +C(\delta_0 +\eps_1 ) \int_\bbr \!a^{-\mb X}  (|u^R_\xi |+ |(v^S)^{-\mb X}_\xi | )  |(v-\bar v,\theta-\bar\theta)|^2 d\xi.
\end{align*}
We now derive the simpler form $\mathcal{G}^R$ from the above good term $K$ as follows. Using
\[
2(v-\bar v)(\theta-\bar\theta) \le  \frac{ 2\bar\theta}{\bar v} (v-\bar v)^2 + \frac{\bar v}{2\bar\theta}(\theta-\bar\theta)^2,
\]
and $u^R_\xi \sim v^R_\xi >0$ by Lemma \ref{lemma1.2} (1), 
we have
\[
K  \ge \int_\bbr \!a^{-\mb X}   \frac{R u^R_\xi }{2\bar v^2}\Big[\frac{(\gamma-1) \bar\theta}{\bar v}(v-\bar v)^2 + \frac{\bar v}{2\bar\theta}(\theta-\bar\theta)^2 \Big] d\xi  \ge C \mathcal{G}^R(U).
\]
Thus, 
\begin{align*}
\int_\bbr \!a^{-\mb X} (J_2 +J_3) d\xi &\le \mb B_1(U)  + C \mathcal{G}^R(U) \\
&\qquad + C \dc(1+t)^{-1} \int_\bbr \!a^{-\mb X}  e^{-\frac{C_1 |\xi+\sigma t|^2}{1+t}} |(v-\bar v,\theta-\bar\theta)|^2 d\xi \\
&\qquad +C(\delta_0 +\eps_1 ) \int_\bbr \!a^{-\mb X}  (|u^R_\xi |+ |(v^S)^{-\mb X}_\xi | )  |(v-\bar v,\theta-\bar\theta)|^2 d\xi.
\end{align*}

Therefore, it holds from Lemma \ref{lem-relative} that
\begin{align}
\begin{aligned}\label{ineq-1}
&\frac{d}{dt}\int_{\bbr} a^{-\mb X}(\xi)\bar \theta(t,\xi) \eta\big(U(t,\xi)|\bar U(t,\xi) \big) d\xi \\
&\qquad \le \dot{\mb X} (t) \mb{Y}(U) +\sum_{i=1}^5 \mb B_i(U) +\mb S_1(U)+\mb S_2(U) - C \mathcal{G}^R(U) -\mb G(U)  -\mb D(U),
\end{aligned}
\end{align}
where $\mb B_1,  \mathcal{G}^R$ are as in \eqref{b1gr}, and
\begin{align}
\begin{aligned}\label{dec-bad}
&\mb B_2(U):=  \int_\bbr a_\xi^{-\mb X} (u-\bar u)(p-\bar p) d\xi, \\
&\mb B_3(U):= - \int_\bbr \!a_\xi^{-\mb X} \Big[ \mu (u-\bar u) \Big(\frac{u_\xi}{v}- \frac{\bar u_\xi}{\bar v} \Big) + \kappa \frac{\theta-\bar\theta}{\theta} \Big(\frac{\theta_\xi}{v}- \frac{\bar \theta_\xi}{\bar v}  \Big) \Big] d\xi, \\
&\mb B_4(U):= \int_\bbr \!a^{-\mb X}  \Big[-\mu \bar u_\xi (\frac{1}{v}-\frac{1}{\bar v})(u-\bar u)_\x+\kappa \frac{\theta-\bar\theta}{\theta^2}\theta_\x(\frac{\theta_\x}{v}-\frac{\bar\theta_\x}{\bar v})-\kappa\frac{(\theta-\bar\theta)_\xi}{\theta}\bar\theta_\x(\frac1v-\frac 1{\bar v})\\
&\qquad\qquad\quad\qquad\qquad +\mu\frac{\theta-\bar\theta}{\theta}(\frac{u_\x^2}{v}-\frac{\bar u_\x^2}{\bar v})\Big] d\xi,\\
&\mb B_5(U):= C(\delta_0 +\eps_1 ) \int_\bbr \!a^{-\mb X}  (|u^R_\xi |+ |(v^S)^{-\mb X}_\xi | )  |(v-\bar v,\theta-\bar\theta)|^2 d\xi \\
&\qquad\qquad\quad + C \dc(1+t)^{-1} \int_\bbr \!a^{-\mb X}  e^{-\frac{C_1 |\xi+\sigma t|^2}{1+t}} |(v-\bar v,\theta-\bar\theta)|^2 d\xi, \\
&\mb S_1(U):= - \int_\bbr \!a^{-\mb X} Q_1(u-\bar u) d\xi, \quad\qquad  \mb S_2(U):= - \int_\bbr \!a^{-\mb X} Q_2(1-\frac{\bar\theta}{\theta}) d\xi, \\
\end{aligned}
\end{align}
and
\begin{align}
\begin{aligned}\label{dec-good}
& \mb G(U) := \s\int_{\bbr} \! a_\xi^{-\mb X} \bar \theta\eta(U|\bar U ) d\xi , \\
&\mb D(U) := \int_\bbr \!a^{-\mb X}\Big[ \frac{\mu}{v}|(u-\bar u)_\xi|^2+\frac{\kappa}{v\theta}|(\theta-\bar \theta)_\xi|^2\Big]d\x.
\end{aligned}
\end{align}

We decompose the functional $\mb Y$ as follows:
\[
\mb Y:= \sum_{i=1}^6\mb{Y}_i,
\]
where
\begin{align*}
\begin{aligned}
&\mb Y_1(U):= \int_\bbr a^{-\mb X} (u^S)^{-\mb{X}}_\xi(u-\bar u) d\xi,\\
&\mb Y_2(U):= \int_{\bbr} a^{-\mb X}\frac{(v^S)^{-\mb{X}}_\xi \bar p}{\bar v}(v-\bar v) d\xi,\\
&\mb Y_3(U):=\int_{\bbr}a^{-\mb X} \frac{R}{\gamma-1}\frac{(\theta^S)^{-\mb{X}}_\xi}{\bar\theta}(\theta-\bar\theta) d\xi,\\
&\mb Y_4(U):= - \int_\bbr a^{-\mb X} R (\theta^S)^{-\mb{X}}_\xi\Phi(\frac{v}{\bar v}) d\xi,\\
&\mb Y_5(U):= \int_\bbr a^{-\mb X} \frac{R }{\gamma-1}(\theta^S)^{-\mb{X}}_\xi\Phi(\frac{\bar \theta}{\theta}) d\xi , \\
&\mb Y_6(U):= -\int_{\bbr} \! a_\xi^{-\mb X} \bar \theta\eta(U|\bar U ) d\xi .
\end{aligned}
\end{align*}
Notice from \eqref{X(t)} that 
\beq\label{defxy}
\dot{\mb{X}}(t)=-\frac{M}{\delta_S} (\mb{Y}_1+\mb{Y}_2+\mb{Y}_3),
\eeq
and so,
\begin{equation}\label{XY}
\dot{\mb{X}}(t)\mb{Y}= -\frac{\deltas}{M} |\dot{\mb X} (t) |^2 +\dot{\mb{X}}(t)\sum_{i=4}^6\mb{Y}_i.
\end{equation}

\subsection{Leading order estimates}
\begin{lemma}\label{lem-sharp}
There exists $C>0$ such that 
\begin{align*}
\begin{aligned}
&-\frac{\deltas}{2M} |\dot{\mb X}|^2 +\mb B_1+\mb B_2 - \mb{G} -\frac{3}{4}\mb D \\&\le -C\mathcal{G}^S(U) + C ( \deltar^{4/3}+ \delta_C^{4/3})  \deltas^{4/3} e^{-C\deltas t},
\end{aligned}
\end{align*}
where
\beq\label{gs}
\mathcal{G}^S(U):=\int_\bbr|(v^S)^{-\mb X}_\x| |(v-\bar v, u-\bar u, \theta-\bar \theta)|^2 d\x,
\eeq
\end{lemma}
\begin{proof}
We first rewrite the main terms in terms of the new variables $y$ and $w$:
\begin{equation}\label{omega}
w :=u-\bar u^{\mb{X}},
\end{equation}
and
\begin{equation}\label{y-xi}
y:=\frac{ v^S(\xi -\mb{X}(t))-v^*}{\deltas}.
\end{equation}
Since $\mb X(t)$ is bounded on $[0,T]$ by \eqref{dxbound}, it follows from $\deltas:=v_+-v^*>0$ and $v^S_\xi>0$ that for any fixed $t$, the change of variable $\xi\in\bbr\mapsto y\in (0,1)$ is well-defined, together with
\begin{equation}\label{dery}
\frac{dy}{d\xi} =\frac{(v^S)^{-\mb{X}}_\xi}{\delta_S}>0.
\end{equation}
Note also that $a(\xi)=1+\lam y$ and  $a'(\xi)=\lam (dy/d\xi)>0$.\\
To perform the sharp estimates, we will use the $O(1)$-constants $p^*, \s^*$ defined in \eqref{pstar},
which are indeed independent of the small constants $\delta_0, \eps_1$, since $\frac{v_+}{2}\le v^*\le v_+$ and $\frac{\theta_+}{2}\le \theta^*\le \theta_+$. \\

\noindent$\bullet$ {\bf Estimates on $\mb B_2 - \mb G$: } \\
First, by \eqref{def-rent}, 
\begin{align*}
\begin{aligned}
&(u-\bar u)(p-\bar p) -\sigma \bar \theta \eta\big(U(t,\xi)|\bar U(t,\xi) \big) \\
&\qquad = (u-\bar u)(p-\bar p) -\sigma \Big[R\bar\theta\Phi(\frac{v}{\bar v})+\frac{R\bar\theta}{\gamma-1}\Phi(\frac{\theta}{\bar \theta}) + \frac{(u-\bar u)^2}{2} \Big].
\end{aligned}
\end{align*}
Using  $p-\bar p =\frac{R}{v} (\theta-\bar\theta) -\frac{\bar p}{v} (v-\bar v) $, and \eqref{sm1}, \eqref{app-phi}, \eqref{perror}, we have
\begin{align*}
\begin{aligned}
&(u-\bar u)(p-\bar p) -\sigma \bar \theta \eta\big(U(t,\xi)|\bar U(t,\xi) \big) \\
&\quad = (u-\bar u)\Big(\frac{R}{v} (\theta-\bar\theta) -\frac{\bar p}{v} (v-\bar v)  \Big) -\sigma \Big[R\bar\theta\Phi(\frac{v}{\bar v})+\frac{R\bar\theta}{\gamma-1}\Phi(\frac{\theta}{\bar \theta}) + \frac{(u-\bar u)^2}{2} \Big]\\
&\quad\le (u-\bar u)\Big[\frac{R}{v^*}(\theta-\bar\theta)-\frac{p^*}{v^*}(v-\bar v)\Big]  -\frac{R\sigma^*\theta^*}{2(v^*)^2}(v-\bar v)^2 - \frac{R\sigma^*}{2(\gamma-1)\theta^*}(\theta-\bar\theta)^2 \\
&\quad\quad -\frac{\sigma^*}{2}(u-\bar u)^2 +C\big(|v-\bar v| + |\bar v - v^*| + |\bar\theta - \theta^*| \big) |(v-\bar v, u-\bar u,\theta-\bar\theta)|^2\\
&\quad=  -\frac{R\sigma^*\theta^*}{2(v^*)^2} \Big[(v-\bar v)+\frac{u-\bar u}{\sigma^*}  \Big]^2 - \frac{R\sigma^*}{2(\gamma-1)\theta^*} \Big[(\theta-\bar\theta)-\frac{(\gamma-1)\theta^*}{v^*\s^*}(u-\bar u)\Big]^2\\
&\quad\quad +C\big(|v-\bar v|+ \deltas + |(v^R - v_* ,\theta^R - \theta_*)| +|(v^C - v^*,\theta^C - \theta^*)|  \big) |(v-\bar v, u-\bar u,\theta-\bar\theta)|^2,
\end{aligned}
\end{align*}
where the last equality is obtained by using
\[
\frac{R\theta^*}{2(v^*)^2 \s^*} + \frac{R(\gamma-1)\theta^*}{2(v^*)^2 \s^*} = \frac{R\gamma\theta^*}{2(v^*)^2 \s^*} = \frac{\s^*}{2} \quad \mbox{by }  \sigma^*= \frac{\sqrt{\gamma R\theta^*}}{v^*},
\]
Therefore, we have
\beq\label{b2g}
\mb B_2(U) - \mb G (U) \le - \mb G_1 (U) -  \mb G_2 (U) + \mb B_{new} (U), 
\eeq
where 
\begin{align}
\begin{aligned} \label{g1g2}
& \mb G_1 (U):= \frac{R\sigma^*\theta^*}{2(v^*)^2} \int_\bbr a_\xi^{-\mb X}  \Big[(v-\bar v)+\frac{u-\bar u}{\sigma^*}  \Big]^2 d\xi,\\
& \mb G_2 (U):= \frac{R\sigma^*}{2(\gamma-1)\theta^*} \int_\bbr a_\xi^{-\mb X}  \Big[(\theta-\bar\theta)-\frac{(\gamma-1)\theta^*}{v^*\s^*}(u-\bar u)\Big]^2 d\x,\\
& \mb B_{new} (U) := C\deltas \int_\bbr a_\xi^{-\mb X}  |(v-\bar v, u-\bar u,\theta-\bar\theta)|^2 d\x \\
&\qquad\qquad  + C  \int_\bbr a_\xi^{-\mb X}  |(v-\bar v, u-\bar u,\theta-\bar\theta)|^3 d\x \\
&\qquad\qquad  + C  \int_\bbr a_\xi^{-\mb X} \big( |(v^R - v_* ,\theta^R - \theta_*)| +|(v^C - v^*,\theta^C - \theta^*)|  \big) |(v-\bar v, u-\bar u,\theta-\bar\theta)|^2 d\x .
\end{aligned}
\end{align}
The two good terms $\mb G_1, \mb G_2$ will be used in the remaining estimates, while
the bad term $\mb B_{new}$ can be controlled by the three good terms $\mb G_1, \mb G_2,\mb D$ and $\mathcal{G}^S$,
as follows.\\
First, using \eqref{a-prime},
\[
\deltas \int_\bbr a_\xi^{-\mb X}  |(v-\bar v, u-\bar u,\theta-\bar\theta)|^2 d\x \le C\l \mathcal{G}^S.
\]
Using \eqref{a-prime} and the interpolation inequality, and $\lam\le C\sqrt{\deltas}$ by \eqref{lamsmall}, we have
\begin{align*}
\begin{aligned}
&  \int_\bbr a_\xi^{-\mb X}  |(v-\bar v, u-\bar u,\theta-\bar\theta)|^3 d\x \\
&\quad \le C \int_\bbr  |a_\xi^{-\mb X} |\Big|(v-\bar v)+\frac{u-\bar u}{\sigma^*}\Big| ^3d\xi+C\int_\bbr |a_\xi^{-\mb X}||w|^3d\xi\\
&\qquad+C \int_\bbr |a_\xi^{-\mb X} |\Big| (\theta-\bar\theta)-\frac{(\gamma-1)\theta^*}{v^*\s^*}(u-\bar u) \Big|^3 d\xi \\
&\quad \le C\eps_1 (\mb{G}_1+\mb{G}_2)+C \frac{\l}{\deltas} \int_\bbr |(v^S)_\xi^{-\mb X} |  \|w\|_{L^\infty(\bbr)}^2 |w| d\xi \\
&\quad \le C \eps_1(\mb{G}_1+\mb{G}_2)+C\frac{\lam}{\deltas}  \|w_\xi\|_{L^2(\bbr)} \|w\|_{L^2(\bbr)}  \sqrt{\int_\bbr   |(v^S)_\xi^{-\mb X} | w^2 d\xi} \sqrt{\int_\bbr  |(v^S)_\xi^{-\mb X} |  d\xi} \\
&\quad \le C \eps_1(\mb{G}_1+\mb{G}_2)+C\eps_1 \frac{\lam}{\sqrt{\deltas}}  \sqrt{\mb D} \sqrt{\mathcal{G}^S}\\
&\quad   \le C \eps_1(\mb{G}_1+\mb{G}_2 + \mb D +\mathcal{G}^S),
\end{aligned}
\end{align*}
Likewise, using the interpolation inequality and Lemma \ref{lemma2.2}, we have
\begin{align*}
\begin{aligned}
& \int_\bbr a_\xi^{-\mb X} \big( |(v^R - v_* ,\theta^R - \theta_*)| +|(v^C - v^*,\theta^C - \theta^*)|  \big) |(v-\bar v, u-\bar u,\theta-\bar\theta)|^2 d\x \\
&\quad \le C(\delta_R +\delta_C)  (\mb{G}_1+\mb{G}_2) \\
&\qquad+C \frac{\l}{\deltas} \int_\bbr  |(v^S)_\xi^{-\mb X} |  \big( |(v^R - v_* ,\theta^R - \theta_*)| +|(v^C - v^*,\theta^C - \theta^*)|  \big) |w|^2 d\xi  \\
&\quad\le C(\delta_R +\delta_C)  (\mb{G}_1+\mb{G}_2)+C\frac{\lam}{\deltas} \|w\|_{L^4(\bbr)}^2 \Big[\| |(v^S)_\xi^{-\mb X}| | (v^R - v_* ,\theta^R - \theta_*)|\|_{L^2(\bbr)}\\
 &\qquad \quad +\| |(v^S)_\xi^{-\mb X} | | (v^C - v^* ,\theta^C - \theta^*)|\|_{L^2(\bbr)}  \Big]  \\
 &\quad \le C \delta_0(\mb{G}_1+\mb{G}_2)+C\frac{1}{\sqrt\deltas} \|w_\xi\|_{L^2(\bbr)}^{1/2}  \|w\|_{L^2(\bbr)}^{3/2}  \deltas^{\frac32}(\deltar+\delta_C)e^{-C\deltas t}\\
&\quad\leq  C \delta_0(\mb{G}_1+\mb{G}_2)+C\mb D^{\frac14} \eps_1^{\frac32}\deltas (\deltar+\delta_C)e^{-C\deltas t}\\
&\quad\leq  C \delta_0(\mb{G}_1+\mb{G}_2)+C\eps_1 \mb D +C \deltas^{4/3} ( \deltar^{4/3}+ \delta_C^{4/3}) e^{-C\deltas t}.
\end{aligned}
\end{align*}
Therefore, 
\beq\label{bnew}
 \mb B_{new} \le C (\l+\eps_1+\delta_0) ( \mb{G}_1+\mb{G}_2 + \mb D +\mathcal{G}^S)+C \deltas^{4/3} ( \deltar^{4/3}+ \delta_C^{4/3}) e^{-C\deltas t}.
\eeq

\noindent$\bullet$ {\bf Estimates on $-\frac{\deltas}{2M} |\dot{\mb X}|^2$:}
First, to estimate the term $-\frac{\deltas}{2M} |\dot{\mb X}|^2$, we will estimate $\mb Y_1, \mb Y_2, \mb Y_3$ due to \eqref{defxy}.\\
By the change of variable \eqref{dery} and $\eqref{VS}_1$, we have
\[
\mb Y_1 = \int_\bbr a^{-\mb X} (u^S)^{-\mb{X}}_\xi(u-\bar u) d\xi =-\deltas \s\int_0^1 a^{-\mb X} w  dy.
\]
Using \eqref{sm1} and $|a-1|\le \lam$, we have
\beq\label{y1}
\left|\mb Y_1 + \deltas \s^*\int_0^1 w dy \right| \le C\deltas \l \int_0^1|w| dy.
\eeq
For 
\[
\begin{array}{ll}
\di \mb Y_2 = \int_{\bbr} a^{-\mb X} \frac{(v^S)^{-\mb{X}}_\xi \bar p}{\bar v}(v-\bar v) d\xi \\[3mm]
\di\quad\ \   =-\int_{\bbr}a^{-\mb X} \frac{(v^S)^{-\mb{X}}_\xi \bar p }{\bar v \s^*}(u-\bar u) d\x +\int_{\bbr} a^{-\mb X} \frac{(v^S)^{-\mb{X}}_\xi \bar p}{\bar v}\Big[(v-\bar v)+\frac{u-\bar u}{\sigma^*}\Big] d\x,
\end{array}
\] 
using \eqref{perror} and \eqref{a-prime}, we have
\beq\label{y2}
\left|\mb Y_2 + \frac{p^*\deltas}{v^*\sigma^*}\int_0^1 \omega dy\right| \le C\deltas (\lam+\delta_0)\int_0^1|\omega| dy+ C\frac{\deltas}{\l}\int_{\bbr} a^{-\mb X}_\xi \left|(v-\bar v)+\frac{u-\bar u}{\sigma^*}\right| d\x.
\eeq
Likewise, since
\[
\begin{array}{ll}
\di \mb Y_3 = \int_{\bbr}a^{-\mb X} \frac{R}{\gamma-1}\frac{(\theta^S)^{-\mb{X}}_\xi}{\bar\theta}(\theta-\bar\theta) d\xi  \\[4mm]
\di\quad\ \   = \frac{R \theta^*}{v^*\s^*} \int_{\bbr} a^{-\mb X} \frac{(\theta^S)^{-\mb{X}}_\xi}{\bar\theta}  (u-\bar u^{\mb X}) d\x \\
\di\qquad\qquad +\int_{\bbr}a^{-\mb X} \frac{R}{\gamma-1}\frac{(\theta^S)^{-\mb{X}}_\xi}{\bar\theta}\Big[(\theta-\bar\theta)-\frac{(\gamma-1)\theta^*}{v^*\s^*}(u-\bar u)\Big] d\x,
\end{array}
\] 
using $|\theta-\theta^*| \le C(\eps_1 +\delta_0)$ and \eqref{theta-s}, \eqref{d-weight}, we have
\beq\label{y3}
\begin{array}{ll}
\di 
\left|\mb Y_3 +  \frac{(\gamma-1)p^*\deltas}{v^*\sigma^*}\int_0^1 \omega dy \right| \le C\deltas (\lam+\delta_0+\varepsilon_1)\int_0^1|\omega| dy\\[3mm]
\di \qquad\qquad\qquad\qquad\qquad  + C\frac{\deltas}{\l}\int_{\bbr} a^{-\mb{X}}_\xi \left| (\theta-\bar\theta)-\frac{(\gamma-1)\theta^*}{v^*\s^*}(u-\bar u)  \right|  d\x.
\end{array}
\eeq
Therefore, using \eqref{defxy}, \eqref{y1}, \eqref{y2} and \eqref{y3} together with $\s^* =\sqrt\frac{\gamma p^*}{v^*}$, we have
\[
\begin{array}{ll}
\di\left| \dot{\mb X} -2\sigma^*M\int_0^1 \omega dy\right| &\di \le C (\lam+\delta_0+\eps_1)\int_0^1|\omega| dy+ \frac{C}{\l}\int_{\bbr} a^{-\mb{X}}_\xi \left|(v-\bar v)+\frac{u-\bar u}{\sigma^*}\right| d\x\\[4mm]
&\di\quad + \frac{C}{\l}\int_{\bbr} a^{-\mb{X}}_\xi \left| (\theta-\bar\theta)-\frac{(\gamma-1)\theta^*}{v^*\s^*}(u-\bar u) \right| d\x,
\end{array}
\]
which yields
\[
 \left( \left|2\sigma^*M \int_0^1 \omega dy\right| - |\dot{\mb X} | \right)^2  \le  C(\lam+\delta_0+\eps_1)^2 \int_0^1|\omega|^2 dy + \frac{C}{\l^2}  (\mb G_1 +\mb G_2)\int_{\bbr} a^{-\mb{X}}_\xi d\x.
\]
This and the algebraic inequality $\frac{p^2}{2}-q^2 \le (p-q)^2$ for all $p,q\ge 0$ imply
\[
2(\sigma^*)^2M^2\left(\int_0^1 \omega dy\right)^2 - |\dot{\mb X}|^2 \le  C(\lam+\delta_0+\eps_1)^2 \int_0^1|\omega|^2 dy + \frac{C}{\l}  (\mb G_1 +\mb G_2).
\]
Thus,
\beq\label{gxest}
 -\frac{\deltas}{2M} |\dot{\mb X}|^2 \le -(\sigma^*)^2M\deltas\left(\int_0^1 \omega dy\right)^2 +C\ds(\lam+\delta_0+\eps_1)^2 \int_0^1\omega^2 dy
+ \frac{C\delta_S}{\l}  (\mb G_1 +\mb G_2). 
\eeq

\noindent$\bullet$ {\bf Estimates on $\mb B_1$ :}\\
Set
\beq\label{b1*}
\begin{array}{ll}
\di \mb B_1(U)=  \int_\bbr a^{-\mb X} (v^S)_\xi^{-\mb X} \Big[ \frac{(\gamma+1)p^*\sigma^*}{2(v^*)^2} (v-\bar v)^2+\frac{R\sigma^*}{2v^*\theta^*}(\theta-\bar\theta)^2-\frac{p^*\sigma^*}{v^*\theta^*}(v-\bar v)(\theta-\bar\theta)\Big]d\x\\[4mm]
\di \qquad\quad := \mb B_{11}(U)+ \mb B_{12}(U)+\mb B_{13}(U).
\end{array}
\eeq
First, using Young's inequality, 
\begin{align*}
\begin{aligned}
\mb B_{11} & = \frac{(\gamma+1)p^*\sigma^*}{2(v^*)^2}  \int_\bbr a^{-\mb X} (v^S)_\xi^{-\mb X} \Big|\Big((v-\bar v)+\frac{u-\bar u}{\sigma^*}\Big) - \frac{u-\bar u}{\sigma^*} \Big|^2 d\xi  \\
&\le  \frac{(\gamma+1)p^*\sigma^*}{2(v^*)^2} \left(1+\l+\Big(\frac{\deltas}{\l}\Big)^\frac12 \right) \int_\bbr (v^S)_\xi^{-\mb X} \Big| \frac{u-\bar u}{\sigma^*} \Big|^2 d\xi \\
&\quad + C\Big(\frac{\deltas}{\l}\Big)^{-\frac12} \int_\bbr  (v^S)_\xi^{-\mb X} \Big|(v-\bar v)+\frac{u-\bar u}{\sigma^*} \Big|^2 d\xi,
\end{aligned}
\end{align*}
which together with \eqref{a-prime} yields
\beq\label{b11}
\mb B_{11}  \le \frac{(\gamma+1)p^*}{2(v^*)^2\s^*} \left(1+\l+\Big(\frac{\deltas}{\l}\Big)^\frac12 \right) \deltas \int_0^1 w^2 dy + C\Big(\frac{\deltas}{\l}\Big)^{\frac12} \mb G_1.
\eeq
Similarly, we estimate (using $p^*=R\theta^*/v^*$)
\begin{align}\label{b12}
\begin{aligned}
\mb B_{12} & =\frac{R\sigma^*}{2v^*\theta^*}  \int_\bbr a^{-\mb X} (v^S)_\xi^{-\mb X}  \left| \Big( (\theta-\bar\theta)-\frac{(\gamma-1)\theta^*}{v^*\s^*}(u-\bar u) \Big) + \frac{(\gamma-1)\theta^*}{v^*\s^*}(u-\bar u) \right|^2 d\xi  \\
&\le  \frac{(\gamma-1)^2p^*}{2(v^*)^2\s^*} \left(1+\l+\Big(\frac{\deltas}{\l}\Big)^\frac12 \right) \deltas \int_0^1 w^2 dy + C\Big(\frac{\deltas}{\l}\Big)^{\frac12} \mb G_2.
\end{aligned}
\end{align}
and 
\begin{align}\label{b13}
\begin{aligned}
\mb B_{13} & \le \frac{p^*\sigma^*}{v^*\theta^*} \int_\bbr a^{-\mb X} (v^S)_\xi^{-\mb X} \Big|\Big((v-\bar v)+\frac{u-\bar u}{\sigma^*}\Big) - \frac{u-\bar u}{\sigma^*} \Big| \\
&\qquad\qquad\quad \cdot \left| \Big( (\theta-\bar\theta)-\frac{(\gamma-1)\theta^*}{v^*\s^*}(u-\bar u) \Big) + \frac{(\gamma-1)\theta^*}{v^*\s^*}(u-\bar u) \right| d\xi  \\
&\le  \frac{(\gamma-1)p^*}{(v^*)^2\s^*} \left(1+\l+\Big(\frac{\deltas}{\l}\Big)^\frac12 \right) \deltas \int_0^1 w^2 dy + C\Big(\frac{\deltas}{\l}\Big)^{\frac12} (\mb G_1 + \mb G_2).
\end{aligned}
\end{align}
Let $\alpha^*$ be the $O(1)$-constant defined by
\beq\label{alpha*}
\alpha^*:= \frac{\gamma(\gamma+1) p^*}{2(v^*)^2\s^*}.
\eeq
Substituting \eqref{b11}, \eqref{b12} and \eqref{b13} into \eqref{b1*}, we have
\beq\label{b1-est}
\mb B_1 \le \alpha^* \left(1+\l+\Big(\frac{\deltas}{\l}\Big)^\frac12 \right) \deltas \int_0^1 w^2 dy+ C\Big(\frac{\deltas}{\l}\Big)^{\frac12} (\mb G_1 + \mb G_2).
\eeq

\noindent$\bullet$ {\bf Estimates on $\mb D$ :} \\
Set
\beq\label{D-*}
\mb D(U)= \int_\bbr \!a^{-\mb X}\Big[ \frac{\mu}{v}|(u-\bar u)_\xi|^2+\frac{\kappa}{v\theta}|(\theta-\bar \theta)_\xi|^2\Big]d\x :=\mb D_1(U)+\mb D_2(U).
\eeq
First, using $a\ge 1$ and the change of variable, we have
\[
\mb D_1\ge  \int_\bbr \frac{\mu}{v}|(u-\bar u)_\xi|^2 d\xi = \int_0^1 |\partial_y \omega|^2 \frac{\mu}{v} \Big(\frac{dy}{d\xi}\Big) dy.
\]
To estimate the right-hand side, set $v^S:=(v^S)^{-\mb X}, p^S:=p ((v^S)^{-\mb X},(\theta^S)^{-\mb X}), p_+:= p(v_+,\theta_+)$ for simplicity.
Then, using $\eqref{VS}_1$ and then integrating $\eqref{VS}_2$ over $(-\infty,\xi]$, we have
\[
\mu \frac{v^S_\xi }{v^S} = -\s (v^S -v^*) - \frac{p^S-p^*}{\s},
\]
and so,
\begin{align*}
\begin{aligned}
\deltas  \frac{\mu}{v^S}\frac{dy}{d\xi}= \mu \frac{v^S_\xi }{v^S}&= -\s (v^S -v^*) - \frac{p^S-p^*}{\s}\\
&=-\frac{1}{\s}\left[ \s^2(v^S -v^*) + p^S-p^*\right],
\end{aligned}
\end{align*}
which together with $\s^2=\frac{p^*-p_+}{v_+-v^*}$ yields
\begin{align*}
\begin{aligned}
\deltas  \frac{\mu}{v^S} \frac{dy}{d\xi}
&=-\frac{1}{\s(v_+ -v^*)}\left[ (p^*-p_+) (v^S -v^*) + (v_+-v^*)(p^S-p^*) \right]\\
&=-\frac{1}{\s(v_+ -v^*)}\bigg[(p^S-p_+) (v^S -v^*) +  (v^S -v^*) (p^*-p^S) \\
&\qquad +  (v^S- v^* ) (p^S-p^*) + (v_+-v^S)(p^S-p^*) \bigg] \\
&=-\frac{1}{\s(v_+ -v^*)}\left[(p^S-p_+) ( v^S -v^*) + (v_+- v^S)(p^S-p^*) \right].
\end{aligned}
\end{align*}
Since $y=\frac{v^S-v^*}{\deltas}$ and $1-y=\frac{v_+-v^S}{\deltas}$,
\[
\frac{1}{y(1-y)} \frac{\mu}{v^S} \frac{dy}{d\xi}=\frac{1}{\s}\left( \frac{p^S-p_+}{v^S - v_+}-\frac{p^S -p^*}{v^S-v^*} \right).
\]
By Appendix \ref{app-est}, we have
\beq\label{sharp-d}
\left| \frac{1}{y(1-y)} \frac{\mu}{v^S} \frac{dy}{d\xi} -\deltas \alpha^*\frac{\mu R\gamma}{\mu R\gamma+\kappa (\gamma-1)^2} \right|\leq C\deltas^2,
\eeq
This together with $|v-v^S|\le C(\eps_1+\delta_0)$ implies
\begin{align}
\begin{aligned}\label{D1-e1}
\mb D_1&\ge  \int_0^1 |\partial_y \omega|^2\frac{\mu}{v^S} \Big(\frac{dy}{d\xi}\Big) dy+\int_0^1 |\partial_y \omega|^2(\frac{v^S}{v}-1)  \frac{\mu}{v^S}\Big(\frac{dy}{d\xi}\Big) dy\\
&\ge \alpha^*\frac{\mu R\gamma}{\mu R\gamma+\kappa (\gamma-1)^2} \big(1-C(\delta_0+\eps_1)\big)  \deltas \int_0^1y(1-y)  |\partial_y \omega|^2 dy.
\end{aligned}
\end{align}
Likewise, using the change of variable, and \eqref{sharp-d}, we have
\beq\label{D2e}
\begin{array}{ll}
\di \mb D_2\geq\int_\bbr \frac{\kappa}{v\theta}|(\theta-\bar \theta)_\xi|^2d\x=\int_0^1\frac{\kappa}{v\theta}|(\theta-\bar \theta)_y|^2 \Big(\frac{dy}{d\xi}\Big)  dy\\[3mm]
\di \qquad =\int_0^1\frac{\kappa v^S}{\mu v\theta}|(\theta-\bar \theta)_y|^2 \frac{\mu}{v^S}\Big(\frac{dy}{d\xi}\Big)  dy\\[3mm]
\di \qquad =\int_0^1\frac{\kappa }{\mu \theta^*}|(\theta-\bar \theta)_y|^2 \frac{\mu}{v^S}\Big(\frac{dy}{d\xi}\Big)  dy+\int_0^1\frac{\kappa }{\mu}\left(\frac{v^S}{v\theta}-\frac{1}{\theta^*}\right)|(\theta-\bar \theta)_y|^2 \frac{\mu}{v^S}\Big(\frac{dy}{d\xi}\Big)  dy\\
\di\qquad \ge  \alpha^*\frac{\mu R\gamma}{\mu R\gamma+\kappa (\gamma-1)^2} \frac{\kappa }{\mu \theta^*}\big(1-C(\delta_0+\eps_1)\big)\deltas  \int_0^1y(1-y)  |(\theta-\bar \theta)_y|^2  \, dy.
\end{array}
\eeq

Our intention is to use the Poincare-type inequality of Lemma \ref{lem-poin} to absorb the main bad term $\mb B_1$ by the diffusion term $\mb D$ with \eqref{gxest}. However, since 
\[
1> \frac{\mu R\gamma}{\mu R\gamma+\kappa (\gamma-1)^2} \to 0 \quad\mbox{as } \gamma\to \infty,
\]
it follows from \eqref{b1-est} and \eqref{D1-e1} that $\mb D_1$ is not enough to control $\mb B_1$. Thus, we need to extract an additional good term on $w$ from $\mb D_2$, as follows.\\
First, using Lemma \ref{lem-poin} and 
\[
\int_0^1 |\omega-\bar \omega|^2 dy = \int_0^1\omega^2 dy -{\bar \omega}^2, \quad \bar \omega:=\int_0^1 \omega dy,
\]
we have
\begin{align*}
\begin{aligned}
\mb D &\ge 2 \alpha^*\frac{\mu R\gamma}{\mu R\gamma+\kappa (\gamma-1)^2} \big(1-C(\delta_0+\eps_1)\big)  \deltas \bigg[ \int_0^1\omega^2 dy -{\bar \omega}^2 \\
&\qquad\qquad\qquad\qquad\qquad\qquad + \frac{\kappa }{\mu \theta^*} \left(\int_0^1 |\theta-\bar \theta|^2  dy -\Big(\int_0^1(\theta-\bar \theta) dy\Big)^2\right) \bigg] .
\end{aligned}
\end{align*}
Observe that Young's inequality yields
\begin{align*}
\begin{aligned}
\int_0^1 |\theta-\bar \theta|^2  dy & =  \int_0^1  \left| \Big( (\theta-\bar\theta)-\frac{(\gamma-1)\theta^*}{v^*\s^*}(u-\bar u) \Big) + \frac{(\gamma-1)\theta^*}{v^*\s^*}(u-\bar u) \right|^2 dy  \\
&\ge  \left(\frac{(\gamma-1)\theta^*}{v^*\s^*}\right)^2 \left(1-\Big(\frac{\deltas}{\l}\Big)^{\frac12} \right) \int_0^1 w^2 dy \\
&\quad - C\Big(\frac{\deltas}{\l}\Big)^{-\frac12}  \int_0^1   \Big( (\theta-\bar\theta)-\frac{(\gamma-1)\theta^*}{v^*\s^*}(u-\bar u) \Big)^2 dy,
\end{aligned}
\end{align*}
and 
\begin{align*}
\begin{aligned}
\left(\int_0^1(\theta-\bar \theta) dy\right)^2 & \le  2\left(\int_0^1 \Big( \frac{(\gamma-1)\theta^*}{v^*\s^*}(u-\bar u) \Big)  dy\right)^2 \\
&\quad +2\left(\int_0^1 \Big( (\theta-\bar\theta)-\frac{(\gamma-1)\theta^*}{v^*\s^*}(u-\bar u) \Big)  dy\right)^2 \\
&\le 2 \left(\frac{(\gamma-1)\theta^*}{v^*\s^*}\right)^2 \bar w^2 +2 \int_0^1   \Big( (\theta-\bar\theta)-\frac{(\gamma-1)\theta^*}{v^*\s^*}(u-\bar u) \Big)^2 dy,
\end{aligned}
\end{align*}
which together with $\s^*=\frac{\sqrt{\gamma R \theta^*}}{v^*}$, \eqref{dery} and \eqref{a-prime} implies
\begin{align}
\begin{aligned}\label{D12}
\mb D &\ge 2 \alpha^* (1-C(\delta_0+\eps_1))  \deltas \frac{\mu R\gamma}{\mu R\gamma+\kappa (\gamma-1)^2}  \left(1+ \frac{\kappa (\gamma-1)^2}{\mu R\gamma}\right) \left(1-\Big(\frac{\deltas}{\l}\Big)^{\frac12} \right) \int_0^1\omega^2 dy \\
&\qquad  -2\alpha^* \deltas \left[1+ 2\frac{\kappa }{\mu \theta^*} \left(\frac{(\gamma-1)\theta^*}{v^*\s^*}\right)^2 \right] \bar w^2  - C\Big(\frac{\deltas}{\l}\Big)^{\frac12} \mb G_2\\
&= 2 \alpha^* \left(1-C(\delta_0+\eps_1)-\Big(\frac{\deltas}{\l}\Big)^{\frac12} \right)  \deltas \int_0^1 w^2 dy \\
&\quad -2\alpha^* \deltas  \left[1+ \frac{2\kappa(\gamma-1)^2}{\mu R\gamma} \right] \bar w^2  - C\Big(\frac{\deltas}{\l}\Big)^{\frac12} \mb G_2.
\end{aligned}
\end{align}

\noindent$\bullet$ {\bf Conclusion:}
First, combining \eqref{b2g}, \eqref{bnew}, \eqref{gxest},   \eqref{b1-est}, \eqref{D12} and using the smallness of $\frac{\deltas}{\l}, \l$ (as \eqref{lamsmall}), and $\delta_0$, $\eps_1$, we have
\begin{align*}
\begin{aligned}
-\frac{\deltas}{2M} |\dot{\mb X}|^2 + \mb B_1+\mb B_2-\mb{G} -\frac{3}{4}\mb D & \le C (\l+\eps_1+\delta_0) \mathcal{G}^S +C \deltas^{4/3} ( \deltar^{4/3}+ \delta_C^{4/3}) e^{-C\deltas t}  \\
&\quad  -\frac{\alpha^*}{4} \deltas \int_0^1 w^2 dy- \frac{1}{2} (\mb G_1+ \mb G_2)\\
&\quad  -(\sigma^*)^2M\deltas \bar w^2   + \frac{3}{2} \alpha^* \deltas  \left[1+ \frac{2\kappa(\gamma-1)^2}{\mu R\gamma} \right] \bar w^2 .
\end{aligned}
\end{align*}
Choosing the specific $O(1)$-constant $M$ as in \eqref{X(t)}, that is,
\[
M=\frac{3}{2(\sigma^*)^2} \alpha^* \left(1+\frac{2\kappa (\gamma-1)^2}{\mu R\gamma}\right),
\]
we have
\begin{align*}
\begin{aligned}
-\frac{\deltas}{2M} |\dot{\mb X}|^2 + \mb B_1+\mb B_2-\mb{G} -\frac{3}{4}\mb D & \le C (\l+\eps_1+\delta_0) \mathcal{G}^S +C \deltas^{4/3} ( \deltar^{4/3}+ \delta_C^{4/3}) e^{-C\deltas t}  \\
&\quad  -\frac{\alpha^*}{4} \deltas \int_0^1 w^2 dy- \frac{1}{2} (\mb G_1+ \mb G_2) .
\end{aligned}
\end{align*}
Finally, using
\[
 \deltas \int_0^1 w^2 dy = \int_\bbr |(v^S)^{-\mb X}_\x| |u-\bar u|^2 d\x,
\]
and
\begin{align*}
\begin{aligned}
\int_\bbr |(v^S)^{-\mb X}_\x| (v-\bar v)^2 d\xi &\leq 2\int_\bbr |(v^S)^{-\mb X}_\x| \Big|(v-\bar v)+\frac{u-\bar u}{\sigma^*}\Big|^2 d\xi + 2\int_\bbr |(v^S)^{-\mb X}_\x| \Big(\frac{u-\bar u}{\sigma^*}\Big) ^2d\xi \\
&\leq C\frac{\deltas}{\l}\mb G_1 + C \int_\bbr   |(v^S)^{-\mb X}_\x|  \big(u-\bar u\big)^2d\xi,
\end{aligned}
\end{align*}
and
\begin{align*}
\begin{aligned}
\int_\bbr (\theta-\bar \theta)^2 d\xi &\leq 2\int_\bbr  |(v^S)^{-\mb X}_\x| \Big|(\theta-\bar\theta)-\frac{(\gamma-1)\theta^*}{v^*\s^*}(u-\bar u)\Big|^2 d\xi + C\int_\bbr   |(v^S)^{-\mb X}_\x| (u-\bar u)^2d\xi\\
&\leq C\frac{\deltas}{\l}\mb G_2 + C \int_\bbr   |(v^S)^{-\mb X}_\x|  \big(u-\bar u\big)^2d\xi,
\end{aligned}
\end{align*}
we have
\[
-\frac{\deltas}{2M} |\dot{\mb X}|^2 + \mb B_1+\mb B_2-\mb{G} -\frac{3}{4}\mb D \le - C\int_\bbr|(v^S)^{-\mb X}_\x| |(v-\bar v, u-\bar u, \theta-\bar \theta)|^2 d\x +C \deltas^{4/3} ( \deltar^{4/3}+ \delta_C^{4/3}) e^{-C\deltas t} ,
\]
which implies the desired estimate in Lemma \ref{lem-sharp}.
\end{proof}

\subsection{Proof of Lemma \ref{lem-zvh}}
First of all, it follows from \eqref{ineq-1} and \eqref{XY} to have
\begin{align*}
\begin{aligned}
\frac{d}{dt}\int_{\bbr} a^{-\mb X}\bar \theta \eta\big(U|\bar U \big) d\xi  &\le -\frac{\deltas}{2M} |\dot{\mb X}|^2 +\mb B_1+\mb B_2 - \mb{G} -\frac{3}{4}\mb D \\
&\quad -\frac{\deltas}{2M} |\dot{\mb X}|^2 +\dot{\mb{X}}\sum_{i=4}^6\mb{Y}_i  +\sum_{i=3}^5 \mb B_i +\mb S_1+\mb S_2 - C \mathcal{G}^R  -\frac{1}{4} \mb D.
\end{aligned}
\end{align*}
Then, using Lemma \ref{lem-sharp} and Young's inequality, there exists a constant $C>0$ such that
\begin{align}
\begin{aligned}\label{ineq-2}
\frac{d}{dt}\int_{\bbr} a^{-\mb X}\bar \theta \eta\big(U|\bar U \big) d\xi  &\le -C(\mathcal{G}^S+\mathcal{G}^R) + C ( \deltar^{4/3}+ \delta_C^{4/3})  \deltas^{4/3} e^{-C\deltas t} \\
&\quad -\frac{\deltas}{4M} |\dot{\mb X}|^2 +\frac{C}{\deltas}\sum_{i=4}^6 |\mb{Y}_i|^2  +\sum_{i=3}^5 \mb B_i +\mb S_1+\mb S_2   -\frac{1}{4} \mb D.
\end{aligned}
\end{align}
In what follows, we will control the bad term on the right-hand side of \eqref{ineq-2}.

\noindent$\bullet$ {\bf Estimates on the terms $\mb Y_i~(i=4, 5, 6)$:}
Using \eqref{app-phi} and \eqref{theta-s},
we have
\[
\di |(\mb Y_4, \mb Y_5)|\le C\int_\bbr |(v^S)^{-\mb{X}}_\xi| |(v-\bar v , \theta-\bar\theta)|^2d\xi.
\]
In addition, since Lemma \ref{lemma1.3} and \eqref{apri-ass} yields
\[
 |(\mb Y_4, \mb Y_5)|\le \deltas \int_\bbr |(v-\bar v , \theta-\bar\theta)|^2d\xi \le \deltas^2\eps_1^2,
\]
we have
\[
\frac{C}{\deltas}|(\mb{Y}_4, \mb Y_5) |^2 \le C\deltas \eps_1^2\mathcal{G}^S.
\]
Similarly, we have
\[
\begin{array}{ll}
\di \frac{C}{\deltas}|\mb{Y}_6|^2 \le \frac{C}{\deltas} \left(\int_\bbr  |a^{-\mb{X}}_\xi| |(v-\bar v, u-\bar u, \theta-\bar\theta)|^2 d\xi\right)^2\\
\di 
\di \qquad\qquad  \le \frac{C\lam^2}{\deltas^3} \left(\int_\bbr |(v^S)_\xi|  |(v-\bar v, u-\bar u, \theta-\bar\theta)|^2 d\xi\right)^2\\
\di \qquad\qquad  \le \frac{C\lam^2}{\deltas} \|(v-\bar v, u-\bar u, \theta-\bar\theta)\|_{L^2(\bbr)}^2 \int_\bbr |(v^S)^{-\mb{X}}_\xi| |(v-\bar v, u-\bar u, \theta-\bar\theta)|^2d\xi\\
\di \qquad\qquad  \le C\eps_1^2  \int_\bbr |(v^S)^{-\mb{X}}_\xi| |(v-\bar v, u-\bar u, \theta-\bar\theta)|^2 d\xi\le C\eps_1^2 \mathcal{G}^S.
\end{array}
\]
Thus,
\beq\label{yis}
\frac{C}{\deltas}\sum_{i=4}^6 |\mb{Y}_i|^2 \le  C \eps_1\mathcal{G}^S.
\eeq

\noindent$\bullet$ {\bf Estimates on the terms $\mb B_i~(i=3,4,5)$:}
First, observe that
\begin{align*}
\begin{aligned}
\mb B_3(U) &\le C \int_\bbr \! |a_\xi^{-\mb X} | |u-\bar u| \Big(  |(u-\bar u)_\xi| + |\bar u_\xi| |v-\bar v| \Big) d\xi \\
&\quad + C \int_\bbr \! |a_\xi^{-\mb X} | |\theta-\bar \theta| \Big(  |(\theta-\bar \theta)_\xi| + |\bar \theta_\xi| |v-\bar v| \Big) d\xi \\
&\le \frac{1}{80} \mb D +  C\lambda\deltas \int_\bbr \! |(v^S)_\xi^{-\mb X} |^2 |(u-\bar u, \theta-\bar \theta) |^2 d\xi  + C \int_\bbr ( |\bar u_\xi|^2 + |\bar \theta_\xi|^2)  |v-\bar v|^2 d\xi \\
&\le  \frac{1}{80}  \mb D + C\deltas \mathcal{G}^S  + C \int_\bbr ( |\bar u_\xi|^2 + |\bar \theta_\xi|^2)  |v-\bar v|^2 d\xi.
\end{aligned}
\end{align*}
Since Lemma \ref{lemma1.2}, \eqref{vc-pe} and \eqref{shock-base} yield
\begin{align}
\begin{aligned} \label{b3model}
& \int_\bbr ( |\bar u_\xi|^2 + |\bar \theta_\xi|^2)  |v-\bar v|^2 d\xi \\
&\quad\le C\int_\bbr \Big[ |(u^R)_\xi|^2+|(u^C)_\xi|^2+|(u^S)^{-\mb X}_\xi|^2 + |(\theta^R)_\xi|^2+|(\theta^C)_\xi|^2+|(\theta^S)^{-\mb X}_\xi|^2\Big] |v-\bar v|^2 \\
&\quad\le C\deltar \mathcal{G}^R+C\deltas \mathcal{G}^S+ C(\delta_C)^2 (1+t)^{-1}\int_\bbr  e^{-\frac{2C_1|\xi+\sigma t|^2}{1+t}}|v-\bar v|^2 d\xi ,
\end{aligned}
\end{align}
we have
\[
\mb B_3(U)  \le  \frac{1}{40}  \mb D + C\deltar \mathcal{G}^R+C\deltas \mathcal{G}^S + C(\delta_C)^2 (1+t)^{-1}\int_\bbr  e^{-\frac{2C_1|\xi+\sigma t|^2}{1+t}}|v-\bar v|^2 d\xi .
\]
Likewise, we have
\begin{align*}
\begin{aligned}
\mb B_4 &\le  \int_\bbr \Big[ |\bar u_\xi| |v-\bar v|  |(u-\bar u)_\xi| +  |\theta-\bar \theta| \Big(  |(\theta-\bar \theta)_\xi| + |\bar \theta_\xi| \Big) \Big(  |(\theta-\bar \theta)_\xi| + |\bar \theta_\xi| |v-\bar v| \Big) \Big]d\xi \\
&\quad +  \int_\bbr \Big[  |(\theta-\bar \theta)_\xi| |\bar \theta_\xi| |v-\bar v| +|\theta-\bar\theta| \Big( |(u-\bar u)_\xi|^2 +  |\bar u_\xi| |(u-\bar u)_\xi| + |\bar u_\xi|^2 |v-\bar v|  \Big) \Big] d\xi ,
\end{aligned}
\end{align*}
which together with \eqref{smp1} yields
\begin{align*}
\begin{aligned}
\mb B_4
&\le \frac{1}{80} \mb D  + C \int_\bbr ( |\bar u_\xi|^2 + |\bar \theta_\xi|^2)  |(v-\bar v,\theta-\bar\theta)|^2 d\xi \\
&\le   \frac{1}{80}  \mb D + C\deltar \mathcal{G}^R+C\deltas \mathcal{G}^S  + C(\delta_C)^2 (1+t)^{-1}\int_\bbr  e^{-\frac{2C_1|\xi+\sigma t|^2}{1+t}}  |(v-\bar v,\theta-\bar\theta)|^2 d\xi.
\end{aligned}
\end{align*}
For $\mb B_6$, since $u^R_\xi \sim v^R_\xi$, we have
\[
\mb B_5 \le   C(\delta_0+\eps_1) (\mathcal{G}^R+ \mathcal{G}^S )+ C \dc(1+t)^{-1} \int_\bbr  e^{-\frac{C_1 |\xi+\sigma t|^2}{1+t}} |(v-\bar v,\theta-\bar\theta)|^2 d\xi.
\]
Therefore, 
\beq\label{bis}
\sum_{i=3}^5 \mb B_i  \le  \frac{1}{20}\mb D + C(\delta_0+\eps_1) (\mathcal{G}^R+ \mathcal{G}^S )  + C\delta_C (1+t)^{-1}\int_\bbr  e^{-\frac{C_1|\xi+\sigma t|^2}{1+t}} |(v-\bar v,\theta-\bar\theta)|^2 d\xi .
\eeq

\noindent$\bullet$ {\bf Estimates on the terms $\mb S_i~(i=1,2)$:}
First, by \eqref{QR}, for $i=1,2,$ 
\beq\label{Q-est}
\begin{array}{ll}
\di |Q^R_i| \le C \Big[|u^R_{\xi\xi}|+|u^R_{\xi}||v^R_{\xi}|+|\theta^R_{\xi\xi}|+|\theta^R_{\xi}||v^R_{\xi}|+|u^R_{\xi}|^2\Big]\\[3mm]
\di \qquad\ \le C  \Big[\big|\big(u^R_{\xi\xi},\theta^R_{\xi\xi}\big)\big|+\big|\big(v^R_{\xi},u^R_{\xi},\theta^R_{\xi}\big)\big|^2\Big],
\end{array}
\eeq
and then
$$
\|Q^R_i\|_{L^1(\bbr)}\leq C\Big[\big\|\big(u^R_{\xi\xi},\theta^R_{\xi\xi}\big)\big\|_{L^1}+\big\|\big(v^R_{\xi},u^R_{\xi},\theta^R_{\xi}\big)\big\|_{L^2}^2\Big].
$$
Notice that by Lemma \ref{lemma1.2},
\[
\big\|\big(u^R_{\xi\xi},\theta^R_{\xi\xi}\big)\big\|_{L^1} \le \left\{ \begin{array}{ll}
     C \deltar&\quad\mbox{if } 1+t \le \deltar^{-1}\\
      C \frac{1}{1+t} &\quad\mbox{if }  1+t \ge \deltar^{-1} , \\
        \end{array} \right.
\]
and
\beq\label{dr2}
\big\|\big(v^R_{\xi},u^R_{\xi},\theta^R_{\xi}\big)\big\|_{L^2} \le \left\{ \begin{array}{ll}
    C  \deltar&\quad\mbox{if } 1+t \le \deltar^{-1}\\
    C  \deltar^{1/2}\frac{1}{(1+t)^{1/2}} &\quad\mbox{if }  1+t \ge \deltar^{-1} , \\
        \end{array} \right.
\eeq
Thus,
\beq\label{v21}
\int_0^\infty \|Q_i^R\|_{L^1}^{4/3}ds \le C\big(\deltar^{1/3}+\deltar^{5/3}\big)\le C\deltar^{1/3},\quad i=1,2.
\eeq
By \eqref{QC12}, it holds that
\beq\label{vc-21}
\|Q_1^C\|_{L^2}\le C \delta_C (1+t)^{-\frac 54},\qquad \|Q_2^C\|_{L^2}\le C \delta_C (1+t)^{-\frac 74}.
\eeq
Therefore,
\begin{align*}
\begin{aligned}
&\mb S_1+\mb S_2=\int_\bbr a^{\mb X}Q_1(u-\bar u)  d\x+ \int_\bbr  a^{\mb X}Q_2 (1-\frac{\bar \theta}{\theta})d\x \\
&\leq C\int_\bbr \big(|Q^I_1|+|Q_1^R|+|Q_1^C|\big)|u-\bar u| d\x+ C\int_\bbr  \big(|Q^I_2|+|Q_2^R|+|Q_2^C|\big)|\theta-\bar \theta|d\x\\
&\le C \big(\|Q_1^I\|_{L^2}+\|Q_1^C\|_{L^2}\big) \|u-\bar u\|_{L^2}+C\|Q_1^R\|_{L^1(\bbr)} \|u-\bar u\|_{L^\infty(\bbr)}\\
&\qquad + C \big(\|Q_2^I\|_{L^2}+\|Q_2^C\|_{L^2}\big) \|\theta-\bar \theta\|_{L^2}+C\|Q_2^R\|_{L^1(\bbr)} \|\theta-\bar \theta\|_{L^\infty(\bbr)} ,
\end{aligned}
\end{align*}
which together with Lemma \ref{lemma2.2} and \eqref{vc-21}
\begin{align}\label{S12-e}
\begin{aligned}
&\mb S_1+\mb S_2\\
&\leq C\big[\deltas (\deltar+\delta_C) e^{-C \deltas t}+\deltar\delta_Ce^{-Ct}+\delta_C (1+t)^{-\frac 54}\big]\big(\|u-\bar u\|_{L^2}+\|\theta-\bar \theta\|_{L^2}\big)\\
&\quad + C\|(Q_1^R, Q_2^R)\|_{L^1}\|(u-\bar u, \theta-\bar\theta)_\x\|_{L^2}^{\frac 12}\|(u-\bar u, \theta-\bar\theta)\|_{L^2}^{\frac 12}\\
&\leq C\big[\deltas (\deltar+\delta_C) e^{-C \deltas t}+\deltar\delta_Ce^{-Ct}+\delta_C (1+t)^{-\frac 54}\big]\|(u-\bar u,\theta-\bar \theta)\|_{L^2}\\
&\quad +\frac{1}{20} \mb D  +C\|(Q_1^R, Q_2^R)\|_{L^1}^{\frac 43}\|(u-\bar u, \theta-\bar\theta)\|_{L^2}^{\frac 23}.
\end{aligned}
\end{align}

\noindent$\bullet$ {\bf \underline{Conclusion}:} From \eqref{ineq-2}, \eqref{yis}, \eqref{bis} and \eqref{S12-e}, with the smallness of $\delta_0, \eps_1$,  we have
\begin{align*}
\begin{aligned}
&\frac{d}{dt}\int_{\bbr} a^{-\mb X}\bar \theta \eta\big(U|\bar U \big) d\xi  \\
&\le -\frac{\deltas}{4M} |\dot{\mb X}|^2  -\frac{C}{2}(\mathcal{G}^S+\mathcal{G}^R) -\frac{1}{8} \mb D  + C ( \deltar^{4/3}+ \delta_C^{4/3})  \deltas^{4/3} e^{-C\deltas t}   \\
&\quad + C\delta_C (1+t)^{-1}\int_\bbr  e^{-\frac{C_1|\xi+\sigma t|^2}{1+t}} |(v-\bar v,\theta-\bar\theta)|^2 d\xi \\
&\quad + C\big[\deltas (\deltar+\delta_C) e^{-C \deltas t}+\deltar\delta_Ce^{-Ct}+\delta_C (1+t)^{-\frac 54}\big] \|(u-\bar u,\theta-\bar \theta)\|_{L^\infty(0,T;L^2(\bbr))}  \\
&\quad   +C\|(Q_1^R, Q_2^R)\|_{L^1}^{\frac 43}  \|(u-\bar u,\theta-\bar \theta)\|_{L^\infty(0,T;L^2(\bbr))}^{\frac 23}.
\end{aligned}
\end{align*}
Integrating the above inequality over $[0,t]$ for any $t\le T$, we have
\begin{align*}
\begin{aligned}
&\sup_{t\in[0,T]}\int_{\bbr}  \eta\big(U|\bar U \big) d\xi +\deltas\int_0^t|\dot{\mb{X}}|^2 d\tau +\int_0^t (\mathcal{G}^S+\mathcal{G}^R+ \mb D) d\tau\\
&\quad\le C\int_{\bbr} \eta\big(U_0|\bar U(0,\xi) \big) d\xi  + C ( \deltar^{4/3}+ \delta_C^{4/3})  \deltas^{1/3} \\
&\qquad   +C(\deltar+\delta_C) \|(u-\bar u,\theta-\bar \theta)\|_{L^\infty(0,T;L^2(\bbr))}+C\deltar^{1/3}  \|(u-\bar u,\theta-\bar \theta)\|_{L^\infty(0,T;L^2(\bbr))}^{\frac 23}    \\
&\qquad  +C\delta_C \int_0^t(1+\tau)^{-1}\int_\bbr  e^{-\frac{C_1|\xi+\sigma \tau|^2}{1+\tau}}|(v-\bar v,\theta-\bar \theta)|^2 d\xi d\tau .
\end{aligned} 
\end{align*}
Then, using Young's inequality with the fact that
$$
\|U-\bar U\|^2_{L^2(\bbr)} \sim \int_{\bbr}  \eta\big(U|\bar U \big) d\xi, \quad \forall t\in[0,T],
$$
we have
\begin{align}
\begin{aligned} \label{last-ecom}
&\sup_{t\in[0,T]}\|U-\bar U\|^2_{L^2(\bbr)} +\deltas\int_0^t|\dot{\mb{X}}|^2 d\tau +\int_0^t (\mathcal{G}^S+\mathcal{G}^R+ \mb D) d\tau\\
&\quad\le \|U_0-\bar U(0,\cdot)\|^2_{L^2(\bbr)} + C\delta_0^{1/2} \\
&\qquad  + C\delta_0 \int_0^t(1+\tau)^{-1}\int_\bbr  e^{-\frac{C_1|\xi+\sigma \tau|^2}{1+\tau}}|(v-\bar v,\theta-\bar \theta)|^2 d\xi d\tau .
\end{aligned} 
\end{align}
Finally, using the following Lemma \ref{cw-lemma} with $\delta_0\ll1$, we have
\begin{align*}
\begin{aligned}
&\sup_{t\in[0,T]}\|U-\bar U\|^2_{L^2(\bbr)} +\deltas\int_0^t|\dot{\mb{X}}|^2 d\tau +\int_0^t (\mathcal{G}^S+\mathcal{G}^R+ \mb D) d\tau\\
&\quad\le \|U_0-\bar U(0,\cdot)\|^2_{L^2(\bbr)} + C\delta_0^{1/2} +C\int_0^t\|(v-\bar v)_\xi \|^2_{L^2(\bbr)} d\tau  ,
\end{aligned} 
\end{align*}
which completes the proof of Lemma \ref{lem-zvh}.

\begin{lemma}\label{cw-lemma}
It holds that
\beq\label{cwe}
\begin{array}{ll}
\di  \int_0^t(1+\tau)^{-1}\int_\bbr  e^{-\frac{C_1|\xi+\sigma \tau|^2}{1+\tau}}|(v-\bar v, u-\bar u, \theta-\bar \theta)|^2 d\xi d\tau \\[3mm]
\di \le  C\sup_{t\in[0,T]}\|U-\bar U\|^2_{L^2(\bbr)}  +C\deltas\int_0^t|\dot{\mb X}(\tau)|^2 d\tau+C\int_0^t (\mathcal{G}^S+\mathcal{G}^R+ \mb D) d\tau  \\[3mm]
\di \quad+ C\int_0^t\|(v-\bar v)_\xi \|^2_{L^2(\bbr)} d\tau +C\delta_0^{\frac13}.
\end{array}
\eeq
\end{lemma}
\begin{remark}
In the above estimate \eqref{last-ecom}, Lemma \ref{cw-lemma} was used  to control the last term in terms of the $v$ and $\theta$ variables only. However, Lemma  \ref{cw-lemma} also provides the estimate on the $u$ variable: $\di \int_0^t(1+\tau)^{-1}\int_\bbr  e^{-\frac{C_1|\xi+\sigma \tau|^2}{1+\tau}}|u-\bar u|^2 d\xi d\tau$, which will be used in Section \ref{sec-vu}.
\end{remark}

\begin{proof}
The proof uses the same argument as in the paper \cite{HLM} that handles the stability of the composition of two rarefaction and a viscous contact wave. Here we should additionally handle the viscous shock wave with shift. However, since the proof is lengthy and mainly follows the argument of  \cite{HLM}, we present the proof in Appendix \ref{app-cont} for completeness. 
\end{proof}

\section{Proof of Proposition \ref{prop2}}\label{sec-vu}
\setcounter{equation}{0}

We here complete the proof of Proposition \ref{prop2}, by the following lemma.

\begin{lemma}\label{lem-u1}
Under the hypotheses of Proposition \ref{prop2}, there exists $C>0$ (independent of $\delta_0, \eps_1, T$) such that for all $ t\in (0,T]$,
\begin{align*}
\begin{aligned}
&\|(v-\bar v,u-\bar u,\theta-\bar\theta)(t,\cdot)\|_{H^1(\bbr)}^2 +\deltas\int_0^t|\dot{\mb{X}}(\tau)|^2 d\tau \\
&+\int_0^t \left( \mathcal{G}^S(U) + \mathcal{G}^R(U) + \|(v-\bar v,u-\bar u,\theta-\bar\theta)_\x\|^2_{L^2(\bbr)}+ \|(u-\bar u,\theta-\bar\theta)_{\x\x}\|^2_{L^2(\bbr)}\right) d\tau\\
&\le C \|(v-\bar v,u-\bar u,\theta-\bar\theta)(0,\cdot) \|_{H^1(\bbr)}^2 + C\delta_0^{1/2}.
\end{aligned}
\end{align*}
\end{lemma}
\begin{proof}
As in the proof of Lemma \ref{cw-lemma}, we use the following notations
$$
\phi(t,\xi):=v(t,\xi)-\bar v(t,\xi),\quad  \psi (t,\xi):=u(t,\xi)-\bar u(t,\xi), \quad \vartheta(t,\xi) :=\theta(t,\x)-\bar \theta(t,\xi),
$$
and 
\[
W(t,\x):=(1+t)^{-\frac12} e^{-\frac{\beta|\xi+\s t|^2}{1+t}}, \quad \beta :=\frac{C_1}{2}.
\]
First, it follows from \eqref{NS-1} and \eqref{bar-system} that (as in \eqref{psiphire})
\beq\label{psiphire2}
\phi_t -\s \phi_\xi - \dot{\mb X}(t) (v^S)^{-\mb X}_\xi - \psi_\xi =0,
\eeq
and
\beq\label{psi-equ}
\psi_t-\s \psi_\xi -\dot{\mb{X}}(t)(u^S)^{-\mb{X}}_\xi+(p-\bar p)_\xi=\mu\left(\frac{u_\xi}{v} - \frac{\bar u_\xi}{\bar v}\right)_\xi -Q_1.
\eeq
Differentiating \eqref{psiphire2} w.r.t. $\xi$ and multiplying the result by $\mu\phi_\xi$, we have
\beq\label{psiphire3}
\mu\Big(\frac{\phi_\xi^2}{2}\Big)_t -\s \mu\Big(\frac{\phi_\xi^2}{2}\Big)_\xi - \dot{\mb X}(t) (v^S)^{-\mb X}_{\xi\xi}\mu\phi_\xi = \mu\phi_\xi\psi_{\xi\xi}.
\eeq
Multiplying \eqref{psi-equ}  by $-v\phi_{\xi}$ yields that
$$
\begin{array}{ll}
\di -v\phi_{\xi}\psi_t+\s v\phi_{\xi} \psi_\xi + v\phi_{\xi} \dot{\mb{X}}(t)(u^S)^{-\mb{X}}_\xi
\\[3mm]
\di \quad =v\phi_{\xi} (p-\bar p)_\xi- v\phi_{\xi} \mu\left(\frac{u_\xi}{v} - \frac{\bar u_\xi}{\bar v}\right)_\xi +v\phi_{\xi} Q_1.
\end{array}
$$
Then, using the above equation and \eqref{psiphire3} together with
\beq\label{pdiff}
(p-\bar p)_\xi=\frac{R\vartheta_\x}{v}-\frac{R\theta\phi_\x}{v^2}+R\bar\theta_\x\left(\frac 1v-\frac 1{\bar v}\right)+R\bar v_\x \left(\frac{\bar\theta}{\bar v^2}-\frac{\theta}{v^2}\right),
\eeq
and
\beq\label{mudiff}
\left(\frac{u_\xi}{v} - \frac{\bar u_\xi}{\bar v}\right)_\xi= \frac{\psi_{\x\x}}{v} -  \frac{\psi_{\x}v_\xi}{v^2} + \bar u_{\x\x}\left(\frac 1v-\frac 1{\bar v}\right)
+ \bar u_\x \left(\frac{\bar v_\xi}{\bar v^2} - \frac{v_\xi}{v^2} \right),
\eeq
we have
$$
\begin{array}{ll}
\di 
\big(\frac{\mu \phi_\x^2}{2}-v\psi\phi_\x\big)_t-\s \big(\frac{\mu \phi_\x^2}{2}-v\psi\phi_\x\big)_\x+\frac{R\theta\phi_\x^2}{v}+(v_t-\s v_\x)\psi\phi_\x+\big(v\psi\phi_t-\s v\psi\phi_\x\big)_\x\\[3mm]
\di\quad - \big(v_\x\psi+v\psi_\x) \big(\phi_t-\s\phi_\x\big)+\dot{\mb{X}}(t)\phi_\xi\big[v(u^S)^{-\mb{X}}_\xi-\mu(v^S)^{-\mb{X}}_{\x\x}\big]\\[3mm]
\di  =R\vartheta_\x\phi_\x+Rv\phi_\x\bar\theta_\x\left(\frac 1v-\frac 1{\bar v}\right)+Rv\phi_\x\bar v_\x \left(\frac{\bar\theta}{\bar v^2}-\frac{\theta}{v^2}\right) -\mu v\phi_\x\bar u_{\x\x}\left(\frac 1v-\frac 1{\bar v}\right)\\[3mm]
\di\quad  -\mu v\phi_\x \bar u_\x \left(\frac{\bar v_\xi}{\bar v^2} - \frac{v_\xi}{v^2} \right) + \mu\frac{\phi_\x \psi_{\x}v_\xi}{v} +v\phi_\x Q_1.
\end{array}
$$
Integrating the above equation over $\mathbb{R}\times [0,t]$ with respect to $\x$ and $t$, we have
\[
\begin{array}{ll}
\di 
\frac{\mu}{2} \int_\bbr \phi_\x^2 d\xi +\int_0^t\int_\bbr \frac{R\theta\phi_\x^2}{v} d\x d\tau\leq C \|\phi_{0\x}\|^2_{L^2(\bbr)} + \int_\bbr v\psi\phi_\x d\xi -  \int_\bbr v_0\psi_0\phi_{0\x} d\xi  \\[3mm]
\di\quad +C\int_0^t\int_\bbr |(v_t-\s v_\x)\psi\phi_\x|d\x d\tau+C\int_0^t\int_\bbr |\big(v_\x\psi+v\psi_\x) \big(\phi_t-\s\phi_\x\big)|d\x d\tau\\[3mm]
\di\quad +C\int_0^t|\dot{\mb{X}}(\tau)|\int_\bbr \big|\phi_\xi \big[v(u^S)^{-\mb{X}}_\xi-\mu(v^S)^{-\mb{X}}_{\x\x}\big]\big|d\x d\tau+C\int_0^t\int_\bbr |R\vartheta_\x\phi_\x|d\x d\tau  \\[3mm]
\di\quad 
+C\int_0^t\int_\bbr \bigg| Rv\phi_\x\bar\theta_\x\left(\frac 1v-\frac 1{\bar v}\right)+Rv\phi_\x\bar v_\x \left(\frac{\bar\theta}{\bar v^2}-\frac{\theta}{v^2}\right) -\mu v\phi_\x\bar u_{\x\x}\left(\frac 1v-\frac 1{\bar v}\right)  \\[3mm]
\di\qquad\qquad\ \  \quad +\mu v\phi_\x \bar u_\x \left(\frac{\bar v_\xi}{\bar v^2} - \frac{v_\xi}{v^2} \right) \bigg|d\x d\tau  
+C\int_0^t\int_\bbr \left| \mu\frac{\phi_\x \psi_{\x}v_\xi}{v} \right|d\x d\tau \\[3mm]
\di\quad  +C\int_0^t\int_\bbr |v\phi_\x Q_1|d\x d\tau .
\end{array}
\]
In addition, using
\begin{align*}
\begin{aligned}
\int_\bbr v\psi\phi_\x d\xi -  \int_\bbr v_0\psi_0\phi_{0\x} d\xi &\le C(\|\psi\|_{L^2(\bbr)} \|\phi_{\xi}\|_{L^2(\bbr)} +\|\psi_0\|_{L^2(\bbr)}\|\phi_{0\xi}\|_{L^2(\bbr)})\\
&\le \frac{\mu}{4}\|\phi_{\xi}\|_{L^2(\bbr)}^2 + C \|\psi\|_{L^2(\bbr)}^2 + \|(\phi_{0\x},\psi_0)\|^2_{L^2(\bbr)},
\end{aligned} 
\end{align*}
we have
\[
\|\phi_\x(t,\cdot)\|^2_{L^2(\bbr)}+\int_0^t \|\phi_\x(\tau,\cdot)\|^2_{L^2(\bbr)} d\tau\leq C\Big[\|(\phi_{0\x},\psi_0)\|^2_{L^2(\bbr)}+\|\psi(t,\cdot)\|^2_{L^2(\bbr)}\Big] +\sum_{j=1}^7 L_j ,
\]
where
\begin{align*}
\begin{aligned}
&L_1 := C\int_0^t\int_\bbr |(v_t-\s v_\x)\psi\phi_\x|d\x d\tau,\\
&L_2 := C\int_0^t\int_\bbr |\big(v_\x\psi+v\psi_\x) \big(\phi_t-\s\phi_\x\big)|d\x d\tau, \\
&L_3 := C\int_0^t|\dot{\mb{X}}(\tau)|\int_\bbr \big|\phi_\xi \big[v(u^S)^{-\mb{X}}_\xi-\mu(v^S)^{-\mb{X}}_{\x\x}\big]\big|d\x d\tau, \\
&L_4 := C\int_0^t\int_\bbr |R\vartheta_\x\phi_\x|d\x d\tau, \\
&L_5 := C\int_0^t\int_\bbr \bigg| Rv\phi_\x\bar\theta_\x\left(\frac 1v-\frac 1{\bar v}\right)+Rv\phi_\x\bar v_\x \left(\frac{\bar\theta}{\bar v^2}-\frac{\theta}{v^2}\right) -\mu v\phi_\x\bar u_{\x\x}\left(\frac 1v-\frac 1{\bar v}\right)   \\
&\qquad \quad\qquad +\mu v\phi_\x \bar u_\x \left(\frac{\bar v_\xi}{\bar v^2} - \frac{v_\xi}{v^2} \right) \bigg|d\x d\tau  , \\
&L_6 := C\int_0^t\int_\bbr \left| \mu\frac{\phi_\x \psi_{\x}v_\xi}{v} \right|d\x d\tau , \\
&L_7 :=  C\int_0^t\int_\bbr |v\phi_\x Q_1|d\x d\tau .
\end{aligned} 
\end{align*}
Now we estimate $L_j,~j=1,2,\cdots,7$. First, using $\eqref{NS-1}_1$ and the same estimate as in \eqref{b3model}, we have
$$
\begin{array}{ll}
\di 
|L_1|\le C\int_0^t\int_\bbr |u_\x\psi\phi_\x|d\x d\tau\leq C\int_0^t\int_\bbr \big(|\psi_\x|+|\bar u_\x|\big)|\psi||\phi_\x|d\x d\tau\\[3mm]
\di \qquad \leq C\eps_1\int_0^t \big(\|\psi_\x\|^2_{L^2(\bbr)}+\|\phi_\x\|^2_{L^2(\bbr)}\big)d\tau+\frac18 \int_0^t\|\phi_\x\|^2_{L^2(\bbr)} d\tau+C\int_0^t  \int_\bbr |\bar u_\x|^2|\psi|^2 d\x d\tau\\[3mm]
\di\qquad \le C\eps_1\int_0^t \big(\|\psi_\x\|^2_{L^2(\bbr)}+\|\phi_\x\|^2_{L^2(\bbr)}\big)d\tau+\frac18 \int_0^t\|\phi_\x\|^2_{L^2(\bbr)} d\tau +C\int_0^t \Big(\deltas \mathcal{G}^S+\deltar \mathcal{G}^R\Big) d\tau \\
\di \qquad \quad +C(\delta_C)^2 \int_0^t \int_\bbr W^2 \psi^2 d\x d\tau ,
\end{array}
$$
where the last term is obtained by $(1+t)^{-2} e^{-\frac{2C_1|\xi+\sigma t|^2}{1+t}} \le W^2 $. \\
Likewise, we use \eqref{psiphire} and $v_\xi=\phi_\xi + \bar v_\xi$ to have
$$
\begin{array}{ll}
\di 
|L_2|\le C\int_0^t\int_\bbr \big( (|\phi_\xi| + |\bar v_\xi|) |\psi|+|\psi_\x| \big) \big| \dot{\mb{X}}(t)(v^S)^{-\mb{X}}_\xi+\psi_\xi\big|d\x d\tau\\[3mm]
\di \qquad \leq C\eps_1\int_0^t\|\phi_\x\|^2_{L^2(\bbr)} d\tau+C\int_0^t \|\psi_\x\|^2_{L^2(\bbr)} d\tau +C\delta_0 \delta_S\int_0^t|\dot{\mb{X}}(\tau)|^2 d\tau\\[3mm]
\di \qquad \quad +C\delta_0\int_0^t \Big( \mathcal{G}^S+ \mathcal{G}^R\Big) d\tau +C\delta_0 \int_0^t\int_\bbr W^2|\psi|^2 d\x d\tau,
\end{array}
$$
and
$$
\begin{array}{ll}
\di 
|L_5|\leq C\int_0^t\int_\bbr \big|\phi_\x\big|\big|(\bar\theta_\x, \bar v_\x, \bar u_{\x\x}, \bar u_\x \bar v_\x) \big|\big|(\phi,\vartheta) \big|d\x d\tau
+ C\int_0^t\int_\bbr |\bar u_\x||\phi_\xi| (|\phi_\x| + |\bar v_\x| |\phi| )d\x d\tau\\[3mm]
\di \qquad \leq \frac 18 \int_0^t\|\phi_\x\|^2_{L^2(\bbr)} d\tau +C\delta_0\int_0^t \Big( \mathcal{G}^S+ \mathcal{G}^R\Big) d\tau +C\delta_0 \int_0^t\int_\bbr W^2|(\phi,\vartheta)|^2 d\x d\tau.
\end{array}
$$
We easily have
$$
\begin{array}{ll}
\di 
|L_3|\leq  C\deltas \int_0^t|\dot{\mb{X}}(\tau)|^2 d\tau+ C\delta_0\int_0^t \|\phi_\x\|^2_{L^2(\bbr)} d\tau,
\end{array}
$$
and
$$
|L_4|\leq \frac 18 \int_0^t\|\phi_\x\|^2_{L^2(\bbr)} d\tau+C\int_0^t\|\vartheta_\x\|^2_{L^2(\bbr)} d\tau.
$$
For $L_6$, we use the interpolation inequality to have
$$
\begin{array}{ll}
\di 
|L_6|\leq C\int_0^t\int_\bbr |\phi_\x| |\psi_\x| |\phi_\x +\bar v_\x| d\x d\tau\\[3mm]
\di \qquad \leq C \int_0^t \|\psi_{\x\x}\|_{L^2(\bbr)} \|\psi_\x\|_{L^2(\bbr)} \|\phi_\x\|^2_{L^2(\bbr)} d\tau +C\int_0^t\int_\bbr |\bar v_\x| |\phi_\x| |\psi_\x| d\x d\tau   \\
\di \qquad \leq C\eps_1^2 \int_0^t \|\psi_{\x\x}\|_{L^2(\bbr)}^2 d\tau+C(\delta_0+\eps_1^2) \int_0^t\|(\phi_\x,\psi_\x)\|^2_{L^2(\bbr)} d\tau.
\end{array}
$$
Since \eqref{Q-est} and Lemma \ref{lemma1.2} (with the same estimates as in \eqref{dr2}) imply that for each $i=1, 2$,
\beq\label{qr12}
\int_0^t\|Q^R_i\|^2_{L^2(\bbr)}d\tau\leq C\int_0^t\Big[\big\|\big(u^R_{\xi\xi},\theta^R_{\xi\xi}\big)\big\|^2_{L^2(\bbr)}+\big\|\big(v^R_{\xi},u^R_{\xi},\theta^R_{\xi}\big)\big\|_{L^4}^4\Big]d\tau\leq C\deltar,
\eeq
we use Lemma \ref{lemma2.2} and \eqref{vc-21} to have
\beq\label{L7}
\begin{array}{ll}
\di 
|L_7|\leq C\int_0^t\int_\bbr|\phi_\x||Q_1| d\x d\tau \leq  \frac 18 \int_0^t\|\phi_\x\|^2_{L^2(\bbr)} d\tau+C\int_0^t  \|Q_1\|^2_{L^2(\bbr)} d\tau\\[3mm]
\di \qquad \leq \frac 18 \int_0^t\|\phi_\x\|^2_{L^2(\bbr)} d\tau+C\int_0^t  \big[\|Q^I_1\|^2_{L^2(\bbr)}+\|Q^R_1\|^2_{L^2(\bbr)}+\|Q^C_1\|^2_{L^2(\bbr)}\big] d\tau\\[3mm]
\di \qquad \leq  \frac 18 \int_0^t\|\phi_\x\|^2_{L^2(\bbr)} d\tau+C\int_0^t\big[C\delta_0\deltas  e^{-C \deltas \tau}+C\delta_0^2 e^{-C\tau}+\delta_C(1+\tau)^{-\frac 52}\big]d\tau+C\deltar\\[3mm]
\di \qquad \leq  \frac 18 \int_0^t\|\phi_\x\|^2_{L^2(\bbr)} d\tau+C\delta_0 .
\end{array}
\eeq
Therefore, combining the above estimates, we have
\begin{align*}
\begin{aligned}
&\|\phi_\x \|^2_{L^2(\bbr)}+\int_0^t \|\phi_\x\|^2_{L^2(\bbr)} d\tau \\
&\leq C \|(\phi_{0\x},\psi_0)\|^2_{L^2(\bbr)} + C\|\psi \|^2_{L^2(\bbr)} +C\int_0^t \|(\psi_\x,\vartheta_\x)\|^2_{L^2(\bbr)} d\tau +C\eps_1^2 \int_0^t \|\psi_{\x\x}\|_{L^2(\bbr)}^2 d\tau \\
&\quad+ C\deltas \int_0^t|\dot{\mb{X}}(\tau)|^2 d\tau  +C\delta_0\int_0^t \Big( \mathcal{G}^S+ \mathcal{G}^R\Big) d\tau +C\delta_0 \underbrace{\int_0^t\int_\bbr W^2|(\phi,\psi,\vartheta)|^2 d\x d\tau}_{=:\mb W} +C\delta_0.
\end{aligned}
\end{align*} 
Applying Lemma \ref{cw-lemma} to the above term $\mb W$ that is the same as the left-hand side of \eqref{cwe}, and using $\mb D \le C\|(\psi_\x,\vartheta_\x)\|^2_{L^2(\bbr)}$, we have
\begin{align}\label{phifinal}
\begin{aligned}
&\|\phi_\x \|^2_{L^2(\bbr)}+\int_0^t \|\phi_\x\|^2_{L^2(\bbr)} d\tau \\
&\leq C \|(\phi_{0\x},\psi_0)\|^2_{L^2(\bbr)} + C\|(\phi,\psi,\vartheta) \|^2_{L^2(\bbr)} +C\int_0^t \|(\psi_\x,\vartheta_\x)\|^2_{L^2(\bbr)} d\tau +C\eps_1^2 \int_0^t \|\psi_{\x\x}\|_{L^2(\bbr)}^2 d\tau \\
&\quad+ C\deltas \int_0^t|\dot{\mb{X}}(\tau)|^2 d\tau  +C\delta_0\int_0^t \Big( \mathcal{G}^S+ \mathcal{G}^R\Big) d\tau +C\delta_0.
\end{aligned}
\end{align}

To estimate $\psi_\x$, multiplying the equation \eqref{psi-equ} by $-\psi_{\xi\xi}$ and integrating the result w.r.t. $\xi$, we have
\begin{align*}
&\frac{d}{dt}\int_{\bbr} \frac{|\psi_\xi|^2}{2} d\xi=    -\dot{\mb{X}}(t) \int_\bbr (u^S)^{-\mb{X}}_\xi \psi_{\xi\xi} d\xi + \int_\bbr (p-\bar p)_\xi \psi_{\xi\xi} d\xi   \\
&\qquad - \mu\int_\bbr \left(\frac{u_\xi}{v}- \frac{\bar u_\xi}{\bar v}\right)_\xi  \psi_{\xi\xi} d\xi + \int_\bbr Q_1 \psi_{\xi\xi} d\xi:= K_1 + K_2 +K_3 +K_4. 
\end{align*} 
First, we find a good term $$\mb D_\psi:= \int_\bbr \frac{\mu}{v} |\psi_{\xi\xi}|^2 d\xi$$  from $K_3$ as follows: (using \eqref{mudiff})
\begin{align*}
K_3 &= - \int_\bbr \frac{\mu}{v} |\psi_{\xi\xi}|^2 d\xi - \mu\int_\bbr \left(\frac{1}{v}\right)_\xi \psi_\xi \psi_{\xi\xi} d\xi  - \mu\int_\bbr \bar u_{\xi\xi} \left(\frac{1}{v} - \frac{1}{\bar v}\right)  \psi_{\xi\xi} d\xi  \\
&\quad - \mu\int_\bbr \bar u_\xi \left(\frac{1}{v}- \frac{1}{\bar v}\right)_\xi  \psi_{\xi\xi} d\xi=: - \mb D_\psi + K_{31}+ K_{32}+ K_{33}.
\end{align*}
Using $v_\x =\phi_\x +\bar v_\x$ and the interpolation inequality with \eqref{apri-ass}, we have
\begin{align*}
|K_{31}| &\le \|\phi_\xi\|_{L^2} \|\psi_\xi\|_{L^\infty}  \|\psi_{\xi\xi}\|_{L^2}  + \|\bar v_\xi\|_{L^\infty}   \|\psi_{\xi}\|_{L^2} \|\psi_{\xi\xi}\|_{L^2} \\
&\le C\eps_1 \|\psi_\xi\|_{L^2}^{1/2} \|\psi_{\xi\xi}\|_{L^2}^{1/2}  \|\psi_{\xi\xi}\|_{L^2}  + C\delta_0 \|\psi_{\xi}\|_{L^2} \|\psi_{\xi\xi}\|_{L^2} \\
&\le C(\eps_1+\delta_0) \big( \|\psi_{\xi}\|_{L^2}^2 + \|\psi_{\xi\xi}\|_{L^2}^2 \big) \le \frac18 \mb D_\psi + C(\eps_1+\delta_0) \mb D.
\end{align*} 
Using $|\bar u^R_{\xi\xi}| \leq C |\bar u^R_{\xi}| $ (by Lemma \ref{lemma1.2}), we have
\[
\begin{array}{ll}
\di 
|K_{32}| \le C\int_\bbr (|u^R_{\xi\xi}|+ |u^C_{\xi\xi}|+|(u^S)^{-\mb X}_{\xi\xi}| ) |\phi| |\psi_{\xi\xi}| d\xi\\[3mm]
\di \qquad\ \  \leq \frac18 \mb D_\psi + C\delta_0\Big( \mathcal{G}^S + \mathcal{G}^R+ \int_\bbr W^2 \phi^2 d\xi \Big),
 \end{array}
\]
and
\begin{align*}
|K_{33}| &\le C\int_\bbr |\bar u_\x| \big(|\phi_\xi| + |\bar v_\x||\phi| \big) |\psi_{\xi\xi}| d\xi \\
&\le \frac18 \mb D_\psi + C\delta_0 \|\phi_\x\|^2_{L^2(\bbr)} + C\delta_0\Big( \mathcal{G}^S + \mathcal{G}^R+ \int_\bbr W^2 \phi^2 d\xi \Big) .
\end{align*} 
We easily have
\[
|K_1| \le  |\dot{\mb{X}}(t)| \|(u^S)^{-\mb{X}}_\xi\|_{L^2(\bbr)}\|\psi_{\xi\xi}\|_{L^2(\bbr)} \le \deltas |\dot{\mb{X}}(t)|^2 + C\deltas \mb D_\psi \le \deltas|\dot{\mb{X}}(t)|^2 +\frac18 \mb D_\psi,
\]
and
\[
|K_4|   \le \frac18 \mb D_\psi +C\|Q_1\|^2_{L^2(\bbr)}.
\]
Using \eqref{pdiff},
\[
\begin{array}{ll}
\di |K_2|\le \frac18 \mb D_\psi + C \int_\bbr |(p-\bar p)_\xi |^2 d\xi \\
\di \qquad \leq \frac18 \mb D_\psi + C\big( \|\phi_\x\|^2_{L^2(\bbr)}+\mb D\big)+C\int_\bbr |(\bar v_\x, \bar\theta_\x)|^2 |(\phi,\vartheta)|^2 d\xi\\[3mm]
\di \qquad \leq \frac18 \mb D_\psi+ C\big( \|\phi_\x\|^2_{L^2(\bbr)}+\mb D\big)+C\delta_0\Big( \mathcal{G}^S + \mathcal{G}^R+ \int_\bbr W^2 \phi^2 d\xi \Big).
\end{array}
\]
Therefore, we find that 
\begin{align}\label{psi-2nd}
\begin{aligned}
&\frac{d}{dt}\int_{\bbr} \frac{|\psi_\xi|^2}{2} d\xi +\frac18 \mb D_\psi \leq  \deltas|\dot{\mb{X}}|^2 + C \big( \|\phi_\x\|^2_{L^2(\bbr)}+\mb D \big)+C\|Q_1\|^2_{L^2(\bbr)}\\[3mm]
&\quad + C\delta_0\Big( \mathcal{G}^S + \mathcal{G}^R+ \int_\bbr W^2 \phi^2 d\xi \Big).
\end{aligned}
\end{align}
Similarly, we estimate $\vartheta_\x$. Multiplying the  equation \eqref{v-c-1}  by $-\vartheta_{\xi\xi}$ and integrating the result w.r.t. $\xi$, we have
\begin{align*}
&\frac{d}{dt}\int_{\bbr} \frac{|\vartheta_\xi|^2}{2} d\xi= \frac{R}{\gamma-1}\dot{\mb{X}}(t) \int_\bbr (\theta^S)^{-\mb{X}}_\xi \vartheta_{\xi\xi} d\xi + \int_\bbr (pu_\x-\bar p\bar u_\xi) \vartheta_{\xi\xi} d\xi   \\
&\qquad \underbrace{- \int_\bbr \kappa\left(\frac{\theta_\xi}{v}- \frac{\bar \theta_\xi}{\bar v}\right)_\xi  \vartheta_{\xi\xi} d\xi}_{=:K_5}-\int_\bbr \mu\left(\frac{u_\xi^2}{v}- \frac{\bar u_\xi^2}{\bar v}\right) \vartheta_{\xi\xi} d\xi + \int_\bbr Q_2 \vartheta_{\xi\xi} d\xi.
\end{align*}
As above, we find a good term $$\mb D_\theta:= \int_\bbr \frac{\kappa}{v} |\vartheta_{\xi\xi}|^2 d\xi$$  from $K_5$ as follows:
\begin{align*}
K_7 &= - \int_\bbr \frac{\kappa}{v} |\vartheta_{\xi\xi}|^2 d\xi - \kappa\int_\bbr \left(\frac{1}{v}\right)_\xi \vartheta_\xi \vartheta_{\xi\xi} d\xi  - \kappa\int_\bbr \bar \theta_{\xi\xi} \left(\frac{1}{v} - \frac{1}{\bar v}\right)  \vartheta_{\xi\xi} d\xi  \\
&\quad - \kappa\int_\bbr \bar \theta_\xi \left(\frac{1}{v}- \frac{1}{\bar v}\right)_\xi  \vartheta_{\xi\xi} d\xi.
\end{align*}
The all terms above can be estimated in a similar way as before.
Therefore, we get
\begin{align}\label{vartheta-2nd}
\begin{aligned}
\frac{d}{dt}\int_{\bbr} \frac{|\vartheta_\xi|^2}{2} d\xi +\frac18 \mb D_\theta &\leq \deltas|\dot{\mb{X}}|^2 + C \big( \|\phi_\x\|^2_{L^2(\bbr)}+\mb D \big)+C\|Q_1\|^2_{L^2(\bbr)}\\[3mm]
&\quad + C\delta_0\Big( \mathcal{G}^S + \mathcal{G}^R+ \int_\bbr W^2 \phi^2 d\xi \Big).
\end{aligned}
\end{align}
Adding \eqref{psi-2nd} and \eqref{vartheta-2nd} and then integrating the result over $[0,t]$, together with using \eqref{cwe} as in \eqref{phifinal}, we find that for some constants $C_*, C>0$,
\beq\label{fineq}
\begin{array}{ll}
\di \|(\psi_{\x}, \vartheta_{\x})\|^2_{L^2(\bbr)}+\int_0^t\|(\psi_{\x\x}, \vartheta_{\x\x})\|^2_{L^2(\bbr)}d\tau\\
\di \leq C\|(\psi_{0\x}, \vartheta_{0\x})\|^2_{L^2(\bbr)}+C\deltas\int_0^t|\dot{\mb{X}}|^2 d\tau + C_* \int_0^t\|\phi_\x\|^2_{L^2(\bbr)} d\tau +  C \int_0^t \|(\psi_\x,\vartheta_\x)\|^2_{L^2(\bbr)} d\tau \\[3mm]
\di\quad +C\int_0^t\|(Q_1,Q_2)\|^2_{L^2(\bbr)}d\tau+ C\delta_0\int_0^t \Big( \mathcal{G}^S+ \mathcal{G}^R\Big) d\tau + C\delta_0\|(\phi,\psi,\vartheta) \|^2_{L^2(\bbr)} + C\delta_0.
\end{array}
\eeq
Notice that by \eqref{qr12}, Lemma \ref{lemma2.2} and \eqref{vc-21} with the same estimates as in \eqref{L7},
\[
\int_0^t\|(Q_1,Q_2)\|^2_{L^2(\bbr)}d\tau \le C\delta_R.
\]
Finally, multiplying \eqref{fineq} by $\frac{1}{2C_*}$ and then adding the result and \eqref{phifinal}, we have
\begin{align}\label{pptfinal}
\begin{aligned}
&\|(\phi_\x, \psi_{\x}, \vartheta_{\x}) \|^2_{L^2(\bbr)}+\int_0^t \|\phi_\x\|^2_{L^2(\bbr)} d\tau +\int_0^t\|(\psi_{\x\x}, \vartheta_{\x\x})\|^2_{L^2(\bbr)}d\tau\\
&\leq C \|(\phi_{0\x},\psi_{0\x}, \vartheta_{0\x}, \psi_0)\|^2_{L^2(\bbr)} + C_{**}\|(\phi,\psi,\vartheta) \|^2_{L^2(\bbr)} +C_{**}\int_0^t \|(\psi_\x,\vartheta_\x)\|^2_{L^2(\bbr)} d\tau  \\
&\quad+ C_{**}\deltas \int_0^t|\dot{\mb{X}}(\tau)|^2 d\tau  +C\delta_0\int_0^t \Big( \mathcal{G}^S+ \mathcal{G}^R\Big) d\tau +C\delta_0.
\end{aligned}
\end{align}  
Finally, multiplying \eqref{pptfinal} by $\frac{1}{2C_{**}}$ and then adding the result with the $L^2$-estimates \eqref{esthv} of Lemma \ref{lem-zvh}, we have the desired estimates.
\end{proof}

\begin{appendix}
\setcounter{equation}{0}
\section{Relative entropy}  \label{app-ent}
Let $U=(v,u,E)$ with $E=e+\frac{u^2}{2}$ and $e=\frac{R}{\gamma-1}\theta + \mbox{const}$. We here compute the relative entropy defined by the entropy 
\[
s(U):=R\log v + \frac{R}{\gamma-1} \log\theta .
\]
Note that the entropy is computed from the Gibbs relation $\theta ds = de + p dv$ and \eqref{state}. \\
Using the Gibbs relation and $E=e+\frac{u^2}{2}$, we have
\[
\theta ds = dE - udu + p dv,
\]
and so,
\[
\nabla_U s(U) = \left( \frac{p}{\theta}, -\frac{u}{\theta}, \frac{1}{\theta} \right).
\]
Thus, for any $\bar U:=(\bar v,\bar u,\bar E)$ with $\bar E:=\bar e+\frac{\bar u^2}{2}$, $\bar e=\frac{R}{\gamma-1}\bar \theta + \mbox{const}$, and $\bar p = \frac{R\bar\theta}{\bar v}$,
\begin{align*}
(-s)(U|\bar U) &= - s(U) + s(\bar U) + \nabla_U s (\bar U)\cdot (U-\bar U)\\
&=  - R\log \frac{v}{\bar v} -\frac{R}{\gamma-1} \log\frac{\theta}{\bar\theta} +  \frac{\bar p}{\bar \theta} (v-\bar v) - \frac{\bar u}{\bar \theta} (u-\bar u) + \frac{1}{\bar \theta} (E-\bar E)\\
&=  R  \left(\frac{v}{\bar v} -1- \log\frac{v}{\bar v}\right) +\frac{R}{\gamma-1} \left(\frac{\theta}{\bar\theta} -1- \log\frac{\theta}{\bar\theta}\right)  + \frac{(u-\bar u)^2}{2\bar\theta}.
\end{align*}

\section{Sharp estimate for the diffusion}  \label{app-est}

\begin{lemma} 
Let 
\[
p^S:=\frac{R\theta^S}{v^S},\,  p_+:=\frac{R\theta_+}{v_+}, \, p^*:=\frac{R\theta^*}{v^*}.
\]
Then, it holds that
\beq\label{p-est}
\left|\frac{p^S-p_+}{v^S-v_+}-\frac{p^S-p^*}{v^S-v^*}-\sigma^*\alpha^* \frac{\mu R\gamma}{\mu R\gamma+\kappa (\gamma-1)^2} (v_+-v^*)\right| \le C\delta_S^2,
\eeq
where
$$
\alpha^*:= \frac{\gamma(\gamma+1) p^*}{2(v^*)^2\s^*}.
$$
\end{lemma}
\begin{proof} 
First of all, by Taylor expansion with $ \frac{\partial^2 p^S}{\partial (\theta^S)^2}=0$, 
$$
\begin{array}{ll}
\di p^S-p_+=\frac{\partial p^S}{\partial v^S}\Big|_{(v_+,\theta_+)}(v^S-v_+)+\frac{\partial p^S}{\partial \theta^S}\Big|_{(v_+,\theta_+)}(\theta^S-\theta_+)+\frac12\frac{\partial^2 p^S}{\partial (v^S)^2}\Big|_{(v_+,\theta_+)}(v^S-v_+)^2 \\[3mm]
\di\qquad +\frac{\partial^2 p^S}{\partial v^S\partial \theta^S}\Big|_{(v_+,\theta_+)}(v^S-v_+)(\theta^S-\theta_+)+O((\delta_S)^3),
\end{array}
$$
and 
\[
\begin{array}{ll}
\di p^S-p^*=\frac{\partial p^S}{\partial v^S}\Big|_{(v^*,\theta^*)}(v^S-v^*)+\frac{\partial p^S}{\partial \theta^S}\Big|_{(v^*,\theta^*)}(\theta^S-\theta^*)+\frac12\frac{\partial^2 p^S}{\partial (v^S)^2}\Big|_{(v^*,\theta^*)}(v^S-v^*)^2\\[3mm]
\di\qquad +\frac{\partial^2 p^S}{\partial v^S\partial \theta^S}\Big|_{(v^*,\theta^*)}(v^S-v^*)(\theta^S-\theta^*)+O((\delta_S)^3).
\end{array}
\]
Then, we have
\begin{align*}
 \frac{p^S-p_+}{v^S-v_+}-\frac{p^S-p^*}{v^S-v^*} &= \left( \frac{\partial p^S}{\partial v^S}\Big|_{(v_+,\theta_+)}-\frac{\partial p^S}{\partial v^S}\Big|_{(v^*,\theta^*)} \right)\\
 & + \left(\frac{\partial p^S}{\partial \theta^S}\Big|_{(v_+,\theta_+)} \frac{\theta^S-\theta_+}{v^S-v_+}  -\frac{\partial p^S}{\partial \theta^S}\Big|_{(v^*,\theta^*)} \frac{\theta^S-\theta^*}{v^S-v^*}  \right)\\
  & + \left(\frac12\frac{\partial^2 p^S}{\partial (v^S)^2}\Big|_{(v_+,\theta_+)}(v^S-v_+)  -\frac12\frac{\partial^2 p^S}{\partial (v^S)^2}\Big|_{(v^*,\theta^*)}(v^S-v^*) \right)\\
  & + \left(\frac{\partial^2 p^S}{\partial v^S\partial \theta^S}\Big|_{(v_+,\theta_+)}(\theta^S-\theta_+)-\frac{\partial^2 p^S}{\partial v^S\partial \theta^S}\Big|_{(v^*,\theta^*)} (\theta^S-\theta^*) \right)+O((\delta_S)^2)\\  
&=: J_1+J_2+J_3+J_4+O((\delta_S)^2).
\end{align*}
Since
\begin{align*}
J_1 &=\frac{\partial^2 p^S}{\partial (v^S)^2}\Big|_{(v^*,\theta^*)} (v_+-v^*) + \frac{\partial^2 p^S}{\partial \theta^S\partial v^S}\Big|_{(v^*,\theta^*)}  (\theta_+ -\theta^*) +O((\delta_S)^2),\\
J_3 & = -\frac12\frac{\partial^2 p^S}{\partial (v^S)^2}\Big|_{(v^*,\theta^*)} (v_+-v^*) + \left(\frac12\frac{\partial^2 p^S}{\partial (v^S)^2}\Big|_{(v_+,\theta_+)} -\frac12\frac{\partial^2 p^S}{\partial (v^S)^2}\Big|_{(v^*,\theta^*)}\right)(v^S-v_+) \\
&= -\frac12\frac{\partial^2 p^S}{\partial (v^S)^2}\Big|_{(v^*,\theta^*)} (v_+-v^*) +O((\delta_S)^2),
\end{align*}
and
\begin{align*}
J_4 &= -\frac{\partial^2 p^S}{\partial v^S\partial \theta^S}\Big|_{(v^*,\theta^*)}    (\theta_+ -\theta^*)  + \left(\frac{\partial^2 p^S}{\partial v^S\partial \theta^S}\Big|_{(v_+,\theta_+)}-\frac{\partial^2 p^S}{\partial v^S\partial \theta^S}\Big|_{(v^*,\theta^*)}  \right)(\theta^S-\theta_+)\\
& =  -\frac{\partial^2 p^S}{\partial v^S\partial \theta^S}\Big|_{(v^*,\theta^*)}    (\theta_+ -\theta^*) +O((\delta_S)^2),
\end{align*}
we have
\begin{align*}
 \frac{p^S-p_+}{v^S-v_+}-\frac{p^S-p^*}{v^S-v^*} &= \left(\frac{\partial p^S}{\partial \theta^S}\Big|_{(v_+,\theta_+)} \frac{\theta^S-\theta_+}{v^S-v_+}  -\frac{\partial p^S}{\partial \theta^S}\Big|_{(v^*,\theta^*)} \frac{\theta^S-\theta^*}{v^S-v^*}  \right)\\
 &\quad + \frac{1}{2}\frac{\partial^2 p^S}{\partial (v^S)^2}\Big|_{(v^*,\theta^*)} (v_+-v^*) +O((\delta_S)^2).
\end{align*}
Thus, using
$$
 \frac{\partial p^S}{\partial \theta^S}=\frac{R}{v^S},\quad \frac{\partial^2 p^S}{\partial (v^S)^2}=\frac{2p^S}{(v^S)^2},
$$
we have
\begin{align*}
 \frac{p^S-p_+}{v^S-v_+}-\frac{p^S-p^*}{v^S-v^*} &= \left(\frac{R}{v_+}  \frac{\theta^S-\theta_+}{v^S-v_+}  - \frac{R}{v^*} \frac{\theta^S-\theta^*}{v^S-v^*}  \right) +\frac{p^*}{(v^*)^2} (v_+-v^*) +O((\delta_S)^2).
\end{align*}
Observe
\[
\frac{R}{v_+}\frac{\theta^S-\theta_+}{v^S-v_+}-\frac{R}{v^*}\frac{\theta^S-\theta^*}{v^S-v^*}=\frac{\theta^S-\theta_+}{v^S-v_+}\left(\frac{R}{v_+}-\frac{R}{v^*}\right) + \bigg(\frac{\theta^S-\theta_+}{v^S-v_+}-\frac{\theta^S-\theta^*}{v^S-v^*}\bigg)\frac{R}{v^*}.
\]
To estimate the first term of the right-hand side, we find from \eqref{theta-s} that
\begin{align*}
\Big| \theta^S-\theta_+ + \frac{(\gamma-1)p^*}{R} (v^S-v_+) \Big| \le  \int_\xi^\infty \Big| (\theta^S)_\xi + \frac{(\gamma-1)p^*}{R}(v^S)_\xi \Big| d\xi \le C \deltas^2,
\end{align*}
and so,
$$
\frac{\theta^S-\theta_+}{v^S-v_+}=- \frac{(\gamma-1)p^*}{R} +O(\delta_S).
$$
In addition, using
$$
\frac{R}{v_+}-\frac{R}{v^*}=-\frac{R}{(v^*)^2}(v_+-v^*)+O((\delta_S)^2),
$$
we have
\beq\label{p-estimate1}
\begin{array}{ll}
\di \frac{p^S-p_+}{v^S-v_+}-\frac{p^S-p^*}{v^S-v^*}=\bigg(\frac{\theta^S-\theta_+}{v^S-v_+}-\frac{\theta^S-\theta^*}{v^S-v^*}\bigg)\frac{R}{v^*}+ \frac{\gamma p^*}{(v^*)^2}(v_+-v^*)+O((\delta_S)^2).
\end{array}
\eeq

Now it remains to estimate 
\beq\label{theta-v}
\frac{\theta^S-\theta_+}{v^S-v_+}-\frac{\theta^S-\theta^*}{v^S-v^*}.
\eeq
For that, we consider the smooth function $\theta^S=\theta^S(v^S)$ as mentioned in the proof for \eqref{theta-s} of Lemma \ref{lemma1.3}.
Then, using the Taylor expansions of $\theta^S$ at points $v_+$ and $v^*$, we have
$$
\theta^S-\theta_+-\frac{d\theta^S}{dv^S}\bigg|_{v^S=v_+}(v^S-v_+)-\frac12\frac{d^2\theta^S}{d(v^S)^2}\bigg|_{v^S=v_+}(v^S-v_+)^2=O(|v^S-v_+|^3),
$$
and
$$
\theta^S-\theta^*-\frac{d\theta^S}{dv^S}\bigg|_{v^S=v^*}(v^S-v^*)-\frac12\frac{d^2\theta^S}{d(v^S)^2}\bigg|_{v^S=v^*}(v^S-v^*)^2=O(|v^S-v^*|^3).
$$
Thus,
$$
\frac{\theta^S-\theta_+}{v^S-v_+}-\frac{d\theta^S}{dv^S}\bigg|_{v^S=v_+}-\frac12\frac{d^2\theta^S}{d(v^S)^2}\bigg|_{v^S=v_+}(v^S-v_+)=O((\deltas)^2),
$$
and
$$
\frac{\theta^S-\theta^*}{v^S-v^*}-\frac{d\theta^S}{dv^S}\bigg|_{v^S=v^*}-\frac12\frac{d^2\theta^S}{d(v^S)^2}\bigg|_{v^S=v^*}(v^S-v^*)=O((\deltas)^2).
$$
Consequently,
\beq\label{theta-v-1}
\begin{array}{ll}
\di \Bigg|\frac{\theta^S-\theta_+}{v^S-v_+}-\frac{\theta^S-\theta^*}{v^S-v^*}+\frac{d\theta^S}{dv^S}\bigg|_{v^S=v^*}-\frac{d\theta^S}{dv^S}\bigg|_{v^S=v_+}\\[3mm]
\di\ \  +\frac12\frac{d^2\theta^S}{d(v^S)^2}\bigg|_{v^S=v^*}(v^S-v^*)-\frac12\frac{d^2\theta^S}{d(v^S)^2}\bigg|_{v^S=v_+}(v^S-v_+)\Bigg|=O((\deltas)^2).
\end{array}
\eeq
Using
$$
\frac{d\theta^S}{dv^S}\bigg|_{v^S=v^*}-\frac{d\theta^S}{dv^S}\bigg|_{v^S=v_+}-\frac{d^2\theta^S}{d(v^S)^2}\bigg|_{v^S=v_+}(v^*-v_+)=O((\deltas)^2),
$$
and
$$
\begin{array}{ll}
\di 
\frac12\bigg[\frac{d^2\theta^S}{d(v^S)^2}\bigg|_{v^S=v^*}(v^S-v^*)-\frac{d^2\theta^S}{d(v^S)^2}\bigg|_{v^S=v_+}(v^S-v_+)+\frac{d^2\theta^S}{d(v^S)^2}\bigg|_{v^S=v_+}(v^*-v_+)\bigg]\\[5mm]
\di =\frac12\bigg[\frac{d^2\theta^S}{d(v^S)^2}\bigg|_{v^S=v^*}(v^S-v^*)-\frac{d^2\theta^S}{d(v^S)^2}\bigg|_{v^S=v_+}(v^S-v^*)\bigg]=O((\deltas)^2),
\end{array}
$$
we have from \eqref{theta-v-1} that
$$
\begin{array}{ll}
\di \Bigg|\frac{\theta^S-\theta_+}{v^S-v_+}-\frac{\theta^S-\theta^*}{v^S-v^*}+\frac12\frac{d^2\theta^S}{d(v^S)^2}\bigg|_{v^S=v_+}(v^*-v_+)\Bigg|=O((\deltas)^2),
\end{array}
$$
which implies 
\beq\label{theta-v-2}
\begin{array}{ll}
\di \Bigg|\frac{\theta^S-\theta_+}{v^S-v_+}-\frac{\theta^S-\theta^*}{v^S-v^*}-\frac12\frac{d^2\theta^S}{d(v^S)^2}\bigg|_{v^S=v^*}(v_+-v^*)\Bigg|=O((\deltas)^2).
\end{array}
\eeq
Therefore, it remains to compute 
$$
\frac{d^2\theta^S}{d(v^S)^2}\bigg|_{v^S=v^*}.
$$
To this end, we use the following facts as in the proof for \eqref{theta-s} of Lemma \ref{lemma1.3}:
\beq\label{first-der}
\frac{d\theta^S}{dv^S}=\frac{\theta^S_\x}{v^S_\x}=\frac{\mu \sigma^2}{\kappa}\frac{\frac{R}{\gamma-1}(\theta^S-\theta^*)+p^*(v^S -v^*)-\frac12\sigma^2(v^S-v^*)^2}{(p^S-p^*)+\sigma^2(v^S-v^*)},
\eeq
\beq\label{elstar}
 \frac{d\theta^S}{dv^S}\bigg|_{v^S=v^*}=\lim_{v^S\to v^*} \frac{d\theta^S}{dv^S} =  \lim_{v^S\to v^*} \frac{\theta^S-\theta^*}{v^S-v^*} =:l^* <0,
\eeq
\beq\label{limit-p}
\lim_{v^S\rightarrow v^*}\frac{p^S-p^*}{v^S-v^*}=\lim_{v^S\rightarrow v^*}\Big(\frac{R}{v^S}\frac{\theta^S-\theta^*}{v^S-v^*}-\frac{p^*}{v^S}\Big)=\frac{Rl^*-p^*}{v^*},
\eeq
and
\beq\label{l*}
(l^*)^2-\left(\frac{p^*-\s^2v^*}{R}+\frac{\mu\s^2 v^*}{\kappa(\gamma-1)}\right)l^*-\frac{\mu\s^2\theta^*}{\kappa}=0.
\eeq
It holds from \eqref{first-der} that
\beq\label{second-der}
\begin{array}{ll}
\di 
\frac{d^2\theta^S}{d(v^S)^2}=\frac{d}{dv^S}\bigg(\frac{d\theta^S}{dv^S}\bigg)=\frac{\mu \sigma^2}{\kappa}\frac{\frac{R}{\gamma-1}\frac{d\theta^S}{dv^S}+p^*-\sigma^2(v^S-v^*)}{(p^S-p^*)+\sigma^2(v^S-v^*)}\\[4mm]
\di \qquad -\frac{\mu \sigma^2}{\kappa}\frac{\frac{R}{\gamma-1}(\theta^S-\theta^*)+p^*(v^S -v^*)-\frac12\sigma^2(v^S-v^*)^2}{\big[(p^S-p^*)+\sigma^2(v^S-v^*)\big]^2}\bigg(\frac{R\frac{d\theta^S}{dv^S}}{v^S}-\frac{R\theta^S}{(v^S)^2}+\s^2\bigg)\\[4mm]
\di \  =\frac{\mu \sigma^2}{\kappa}\frac{\frac{R}{\gamma-1}\frac{d\theta^S}{dv^S}+p^*-\sigma^2(v^S-v^*)}{(p^S-p^*)+\sigma^2(v^S-v^*)}-\frac{\frac{d\theta^S}{dv^S}}{(p^S-p^*)+\sigma^2(v^S-v^*)}\bigg(\frac{R\frac{d\theta^S}{dv^S}}{v^S}-\frac{R\theta^S}{(v^S)^2}+\s^2\bigg)\\[4mm]
\di\  =-\frac{R}{v^S[(p^S-p^*)+\sigma^2(v^S-v^*)]}\bigg[\bigg(\frac{d\theta^S}{dv^S}\bigg)^2-\bigg(\frac{p^S-\s^2v^S}{R} + \frac{\mu \s^2v^S}{\kappa(\gamma-1)}\bigg)\frac{d\theta^S}{dv^S}\\[4mm]
\di \qquad\qquad\qquad\qquad\qquad\qquad\quad\   -\frac{v^S}{R}\bigg(\frac{\mu\s^2p^*}{\kappa}-\frac{\mu\s^4}{\kappa}\big(v^S-v^*\big)\bigg)\bigg].
\end{array}
\eeq
Then, using \eqref{l*}, 
$$
\begin{array}{ll}
\di 
\frac{d^2\theta^S}{d(v^S)^2}=-\frac{R}{v^S\big[(p^S-p^*)+\sigma^2(v^S-v^*)\big]}\bigg[\bigg(\frac{d\theta^S}{dv^S}\bigg)^2-\bigg(\frac{p^S-\s^2v^S}{R}+\frac{\mu \s^2v^S}{\kappa(\gamma-1)}\bigg)\frac{d\theta^S}{dv^S}\\[4mm]
\di \qquad  -\frac{v^S}{R}\bigg(\frac{\mu\s^2p^*}{\kappa}-\frac{\mu\s^4}{\kappa}\big(v^S-v^*\big)\bigg)-\bigg((l^*)^2-\Big(\frac{p^*-\s^2v^*}{R}+\frac{\mu\s^2 v^*}{\kappa(\gamma-1)}\Big)l^*-\frac{\mu\s^2\theta^*}{\kappa} \bigg)\bigg].
\end{array}
$$
Moreover, since
\[
 -\frac{v^S}{R}\bigg(\frac{\mu\s^2p^*}{\kappa}-\frac{\mu\s^4}{\kappa}\big(v^S-v^*\big)\bigg) +\frac{\mu\s^2\theta^*}{\kappa}=-\frac{\mu\s^2\big(p^*-\s^2 v^S\big)}{R\kappa}\big(v^S-v^*\big),
\]
we have
$$
\begin{array}{ll}
\di 
v^S\big[(p^S-p^*)+\sigma^2(v^S-v^*)\big] \frac{d^2\theta^S}{d(v^S)^2} \\[4mm]
\di =- R\bigg[\bigg(\frac{d\theta^S}{dv^S}\bigg)^2-(l^*)^2-\Big(\frac{p^*-\s^2v^*}{R}+\frac{\mu\s^2 v^*}{\kappa(\gamma-1)}\Big)\Big(\frac{d\theta^S}{dv^S}-l^*\Big)\\[4mm]
\di \qquad  -\bigg(\frac{\big(p^S-p^*\big)-\s^2\big(v^S-v^*\big)}{R} + \frac{\mu \s^2\big(v^S-v^*\big)}{\kappa(\gamma-1)}\bigg)\frac{d\theta^S}{dv^S}-\frac{\mu\s^2\big(p^*-\s^2 v^S\big)}{R\kappa}\big(v^S-v^*\big)\bigg],
\end{array}
$$
and so,
$$
\begin{array}{ll}
\di 
v^S\Big(\frac{p^S-p^*}{v^S-v^*}+\sigma^2\Big)  \frac{d^2\theta^S}{d(v^S)^2} \\[4mm]
\di =-R\bigg[\frac{\frac{d\theta^S}{dv^S}-l^*}{v^S-v^*}\Big(\frac{d\theta^S}{dv^S}+l^*\Big)-\Big(\frac{p^*-\s^2v^*}{R}+\frac{\mu\s^2 v^*}{\kappa(\gamma-1)}\Big)\frac{\frac{d\theta^S}{dv^S}-l^*}{v^S-v^*}\\[5mm]
\di \qquad  -\bigg(\frac{\frac{p^S-p^*}{v^S-v^*}-\s^2}{R} + \frac{\mu \s^2}{\kappa(\gamma-1)}\bigg)\frac{d\theta^S}{dv^S}-\frac{\mu\s^2\big(p^*-\s^2 v^S\big)}{R\kappa}\bigg].
\end{array}
$$
Taking $v^S\rightarrow v^*$ on the both sides of the above equality and using \eqref{limit-p}, we have
$$
\begin{array}{ll}
\di 
\big(Rl^*-p^*+\s^2 v^*\big)\frac{d^2\theta^S}{d(v^S)^2}\bigg|_{v^S=v^*}= -R \bigg(2l^*-\frac{p^*-\s^2v^*}{R}-\frac{\mu\s^2 v^*}{\kappa(\gamma-1)}\bigg)\frac{d^2\theta^S}{d(v^S)^2}\bigg|_{v^S=v^*}\\[5mm]
\di \qquad \qquad\qquad\quad  +\bigg(\frac{Rl^*-p^*}{v^*}-\s^2 + \frac{R \mu \s^2}{\kappa(\gamma-1)}\bigg)l^* + \frac{\mu\s^2\big(p^*-\s^2 v^*\big)}{\kappa} ,
\end{array}
$$
equivalently,
$$
\begin{array}{ll}
\di 
\bigg(3Rl^*-2(p^*-\s^2v^*)-\frac{R\mu\s^2 v^*}{\kappa(\gamma-1)}\bigg)  \frac{d^2\theta^S}{d(v^S)^2}\bigg|_{v^S=v^*}\\[5mm]
\di \qquad \qquad\qquad\quad  =\bigg(\frac{Rl^*-p^*}{v^*}-\s^2 + \frac{R \mu \s^2}{\kappa(\gamma-1)}\bigg)l^* + \frac{\mu\s^2\big(p^*-\s^2 v^*\big)}{\kappa} ,
\end{array}
$$
Using \eqref{elstar} and $\sigma^2 = \frac{\gamma p^*}{v^*} + O(\deltas)$ by \eqref{sm1}, we have
\begin{align*}
3Rl^*-2(p^*-\s^2v^*)-\frac{R\mu\s^2 v^*}{\kappa(\gamma-1)} &= -3(\gamma-1)p^* -2(1-\gamma) p^* - \frac{R\mu\gamma p^*}{\kappa(\gamma-1)} + O(\deltas)\\
&= -\frac{\kappa(\gamma-1)^2 + R\mu\gamma}{\kappa(\gamma-1)}p^* + O(\deltas),
\end{align*}
and
\begin{align*}
\bigg(\frac{Rl^*-p^*}{v^*}-\s^2 + \frac{R \mu \s^2}{\kappa(\gamma-1)}\bigg)l^* + \frac{\mu\s^2\big(p^*-\s^2 v^*\big)}{\kappa}  = \frac{2\gamma(\gamma-1)(p^*)^2}{Rv^*} -\frac{\mu\gamma^2 (p^*)^2}{\kappa v^*}+ O(\deltas),
\end{align*}
which yields
\[
 \frac{d^2\theta^S}{d(v^S)^2}\bigg|_{v^S=v^*} = \frac{(\gamma-1)[R\mu\gamma^2 - 2\kappa\gamma(\gamma-1)]} {R[\kappa(\gamma-1)^2 + R\mu\gamma]} \frac{p^*}{v^*} + O(\deltas).
\]
This together with \eqref{p-estimate1} and \eqref{theta-v-2} yields
\begin{align*}
\frac{p^S-p_+}{v^S-v_+}-\frac{p^S-p^*}{v^S-v^*} &=\frac{1}{2} \frac{d^2\theta^S}{d(v^S)^2}\bigg|_{v^S=v^*} (v_+-v^*) \frac{R}{v^*}+ \frac{\gamma p^*}{(v^*)^2}(v_+-v^*)+O((\delta_S)^2)\\
&=  \frac{\gamma^2(\gamma+1)\mu R p^*}{2[\kappa(\gamma-1)^2 + R\mu\gamma] (v^*)^2} (v_+-v^*)+O((\delta_S)^2),
\end{align*}
which completes the proof.
\end{proof}

\section{Proof of Lemma \ref{cw-lemma} }  \label{app-cont}
For simplicity, let us introduce the following notations:
$$
\phi(t,\xi):=v(t,\xi)-\bar v(t,\xi),\quad  \psi (t,\xi):=u(t,\xi)-\bar u(t,\xi), \quad \vartheta(t,\xi) :=\theta(t,\x)-\bar \theta(t,\xi),
$$
and
\beq\label{W}
W(t,\x):=(1+t)^{-\frac12} e^{-\frac{\beta|\xi+\s t|^2}{1+t}}, \quad \beta :=\frac{C_1}{2},
\eeq
and
\beq\label{hH}
h(t,\x):=\int_{-\infty}^\x W(t,\zeta)d\zeta,\qquad H(t,\xi):=\int_{-\infty}^\x W^2(t,\zeta)d\zeta.
\eeq
Notice that $W(t,\xi)= \Phi (t, \xi+\s t)$ for the fundamental solution $\Phi$ of the heat equation $\Phi_t = \frac{1}{4\beta}\Phi_{xx}$, and so
it holds that
\beq\label{hW-r}
h_t-\s h_\x=\frac{1}{4\beta}h_{\x\x}=\frac{1}{4\beta} W_\x, \, \mbox{\rm and note that } h_\xi=W,\quad H_\xi=W^2 .
\eeq
We will show the two estimates:
\beq\label{b5e1}
\begin{array}{ll}
\di 
\int_0^t\int_\bbr\big(\bar p \phi+\frac{R}{\gamma-1} \vartheta\big)^2W^2d\xi d\tau\leq  C\sup_{t\in[0,T]}\|(\phi,\vartheta)(t,\cdot)\|^2_{L^2(\bbr)} \\[3mm]
\di  
\quad +C\deltas\int_0^t|\dot{\mb X}(\tau)|^2 d\tau+(\nu +C\delta_0)\int_0^t \int_\bbr W^2|(\phi,\vartheta)|^2d\x d\tau\\[3mm]
\di \quad  +C_\nu\int_0^t\|(\phi_\xi, \psi_\x, \vartheta_\xi)\|^2_{L^2(\bbr)} d\tau +C\int_0^t\big({\mb G}^S +{\mb G}^R\big) d\tau+C\delta_0^{\frac13},
\end{array}
\eeq
where $\nu$ is a small positive constant to be determined below, and $C_\nu$ depends on  $\nu$, and
\beq\label{b5e2}
\begin{array}{ll}
\di 
\int_0^t\int_\bbr\Big[\frac{(R\vartheta-\bar p \phi)^2}{2v}+\frac{(\gamma+1)\bar p\psi^2}{2}\Big] W^2d\xi d\tau \leq  C\sup_{t\in[0,T]}\|(\phi,\psi,\vartheta)(t,\cdot)\|^2_{L^2(\bbr)}\\[3mm]
\di\quad  +C\deltas\int_0^t|\dot{\mb X}(\tau)|^2 d\tau+C\delta_0\int_0^t\int_\bbr W^2|(\phi,\vartheta)|^2d\x d\tau\\[3mm]
\di \quad  +C\int_0^t\|(\phi_\xi, \psi_\x, \vartheta_\xi)\|^2_{L^2(\bbr)} d\tau +C(\delta_0+\eps_1)\int_0^t\big({\mb G}^S +{\mb G}^R\big) d\tau+C\deltar^{\frac13}.
\end{array}
\eeq
Then, the above estimates \eqref{b5e1} and \eqref{b5e2} imply the desired result \eqref{cwe}. \\
Indeed, since $F:= \bar p \phi+\frac{R}{\gamma-1} \vartheta$ satisfies
\[
R\vartheta -\bar p \phi = (\gamma-1) F -\gamma \bar p \phi,
\]
it follows from \eqref{b5e1} and \eqref{b5e2} that
\begin{align*}
\begin{aligned}
&\int_0^t\int_\bbr \Big[ F^2 + \frac{\big((\gamma-1) F -\gamma \bar p \phi\big)^2}{2v} \Big] W^2d\xi d\tau + \int_0^t\int_\bbr  \frac{(\gamma+1)\bar p\psi^2}{2} W^2d\xi d\tau \\
&\quad \le \mbox{[r.h.s. of \eqref{b5e1}]} + \mbox{[r.h.s. of \eqref{b5e2}]}.
\end{aligned} 
\end{align*}
Here, using 
\[
F^2 + \frac{\big((\gamma-1) F -\gamma \bar p \phi\big)^2}{2v} \ge C^{-1} (F^2 +\phi^2) \quad \mbox{ for some $C>0$ by Young's inequality},
\]
we have
\[
\int_0^t\int_\bbr (F^2 +\phi^2 +\psi^2) W^2 d\xi d\tau \le C(  \mbox{[r.h.s. of \eqref{b5e1}]} + \mbox{[r.h.s. of \eqref{b5e2}]} ),
\]
which together with $F^2 \ge C(\vartheta^2-\phi^2)$ implies
\[
\int_0^t\int_\bbr (\vartheta^2 +\phi^2 +\psi^2) W^2 d\xi d\tau  \le C(  \mbox{[r.h.s. of \eqref{b5e1}]} + \mbox{[r.h.s. of \eqref{b5e2}]} ).
\]
Now, use the smallness of $\delta_0, \eps_1$ and choose $\nu$ small enough such that the third terms on the r.h.s. of \eqref{b5e1} and \eqref{b5e2} can be absorbed into the above left-hand side, which completes the proof of \eqref{cwe}.
Therefore, it remains to prove \eqref{b5e1} and \eqref{b5e2}.

\noindent$\bullet$ {\bf Proof of  \eqref{b5e1}:} 
By the energy equations $\eqref{NS-1}_3$ and $\eqref{bar-system}_3$, we have
\beq\label{v-c-1}
\begin{array}{ll}
\di \frac{R}{\gamma-1} \vartheta_t-\frac{R\s}{\gamma-1} \vartheta_\x-\frac{R}{\gamma-1} \dot{\mb X}(t)(\theta^S)^{-\mb X}_\xi+(pu_\xi-\bar p\bar u_\x)\\[3mm]
\di \qquad\qquad  =\kappa \big( \frac{\theta _{ \x}
}{v}-\frac{\bar \theta_{ \x}}{\bar v} \big)_\x
 +  \mu \big( \frac{(u_{\x})^2}{v} -\frac{ (\bar u_\x)^2}{\bar v}
   \big)-Q_2.
\end{array}
\eeq
Observe that since the mass equations $\eqref{NS-1}_1$ and $\eqref{bar-system}_1$ yield
\beq\label{psiphire}
\psi_\xi = \phi_t -\s \phi_\xi - \dot{\mb X}(t) (v^S)^{-\mb X}_\xi,
\eeq
we have
\beq\label{v-c-2}
\begin{array}{ll}
\di pu_\xi-\bar p\bar u_\x=\bar p \psi_\x+(p-\bar p)\psi_\x+\bar u_\xi (p-\bar p)\\[3mm]
\di \qquad =\bar p (\phi_t-\s\phi_\x)-\dot{\mb X}(t)(v^S)^{-\mb X}_\xi\bar p+(p-\bar p)\psi_\x+\bar u_\xi (p-\bar p)\\[3mm]
\di \qquad =(\bar p \phi)_t-\s(\bar p\phi)_\x-\dot{\mb X}(t)(v^S)^{-\mb X}_\xi\bar p-(\bar p_t-\s\bar p_\x)\phi+(p-\bar p)\psi_\x+\bar u_\xi (p-\bar p).
\end{array}
\eeq
Substituting \eqref{v-c-2} into \eqref{v-c-1} yields that
\beq\label{v-c-3}
\begin{array}{ll}
\di \big(\bar p \phi+\frac{R}{\gamma-1} \vartheta\big)_t-\s\big(\bar p \phi+\frac{R}{\gamma-1} \vartheta\big)_\x-\dot{\mb X}(t)\big((v^S)^{-\mb X}_\xi\bar p+\frac{R}{\gamma-1} (\theta^S)^{-\mb X}_\xi\big)\\[3mm]
\di +(p-\bar p)\psi_\x~ =(\bar p_t-\s\bar p_\x)\phi-\bar u_\xi (p-\bar p)+\kappa \big( \frac{\theta _{ \x}
}{v}-\frac{\bar \theta_{ \x}}{\bar v} \big)_\x
 +  \mu \big( \frac{(u_{\x})^2}{v} -\frac{ (\bar u_\x)^2}{\bar v}
   \big)-Q_2.
\end{array}
\eeq
Use the notation $F:= \bar p \phi+\frac{R}{\gamma-1} \vartheta$ as above for simplicity.
Then, multiplying the equation \eqref{v-c-3} by $Fh^2$, and using \eqref{hW-r}, we have
\beq\label{v-c-31}
\begin{array}{ll}
\di \big[F^2\frac{h^2}{2}\big]_t-\big[F^2\frac{\s h^2}{2}\big]_\x-\frac{1}{4\beta}F^2hW_{\x}\\
\di\quad -\dot{\mb X}(t)\big((v^S)^{-\mb X}_\xi\bar p+\frac{R}{\gamma-1} (\theta^S)^{-\mb X}_\xi\big)Fh^2+(p-\bar p)\psi_\x F h^2\\[3mm]
\di =\kappa \big( \frac{\theta _{\x}
}{v}-\frac{\bar \theta_{ \x}}{\bar v} \big)_\x F h^2+\Big[(\bar p_t-\s\bar p_\x)\phi-\bar u_\xi (p-\bar p)+ \mu \big( \frac{(u_{\x})^2}{v} -\frac{ (\bar u_\x)^2}{\bar v}
   \big)-Q_2\Big] F h^2.
\end{array}
\eeq
Observe that
$$
\begin{array}{ll}
\di -\frac{1}{4\beta} F^2hW_{\x}=\big[-\frac{1}{4\beta}F^2hW\big]_\x  +\frac{1}{4\beta}F^2W^2+\frac{1}{4\beta}\big( F^2\big)_\x hW,
\end{array}
$$
where notice that the second term of the right-hand side is a good term as desired.\\
Substituting the above relation into \eqref{v-c-31} and then integrating the result over $\mathbb{R}\times [0,t]$, we have
\beq\label{v-c-4}
\begin{array}{ll}
\di  \frac{1}{4\beta}\int_0^t\int_\bbr F^2W^2d\xi d\tau=\int_\bbr F(0,\xi)^2\frac{h_0^2}{2} d\xi\\[4mm]
\di  ~  -\int_\bbr F^2\frac{h^2}{2} d\xi+\int_0^t\dot{\mb X}(\tau)\int_\bbr \big((v^S)^{-\mb X}_\xi\bar p+\frac{R}{\gamma-1} (\theta^S)^{-\mb X}_\xi\big) F h^2d\xi d\tau\\[4mm]
\di  ~ -\int_0^t\int_\bbr\frac{1}{4\beta}\big[ F^2\big]_\x hW d\xi d\tau
-\int_0^t\int_\bbr (p-\bar p)\psi_\x Fh^2d\xi d\tau\\[3mm]
\di ~ +\int_0^t\int_\bbr\Big[(\bar p_t-\s\bar p_\x)\phi-\bar u_\xi (p-\bar p)+ \mu \big( \frac{(u_{\x})^2}{v} -\frac{ (\bar u_\x)^2}{\bar v}
   \big)-Q_2\Big] F h^2d\xi d\tau\\[3mm]
\di ~
-\int_0^t\int_\bbr \kappa \big( \frac{\theta _{\x}
}{v}-\frac{\bar \theta_{ \x}}{\bar v} \big)\big[ F h^2\big]_\x d\xi d\tau:=\sum_{i=1}^7 J_i.
\end{array}
\eeq
Using the fact that $h$ is bounded as $|h(t,\xi)|\leq \int_\bbr W d\xi = \sqrt{\frac{\pi}{\beta}}$ for all $\xi, t$, we estimate the right hand sides of \eqref{v-c-4} one by one.
First we have
$$
|J_1+J_2|\leq C\big[\|(\phi_0, \vartheta_0)\|^2_{L^2(\bbr)}+\|(\phi, \vartheta)(t,\cdot)\|^2_{L^2(\bbr)}\big],
$$
and 
$$
\begin{array}{ll}
\di |J_3|\leq \deltas \int_0^t|\dot{\mb X}(\tau)|^2 d\tau +\frac{C}{\deltas}\int_0^t\int_\bbr |(v^S)^{-\mb X}_\xi|^2 |(\phi,\vartheta)|^2d\xi d\tau\\[3mm]
\di \qquad \leq \deltas\int_0^t|\dot{\mb X}(\tau)|^2 d\tau +C\deltas\int_0^t\int_\bbr |(v^S)^{-\mb X}_\xi| |(\phi,\vartheta)|^2d\xi d\tau\\[3mm]
\di \qquad =\deltas\int_0^t|\dot{\mb X}(\tau)|^2 d\tau +C\deltas\int_0^t \mathcal{G}^S d\tau.
\end{array}
$$
Since
$$
\begin{array}{ll}
\di 
J_4 \di = -\int_0^t\int_\bbr\frac{1}{2\beta}  FF_\x hW d\xi d\tau \\
\di \qquad \leq \frac{1}{16\beta} \int_0^t\int_\bbr F^2W^2d\xi d\tau+C\int_0^t\int_\bbr\big[|(\phi_\xi,  \vartheta_\xi)|^2+|\bar p_\x\phi|^2\big]d\xi d\tau ,
\end{array}
$$
and $|\bar p_\xi| \le C( |\bar\theta_\xi| +|\bar v_\xi| )$, we use the same estimates as in \eqref{b3model}, and
\[
(1+t)^{-1}  e^{-\frac{2C_1|\xi+\sigma t|^2}{1+t}} \le W^2,
\]
to have
$$
\begin{array}{ll}
\di 
J_4 \di  \leq \frac{1}{16\beta} \int_0^t\int_\bbr F^2W^2d\xi d\tau+C\int_0^t\int_\bbr|(\phi_\xi,  \vartheta_\xi)|^2 d\xi d\tau\\[3mm]
\di\qquad \quad +C\delta_0\int_0^t\int_\bbr \Big[ W^2|\phi|^2+\mathcal{G}^R+\mathcal{G}^S\Big] d\xi d\tau.
\end{array}
$$
For $J_5$, we first use \eqref{psiphire} to have
$$
\begin{array}{ll}
\di -(p-\bar p)\psi_\x F h^2=\frac{\bar p \phi-R\vartheta}{v}\big(\phi_t-\s\phi_\x-\dot{\mb X}(t)(v^S)^{-\mb X}_\x\big) F h^2\\
\di \qquad\qquad  =\frac{1}{v}\big[\gamma\bar p \phi-(\gamma-1)F \big]\big(\phi_t-\s\phi_\x\big) Fh^2  -\dot{\mb X}(t)(v^S)^{-\mb X}_\x\frac{\bar p \phi-R\vartheta}{v}Fh^2\\[3mm]
\di \qquad\qquad  =\frac{\gamma\bar p}{v}\big[(\frac{\phi^2}2)_t-\s (\frac{\phi^2}2)_\x\big]Fh^2-\frac{\gamma-1}{v} \big(\phi_t-\s\phi_\x\big)F^2h^2 -\dot{\mb X}(t)(v^S)^{-\mb X}_\x\frac{\bar p \phi-R\vartheta}{v}Fh^2\\[3mm]
\di \qquad\qquad  =: -\s \Big[ \frac{\gamma\bar p}{2v}\phi^2Fh^2-\frac{\gamma-1}{v}\phi F^2h^2  \Big]_\xi + \sum_{i=1}^5J_{5i},
\end{array}
$$
where
\begin{align*}
\begin{aligned}
& J_{51} : = \Big[\frac{\gamma\bar p}{2v}\phi^2Fh^2-\frac{\gamma-1}{v}\phi F^2h^2 \Big]_t,\\
& J_{52} : = -\Big[\frac{\gamma\bar p}{v}\phi^2F-\frac{\gamma-1}{v}\phi F^2\Big] h(h_t-\s h_\xi) ,\\
& J_{53} : =- \Big[\Big(\frac{\gamma\bar p}{v}F\Big)_t-\sigma \Big(\frac{\gamma\bar p}{v}F\Big)_\x\Big] \frac{\phi^2}{2}h^2,\\
& J_{54} : = \Big[\Big(\frac{\gamma-1}{v}F^2\Big)_t-\sigma \Big(\frac{\gamma-1}{v} F^2\Big)_\x\Big] \phi h^2,\\
& J_{55} : =-\dot{\mb X}(t)(v^S)^{-\mb X}_\x\frac{\bar p \phi-R\vartheta}{v}Fh^2 .
\end{aligned} 
\end{align*}
So, $J_5$ can be written as
$$
J_5=-\int_0^t\int_\bbr (p-\bar p)\psi_\x Fh^2 d\tau:=\sum_{i=1}^5 \int_0^t\int_\bbr J_{5i},
$$
where the integration of the first term of the right-hand side is zero.\\
We use \eqref{smp1} to have
$$
 \int_0^t\int_\bbr |J_{51}|\leq C\big[\|(\phi,\vartheta)(t,\cdot)\|^2_{L^2(\bbr)}+\|(\phi_0,\vartheta_0)\|^2_{L^2(\bbr)}\big]\leq C\sup_{t\in[0,T]}\|(\phi,\vartheta)\|^2_{L^2(\bbr)},
$$
Using \eqref{hW-r} with $|W_\xi| \le \frac{C}{1+t}$ for all $\xi, t$, we have
$$
\begin{array}{ll}
\di 
 \int_0^t\int_\bbr |J_{52}|\leq C\int_0^t (1+\tau)^{-1}\|\phi\|_{L^\infty}\|(\phi,\vartheta)\|^2_{L^2(\bbr)} d\tau\\
\di \qquad\  \leq C\int_0^t (1+\tau)^{-1}\|\phi\|_{L^2(\bbr)}^{\frac 12}\|\phi_\x\|_{L^2(\bbr)}^{\frac 12}\|(\phi,\vartheta)\|^2_{L^2(\bbr)} d\tau\\
\di \qquad\  \leq C\int_0^t\|\phi_\x\|^2_{L^2(\bbr)} d\tau +C\eps_1^{\frac43}\int_0^t(1+\tau)^{-\frac43}\|(\phi,\vartheta)\|^2_{L^2(\bbr)} d\tau\\
\di \qquad\ \leq C\int_0^t\|\phi_\x\|^2_{L^2(\bbr)} d\tau +C\eps_1^{\frac43}\sup_{t\in[0,T]}\|(\phi,\vartheta)\|^2_{L^2(\bbr)}.
\end{array}
$$
Using the equation \eqref{v-c-3} for $F$, we first have
$$
\begin{array}{ll}
\di 
J_{53}=-\int_0^t \int_\bbr\Big[F_t-\s F_\x\Big]\frac{\gamma\bar p\phi^2h^2}{2v} d\x d\tau\\[3mm]
\di \qquad -\int_0^t \int_\bbr\Big[(\frac{\bar p}{v})_t-\s(\frac{\bar p}{v})_\x)\Big]F\frac{\gamma\phi^2h^2}{2} d\x d\tau\\[3mm] 
\di \quad =-\int_0^t \int_\bbr\dot{\mb X}(\tau)\big[(v^S)^{-\mb X}_\xi\bar p+\frac{R}{\gamma-1} (\theta^S)^{-\mb X}_\xi\big]\frac{\gamma\bar p^2\phi^2 h^2}{2v} d\x d\tau\\[3mm]
\di \qquad +\int_0^t \int_\bbr\kappa \big( \frac{\theta _{ \x}
}{v}-\frac{\bar \theta_{ \x}}{\bar v} \big)\big( \frac{\gamma \bar p\phi^2h^2}{2v}\big)_\x d\x d\tau\\[3mm]
\di  \qquad  -\int_0^t \int_\bbr\Big[-(p-\bar p)\psi_\x -\bar u_\x(p-\bar p)
 + \mu \big( \frac{(u_{\x})^2}{v} -\frac{ (\bar u_\x)^2}{\bar v}
   \big)-Q_2\Big]\frac{\gamma \bar p\phi^2h^2}{2v} d\x d\tau\\[3mm]
\di \qquad +O(1)\int_0^t \int_\bbr\big[|\bar p_t-\s\bar p_\x|+|v_t-\s v_\x|\big] |(\phi, \vartheta)||\phi|^2 d\x d\tau.
\end{array}
$$
Using
$$
\begin{array}{ll}
\di \left|\int_0^t \int_\bbr(p-\bar p)\psi_\x\frac{\gamma\bar p^2\phi^2 h^2}{2v} d\x d\tau\right|\di \leq C\int_0^t \int_\bbr|\psi_\x|(|\phi|^3+|\vartheta|^3) d\x d\tau\\
\di\qquad\leq C \int_0^t \|\psi_\x\|_{L^2(\bbr)} \|(\phi,\vartheta)\|_{L^2(\bbr)} \|(\phi,\vartheta)\|_{L_\infty}^2 d\tau\\
\di \qquad \leq C\int_0^t \|\psi_\x\|_{L^2(\bbr)}  \|(\phi_\x,\vartheta_\x)\|_{L^2(\bbr)} \|(\phi,\vartheta)\|^2_{L^2(\bbr)} d \tau\\
\di\qquad\leq C\varepsilon_1^2\int_0^t (\mb D+\|\phi_\x\|^2_{L^2(\bbr)})d\tau,
\end{array}
$$
and
\beq\label{barp-d}
\begin{array}{ll}
\di \bar p_t-\s\bar p_\x &\di =\frac{R(\bar \theta_t-\s\bar \theta_\x)}{\bar v}-\frac{R\bar \theta(\bar v_t-\s\bar v_\x)}{\bar v^2}\\[3mm]
&\di =\frac{R\big[(\theta^R_t-\s\theta^R_\x)+(\theta^C_t-\s\theta^C_\x)-(\dot{\mb X}(t)+\s)(\theta^S)^{-\mb X}_\x\big]}{\bar v}\\[3mm]
&\di \quad-\frac{R\bar\theta\big[(v^R_t-\s v^R_\x)+(v^C_t-\s v^C_\x)-(\dot{\mb X}(t)+\s)(v^S)^{-\mb X}_\x\big]}{\bar v^2},
\end{array}
\eeq
and $v_t+\s v_\x = u_\x =\psi_\x +\bar u_\x$, we have
$$
\begin{array}{ll}
\di 
|J_{53}| \leq C\deltas \int_0^t|\dot{\mb X}(\tau)|^2d\tau+C(\delta_0+\eps_1)\int_0^t \big(\mathcal{G}^R+\mathcal{G}^S+\mb D+\|\phi_\x\|^2_{L^2(\bbr)} \big)d\tau\\[3mm]
 \di \qquad\quad \  +C\dc\int_0^t \int_\bbr W^2|(\phi, \vartheta)|^2 d\xi d\tau+C(\deltar^{\frac13}+\delta_C)
\end{array}
$$
Similar estimates hold for $J_{54}$ and $J_6$.\\
We easily have
$$
\begin{array}{ll}
\di 
|J_{55}|\leq \deltas \int_0^t|\dot{\mb X}(\tau)|^2d\tau+C(\delta_0+\eps_1)\int_0^t \mathcal{G}^Sd\tau.
\end{array}
$$
Using Young's inequality with any small $\nu>0$,
$$
\begin{array}{ll}
\di 
|J_{7}|\leq C\int_0^t \int_\bbr \big(|\vartheta_\x|+|\bar\theta_\x||\phi|\big)\big(|\bar p_\x||\phi|+|(\phi_\x, \vartheta_\x)|+W|(\phi,\vartheta)|\big)d\x d\tau\\[3mm]
\di\qquad \leq (\nu +C\delta_0)\int_0^t \int_\bbr W^2|(\phi,\vartheta)|^2d\x d\tau\\[3mm]
\di\qquad \quad + C_\nu\int_0^t \|(\phi_\x,\vartheta_\x)\|^2_{L^2(\bbr)} d\tau+C\delta_0\int_0^t \big(\mathcal{G}^S+\mathcal{G}^R\big) d\tau.
\end{array}
$$

\noindent$\bullet$ {\bf Proof of  \eqref{b5e2}:} 
From the momentum equations $\eqref{NS-1}_2$ and $\eqref{bar-system}_2$ and the fact $p-\bar p=\frac{R\vartheta-\bar p \phi}{v}$, we have
$$
\frac{(R\vartheta-\bar p \phi)_\x}{v}-v_\x\frac{R\vartheta-\bar p \phi}{v^2}=-(\psi_t-\s \psi_\xi)+\dot{\mb X}(t)(u^S)^{\mb X}_\xi+\mu\left(\frac{u_\x}{v}-\frac{\bar u_\x}{\bar v}\right)_\xi-Q_1.
$$
Multiplying the above equation by $H(R\vartheta-\bar p \phi)$ with $H$ defined in \eqref{hH} implies that
\beq\label{vce4}
\begin{array}{ll}
\di 
\Big[\frac{H(R\vartheta-\bar p \phi)^2}{2v}\Big]_\x-H_\x\frac{(R\vartheta-\bar p \phi)^2}{2v}+v_\x\frac{H(R\vartheta-\bar p \phi)^2}{2v^2}\\
\di =\Big[-(\psi_t-\s \psi_\xi)+\dot{\mb X}(t)(u^S)^{\mb X}_\xi+\mu\left(\frac{u_\x}{v}-\frac{\bar u_\x}{\bar v}\right)_\xi-Q_1\Big]H(R\vartheta-\bar p \phi).
\end{array}
\eeq
By $H_\x=W^2$ and the facts that
$$
\begin{array}{ll}
\di 
(\psi_t-\s \psi_\xi)H(R\vartheta-\bar p \phi)=\big[H\psi(R\vartheta-\bar p \phi)\big]_t-\sigma \big[H\psi(R\vartheta-\bar p \phi)\big]_\x\\[3mm]
\qquad \qquad \qquad \di -\psi(R\vartheta-\bar p \phi)(H_t-\s H_\x)-H\psi \big[(R\vartheta-\bar p \phi)_t-\s (R\vartheta-\bar p \phi)_\x\big],
\end{array}
$$
and 
$$
\begin{array}{ll}
\di 
(R\vartheta-\bar p \phi)_t-\s (R\vartheta-\bar p \phi)_\x\\[2mm]
\di =(\gamma-1)\Big[\dot{\mb X}(t)R(\theta^S)^{\mb X}_\xi-(pu_\xi-\bar p\bar u_\x)+\kappa \big( \frac{\theta _{ \x}
}{v}-\frac{\bar \theta_{ \x}}{\bar v} \big)_\x
 +  \mu \big( \frac{(u_{\x})^2}{v} -\frac{ (\bar u_\x)^2}{\bar v}
\big)-Q_2\Big]\\[3mm]
\di \quad -\bar p\big(\psi_\x+\dot{\mb X}(t)(v^S)^{\mb X}_\xi\big)-(\bar p_t-\s \bar p_\x)\phi,
\end{array}
$$
it follows from \eqref{vce4} that
\beq\label{vce5}
\begin{array}{ll}
\di 
W^2\frac{(R\vartheta-\bar p \phi)^2}{2v}=(\cdots)_\x+ v_\x\frac{H(R\vartheta-\bar p \phi)^2}{2v^2}+\big[H\psi(R\vartheta-\bar p \phi)\big]_t\\[3mm]
\di\quad  -\psi(R\vartheta-\bar p \phi)(H_t-\s H_\x)+\bar p H\psi\big(\psi_\x+\dot{\mb X}(t)(v^S)^{\mb X}_\xi\big)+H\psi\phi(\bar p_t-\s \bar p_\x)\\[3mm]
\di\quad  -(\gamma-1)H\psi\Big[\dot{\mb X}(t)R(\theta^S)^{\mb X}_\xi-(pu_\xi-\bar p\bar u_\x)+\kappa \big( \frac{\theta _{ \x}
}{v}-\frac{\bar \theta_{ \x}}{\bar v} \big)_\x
 +  \mu \big( \frac{(u_{\x})^2}{v} -\frac{ (\bar u_\x)^2}{\bar v}
\big)\\
\di\quad -Q_2\Big]+H(R\vartheta-\bar p \phi)\Big[-\dot{\mb X}(t)(u^S)^{\mb X}_\xi +\mu\left(\frac{u_\x}{v}-\frac{\bar u_\x}{\bar v}\right)_\x+Q_1\Big].
\end{array}
\eeq
Noting that
$$
\bar p H\psi\psi_\x=(\cdots)_\x -(\bar p H)_\x\frac{\psi^2}{2}=(\cdots)_\x -W^2\frac{\bar p\psi^2}{2} -\bar p _\x H\frac{\psi^2}{2},
$$
and 
$$
\begin{array}{ll}
\di (\gamma-1)H\psi (pu_\xi-\bar p\bar u_\x) \di =(\gamma-1)H\psi\big[\bar p\psi_\x+(p-\bar p)\psi_\x+\bar u_\x(p-\bar p)\big]\\[3mm]
\di\qquad =(\cdots)_\x-(\bar p H)_\xi
\frac{(\gamma-1)\psi^2}{2}+(\gamma-1)H\psi\big[(p-\bar p)\psi_\x+\bar u_\x(p-\bar p)\big]\\[3mm]
\di\qquad =(\cdots)_\x-W^2\frac{(\gamma-1)\bar p\psi^2}{2}-\bar p_\x H
\frac{(\gamma-1)\psi^2}{2}+(\gamma-1)H\psi\big[(p-\bar p)\psi_\x+\bar u_\x(p-\bar p)\big],
\end{array}
$$
and then integrating the resulting equation over $\bbr\times [0,t]$ give that
\beq\label{vce6}
\begin{array}{ll}
\di 
\int_0^t\int_\bbr W^2\Big[\frac{(R\vartheta-\bar p \phi)^2}{2v}+\frac{(\gamma+1)\bar p\psi^2}{2}\Big]  d\xi d\tau=\int_\bbr H\psi(R\vartheta-\bar p \phi)|^{\tau=t}_{\tau=0} d\xi\\[3mm]
\di ~ +\int_0^t\int_\bbr  v_\x\frac{H(R\vartheta-\bar p \phi)^2}{2v^2}d\xi d\tau  -\int_0^t\int_\bbr \psi(R\vartheta-\bar p \phi)(H_t-\s H_\x) d\xi d\tau\\[4mm] 
\di ~ + \int_0^t\dot{\mb X}(\tau)\int_\bbr H\big[\bar p \psi(v^S)^{\mb X}_\xi-R(\gamma-1)\psi (\theta^S)^{\mb X}_\xi-(R\vartheta-\bar p \phi)(u^S)^{\mb X}_\xi\big] d\xi d\tau\\[3mm]
\di ~
-\int_0^t\int_\bbr \bar p_\x H\frac{(\gamma+1)\psi^2}{2} d\xi d\tau+ \int_0^t\int_\bbr (\gamma-1)H\psi\big[(p-\bar p)\psi_\x+\bar u_\x(p-\bar p)\big] d\xi d\tau\\[3mm]
\di~
+ \int_0^t\int_\bbr H\psi\phi(\bar p_t-\s \bar p_\x) d\xi d\tau  + \int_0^t\int_\bbr \mu (\gamma-1)H\psi \big( \frac{(u_{\x})^2}{v} -\frac{ (\bar u_\x)^2}{\bar v}
\big)   d\xi d\tau \\[3mm] 
\di ~ + \int_0^t\int_\bbr \big(H(R\vartheta-\bar p \phi)\big)_\xi \mu\big(\frac{u_\x}{v}-\frac{\bar u_\x}{\bar v}\big) d\xi d\tau + \int_0^t\int_\bbr    \kappa(\gamma-1)(H\psi)_\xi \big( \frac{\theta _{ \x}
}{v}-\frac{\bar \theta_{ \x}}{\bar v} \big) d\xi d\tau  \\[3mm]
\di ~+ \int_0^t\int_\bbr \big[(\gamma-1)H Q_2 +H(R\vartheta-\bar p \phi) Q_1\big] d\xi d\tau.
\end{array}
\eeq
The estimations for the right hand side of \eqref{vce6} is almost same as \eqref{v-c-4} and even we have some better estimates since $|H(t,\x)|\leq C (1+t)^{-\frac 12}$ and $
|H_t-\s H_\x|\leq C(1+t)^{-\frac 32}$, and we omit the details.

\end{appendix}

\noindent{\bf Conflict of Interest:} The authors declared that they have no conflicts of interest to this work.

\bibliography{KVW}

\end{document}